\documentclass[12pt]{article}
\usepackage{fullpage,amsthm,amssymb,amsmath}
\usepackage{enumerate,enumitem}
\usepackage{hyperref}
\newtheorem{thm}{Theorem}[section]
\newtheorem{lemma}[thm]{Lemma}
\newtheorem{cor}[thm]{Corollary}
\newtheorem{rmk}[thm]{Remark}
\newtheorem{defn}[thm]{Definition}
\newtheorem{proposition}[thm]{Proposition}
\newtheorem{mainthm}{Theorem}

\newcommand{\N}{\mathbb{N}}
\newcommand{\R}{\mathbb{R}}
\newcommand{\Z}{\mathbb{Z}}

\newcommand{\mcl}{\mathcal L}
\newcommand{\mcn}{\mathcal Q}
\newcommand{\bbp}{\mathbb P}
\newcommand{\C}{\mathbb{C}}

\newcommand{\al}{\alpha}
\newcommand{\be}{\beta}
\newcommand{\ga}{\gamma}

\newcommand{\del}{\delta}

\newcommand{\ep}{\epsilon}
\newcommand{\eps}{\epsilon}
\newcommand{\sig}{\sigma}
\newcommand{\Sig}{\Sigma}
\newcommand{\ka}{\kappa}
\newcommand{\lam}{\lambda}
\newcommand{\Lam}{\Lambda}
\newcommand{\Om}{\Omega}
\newcommand{\om}{\omega}

\newcommand{\tld}[1]{\tilde{#1}}
\newcommand{\mc}[1]{\mathcal{#1}}

\newcommand{\lni}{\lim_{n\to\infty}}
\DeclareMathOperator{\essinf}{ess\ inf}
\DeclareMathOperator{\esssup}{ess\ sup}
\DeclareMathOperator{\var}{var}

\newcommand{\lex}[2]{\lam_{#2}({#1})}
\newcommand{\ba}{B} 

\newcommand{\B}{\mathcal{B}} 
\newcommand{\BV}{\B} 
\newcommand{\RR}{\mathcal R} 
\newcommand{\X}{X} 
\newcommand{\paeom}{\ensuremath{\bbp\text{-a.e. } \om \in \Om}}
\newcommand{\vot}{v_\omega^\theta} 
\newcommand{\lot}{\lambda_\omega^\theta} 
\newcommand{\Lsot}{\mcl_{\sigma^{-1}\omega}^\theta} 

\title{A spectral approach for quenched limit theorems for random expanding dynamical systems}
\author{D.\ Dragi\v cevi\' c \footnote{School of Mathematics and Statistics,
University of New South Wales,
Sydney NSW 2052, Australia. E-mail: \texttt{d.dragicevic@unsw.edu.au}. Department of Mathematics, University of Rijeka, Croatia. E-mail:\texttt{ddragicevic@math.uniri.hr}.},
G.\ Froyland\footnote{School of Mathematics and Statistics,
University of New South Wales,
Sydney NSW 2052, Australia. E-mail: \texttt{G.Froyland@unsw.edu.au }.},
C.\   Gonz\'alez-Tokman\footnote{School of Mathematics and Physics,
The University of Queensland,
St Lucia QLD 4072,
Australia. E-mail: \texttt{cecilia.gt@uq.edu.au}.},
S.\ Vaienti\footnote{
Sandro Vaienti, Aix Marseille Universit\'e, CNRS, CPT, UMR 7332, 13288 Marseille, France and Universit\'e de Toulon, CNRS, CPT, UMR 7332, 83957 La Garde, France. E-mail: \texttt{vaienti@cpt.univ-mrs.fr}.}}

\begin{document}
\maketitle
\begin{abstract}
We prove quenched versions of (i) a large deviations principle (LDP), (ii) a central limit theorem (CLT), and (iii) a local central limit theorem (LCLT) for non-autonomous dynamical systems.
A key advance is the extension of the spectral method, commonly used in limit laws for deterministic maps, to the general random setting.
We achieve this via multiplicative ergodic theory and the development of a general framework to control the regularity of Lyapunov exponents of \emph{twisted transfer operator cocycles} with respect to a twist parameter.
While some versions of the LDP and CLT have previously been proved with other techniques, the local central limit theorem is, to our knowledge, a completely new result, and one that demonstrates the strength of our method.
Applications include non-autonomous (piecewise) expanding maps, defined by random compositions of the form $T_{\sigma^{n-1}\omega}\circ\cdots\circ T_{\sigma\omega}\circ T_\omega$.
An important aspect of our results is that we only assume ergodicity and invertibility of the random driving $\sigma:\Omega\to\Omega$;  in particular no expansivity or mixing properties are required.

\end{abstract}
\tableofcontents

\section{Introduction}
The Nagaev-Guivarc'h spectral method for proving the central limit theorem (due to Nagaev \cite{N57,N61} for Markov chains and Guivarc'h \cite{RE83,GH88} for deterministic dynamics) is a powerful approach with applications to several other limit theorems, in particular large deviations and the local limit theorem.
In the deterministic setting a map $T:X\to X$ on a state space $X$ preserves a probability measure $\mu$ on $X$.
An observable $g:X\to \mathbb{R}$ generates $\mu$-stationary process $\{g(T^n x)\}_{n\ge 0}$ and one studies the statistics of this process.
Central to the spectral method is the transfer operator $\mathcal{L}:\mathcal{B}\circlearrowleft$, acting on a Banach space $\mathcal{B}\subset L^1(\mu)$ of complex-valued functions with regularity properties compatible with the regularity of $T$\ \footnote{The transfer operator satisfies $\int_X f\cdot g\circ T\ d\mu=\int_X \mathcal{L}f\cdot g\ d\mu$ for $f\in L^1(\mu), g\in L^\infty(\mu)$.}.
A \textit{twist} is introduced to form the twisted transfer operator $\mathcal{L}^\theta f:=\mathcal{L}(e^{\theta g}f)$.
The three key steps to the spectral approach are:
\begin{enumerate}
\item[S1.]
Representing the characteristic function of Birkhoff (partial) sums $S_ng=\sum_{i=0}^{n-1} g\circ T^i$ as integrals of $n^{th}$ powers of twisted transfer operators.
\item[S2.]
Quasi-compactness (existence of a spectral gap) for the twisted transfer operators $\mathcal{L}^\theta$ for $\theta$ near zero.
\item[S3.]
Regularity (e.g.\ twice differentiable for the CLT) of the leading eigenvalue of the twisted transfer operators $\mathcal{L}^\theta$ with respect to the twist parameter $\theta$, for $\theta$ near zero.
\end{enumerate}
This spectral approach has been widely used to prove limit theorems for deterministic dynamics, including large deviation principles \cite{HennionHerve,rey-young08}, central limit theorems \cite{RE83,Broise,HennionHerve,AyyerLiveraniStenlund}, Berry-Esseen theorems \cite{GH88,gouezel05}, local central limit theorems \cite{RE83,HennionHerve,gouezel05}, and vector-valued almost-sure invariance principles \cite{melbourne_nicol09,gouezel10}.
We refer the reader to the excellent review paper \cite{G15}, which provides a broader overview of how to apply the spectral method to problems of these types, and the references therein.

In this paper, we extend this spectral approach to the situation where we have a family of maps $\{T_\omega\}_{\omega\in\Omega}$, parameterised by elements of a probability space $(\Omega,\mathbb{P})$.
These maps are composed according to orbits of a driving system $\sigma:\Omega\to\Omega$.
The resulting dynamics takes the form of a map cocycle $T_{\sigma^{n-1}\omega}\circ\cdots\circ T_{\sigma\omega}\circ T_\omega$.
In terms of real-world applications, we imagine that $\Omega$ is the class of underlying configurations that govern the dynamics on the (physical or state) space $X$.
As time evolves, $\sigma$ updates the current configuration and the dynamics $T_\omega$ on $X$ correspondingly changes.
To retain the greatest generality for applications, we make minimal assumptions on the configuration updating (the driving dynamics) $\sigma$, and only assume $\sig$ is $\mathbb{P}$-preserving, ergodic and invertible;  in particular, no mixing hypotheses are imposed on $\sigma$.

We will assume certain uniform-in-$\omega$ (eventual) expansivity conditions for the maps $T_\omega$.
Our observable $g:\Omega\times X\to \mathbb{R}$ can (and, in general, will)  depend on the base configuration $\omega$ and will satisfy a fibrewise finite variation condition.
One can represent the random dynamics by a deterministic skew product  transformation $\tau(\omega, x)=( \sigma (\omega), T_{\omega}(x)), \quad \om \in \Om, \ x\in X$.
It is well known that whenever $\sigma$ is invertible and $\tilde{\mu}$ is a $\tau$-invariant probability measure with marginal $\mathbb{P}$ on the base $\Omega$, the disintegration of $\tilde\mu$ with respect to $\mathbb{P}$ produces conditional measures $\mu_{\om}$ which are {\em equivariant};  namely $\mu_{\om}\circ T_{\omega}^{-1}=\mu_{\sigma \om}$.
Our limit theorems will be established $\mu_{\om}$-almost surely and for $\mathbb{P}$-almost all choices of $\om$;  we therefore develop \emph{quenched} limit theorems.
In the much simpler case where $\sigma$ is Bernoulli, which yields an i.i.d.\ composition of the elements of $\{T_\omega\}_{\omega\in\Omega}$, one is often interested in the study of limit laws with respect to a measure $\hat{\mu}$ which is invariant with respect to the \emph{averaged} transfer operator, and reflects the outcomes of {\em averaged observations} \cite{O83,Arnoldbook}.
The corresponding limit laws with respect to $\hat{\mu}$ are typically called \emph{annealed} limit laws; see \cite{ANV15} and references therein for recent results in this framework.

As is common in the quenched setting, we impose a \textit{fiberwise centering} condition for the observable.
Thus, limit theorems in this context deal with fluctuations about a time-dependent mean.
For example, if the observable is temperature, the limit theorems would characterise temperature fluctuations about the mean, but this mean is allowed to vary with the seasons.
The recent work \cite{AbdelkaderAimino} provides a discussion of annealed and quenched limit theorems, and in particular an example regarding the necessity of fibrewise centering the observable for the quenched case. Without such a condition, quenched limit theorems have been established exclusively in special cases where all maps preserve a common invariant measure \cite{AyyerLiveraniStenlund,NandoriSzaszVarju} (and where the centering is obviously identical on each fibre).

In the quenched random setting we generalise the above three key steps of the spectral approach:
\begin{enumerate}
\item[R1.]
Representing the ($\omega$-dependent) characteristic function of Birkhoff (partial) sums $S_ng(\omega,\cdot)$ defined by~\eqref{birkhoff} as an integral of $n^{th}$ \emph{random compositions} of twisted transfer operators.
\item[R2.]
Quasi-compactness for the twisted transfer operator \emph{cocycle};  equivalently, existence of a gap in the Lyapunov spectrum of the cocycle $\mathcal{L}^{\theta,(n)}_\omega:=\mathcal{L}^\theta_{\sigma^{n-1}\omega}\circ\cdots\circ \mathcal{L}^\theta_{\sigma\omega}\circ \mathcal{L}^\theta_{\omega}$ for $\theta$ near zero.
\item[R3.]
Regularity (e.g.\ twice differentiable for the CLT) of the \emph{leading Lyapunov exponent} and \emph{Oseledets spaces} of the twisted transfer operators cocycle with respect to the twist parameter $\theta$, for $\theta$ near zero.
\end{enumerate}
At this point we note that the key steps S1--S3 in the deterministic spectral approach mean that one satisfies the requirement for a \emph{naive version} of the Nagaev-Guivarc'h method \cite{G15};  namely $\mathbb{E}(e^{i\theta S_n})=c(\theta)\lambda(\theta)^n+d_n(\theta)$ for $c$ continuous at 0 and $|d_n|_\infty\to 0.$
In this case, $\lambda(\theta)$ is the leading eigenvalue of $\mathcal{L}^\theta$.
Similarly, the key steps R1--R3 yield an analogue naive version of a random Nagaev-Guivarc'h method, where
for all complex $\theta$ in a neighborhood of 0, and \paeom, we have that
\[
 \lim_{n\to \infty} \frac 1 n \log | \mathbb{E}_{\mu_\om}(e^{\theta S_ng(\omega,\cdot)}) |=\Lambda (\theta),
\]
where $\Lam(\theta)$ is the top Lyapunov exponent of the random cocycle generated by $\mcl_\om^\theta$ (see Lemma~\ref{need}).
This condition is of course weaker than the asymptotic equivalence of  \cite{G15},
but together with the exponential decay of the norm of the projections to the complement of the top Oseledets space (see Section~\ref{sec:pfclt}), which handles the error corresponding to quantity $d_n$ above, we are able to achieve the desired limit theorems. Under this analogy, we could consider our result as a new {\em naive} version of the Nagaev-Guivarc'h method, framed and adapted to random dynamical systems.

The quasi-compactness of the twisted transfer operator cocycle (item 2 above) will be based on the works \cite{FLQ2,GTQuas1}, which have adapted multiplicative ergodic theory to the setting of cocycles of possibly non-injective operators;  the non-injectivity is crucial for the study of endomorphisms $T_\omega$.
These new multiplicative ergodic theorems, and in particular the quasi-compactness results, utilise random Lasota-Yorke inequalities in the spirit of Buzzi \cite{Buzzi}.
For the regularity of the leading Lyapunov exponent (item 3 above) we develop \emph{ab initio} a cocycle-based perturbation theory, based on techniques of \cite{HennionHerve}.
This is necessary because in the random setting objects such as eigenvalues and eigenfunctions of individual transfer operators have no dynamical meaning and therefore one cannot simply apply standard perturbation results such as \cite{kato}, as is done in \cite{HennionHerve} and all other spectral approaches for limit theorems.
Multiplicative ergodic theorems do not provide, in general, a spectral decomposition with eigenvalues and eigenvectors as in the classical sense, but only a hierarchy of equivariant Oseledets spaces containing vectors which grow at a fixed asymptotic exponential rate, determined by the corresponding Lyapunov exponent.

Let us now summarise the main results of the present paper,  obtained with our new cocycle-based perturbation theory. These are limit theorems for \emph{random Birkhoff sums} $S_n g$, associated to an {observable} $g \colon \Omega \times X \to \mathbb R$, and defined by
\begin{equation}\label{birkhoff}
 S_n g(\om, x):=\sum_{i=0}^{n-1} g(\tau^i (\om, x))=\sum_{i=0}^{n-1} g(\sigma^i \omega, T_\omega^{(i)} x), \quad (\om, x)\in \Om \times X, \ n\in \mathbb N,
\end{equation}
where $T_\omega^{(i)} = T_{\sig^{i-1}\om} \circ \dots \circ T_{\sig\om} \circ T_\om$.
The observable will be required to satisfy some regularity properties, which are made precise in Section~\ref{sec:obs}.
Moreover,  we will suppose that $g$ is \textit{fiberwise  centered} with respect to the invariant measure $\mu$ for $\tau$. That is,
\begin{equation}\label{zeromean0}
 \int g(\om, x) \, d\mu_\om(x)=0 \quad \text{for } \paeom.
\end{equation}
The necessary conditions on the dynamics are summarised in an \textit{admissibility} notion, which is introduced in Definition~\ref{def:admis}.
Our first results are quenched forms of the Large Deviations Theorem 
 and the Central Limit Theorem. 
 We remark that, while our results are all stated in terms of the fiber measures $\mu_\om$,
 in our examples, the same results hold true when $\mu_\om$ is replaced by Lebesgue measure $m$.
This is a consequence of a the result of Eagleson \cite{Eagleson} combined with the fact that, in our examples, $\mu_\om$ is equivalent to $m$.

\begin{mainthm}[Quenched large deviations theorem]\label{thm:ldt}
Assume the transfer operator cocycle $\mc{R}$ is admissible, and the observable $g$ satisfies conditions~\eqref{zeromean0} and \eqref{obs}.
Then, there exists $\epsilon_0>0$ and a non-random function $c\colon (-\epsilon_0, \epsilon_0) \to \mathbb R$ which is nonnegative, continuous, strictly convex, vanishing only at $0$ and such that
\[
\lim_{n\to \infty} \frac 1 n \log \mu_\om(S_n g(\om, \cdot ) >n\epsilon)=-c(\epsilon), \quad \text{for  $0<\epsilon <\epsilon_0$ and \paeom}.
\]
\end{mainthm}
\begin{mainthm}[Quenched central limit theorem] \label{thm:clt}
Assume the transfer operator cocycle $\mc{R}$ is admissible, and the observable $g$ satisfies conditions~\eqref{zeromean0} and \eqref{obs}.
Assume also that the non-random \textit{variance} $\Sig^2$, defined in \eqref{variance} satisfies $\Sig^2>0$.
Then, for every bounded and continuous function $\phi \colon \R \to \R$ and \paeom, we have
\[
\lim_{n\to\infty}\int \phi \bigg{(}\frac{S_n g(\om, x)}{ \sqrt n}\bigg{)}\, d\mu_\om (x)=\int \phi \, d\mathcal N(0, \Sig^2).
\]
(The discussion after \eqref{variance} deals with the degenerate case $\Sig^2=0$).
\end{mainthm}

Similar LDT and CLT results were previously obtained in different contexts, and using other methods, by Kifer \cite{Kifer-LD90,Kifer92,K98} and Bakhtin \cite{B1,B2}.
In \cite{Kifer-LD90}, Kifer shows a large deviations result for occupational measures, relying on existence of a pressure functional and uniqueness of equilibrium states for some dense sets of functions.
For the CLT, Kifer used martingale techniques. To control the rate of mixing,  conditions such as $\phi$-mixing and $\alpha$-mixing are assumed in \cite{K98}. His examples include random subshifts of finite type and random smooth expanding maps.
Bakhtin obtains a central limit theorem and some estimates on large deviations for sequences of smooth hyperbolic maps with common expanding/contracting distributions, under  a mixing assumption and a variance growth condition on the Birkhoff sums \cite{B1, B2}.
Finally, we note that in our recent article \cite{DFGTV} we provide the first complete proof of the Almost Sure Invariance Principle for random transformations of the type covered in this paper using martingale techniques.


In this work, we prove for the first time a Local Central Limit Theorem for random transformations. Theorem~\ref{thm:lclt} presents the aperiodic version: This result relies on an assumption concerning fast decay in $n$ of the norm of the twisted operator cocycle $\|\mathcal{L}^{it,(n)}_\omega\|_{\mathcal{B}}$, for $t\in\mathbb{R}\setminus\{0\}$ and \paeom.
This hypothesis is made precise in  \ref{cond:aper}.
Such an assumption is usually stated in the deterministic case (resp.\ in the random annealed situation), by asking that the twisted operator (resp.\ the averaged random twisted operator) $\mathcal{L}^{it}$ has spectral radius strictly less than one for $t\in\mathbb{R}\setminus\{0\}$; this is called the {\em aperiodicity} condition.
\begin{mainthm}[Quenched local central limit theorem] \label{thm:lclt}
Assume the transfer operator cocycle $\mc{R}$ is admissible, and the observable $g$ satisfies conditions~\eqref{zeromean0} and \eqref{obs}. In addition, suppose the aperiodicity condition \ref{cond:aper} is satisfied. Then, for \paeom\  and every bounded interval $J\subset \R$, we have
 \[
  \lim_{n\to \infty}\sup_{s\in \R} \bigg{\lvert} \Sig \sqrt{n} \mu_\om (s+S_n g(\om, \cdot)\in J)-\frac{1}{\sqrt{2\pi}}e^{-\frac{s^2}{2n\Sig^2}}\lvert J\rvert \bigg{\rvert}=0.
 \]
\end{mainthm}
In the autonomous case, aperiodicity is equivalent to a co-boundary condition, which can be checked in particular examples \cite{Morita}.
We are also able to state an equivalence between the decay of $\mathcal{L}^{it,(n)}_\omega$ and a (random) co-boundary equation (Lemma~\ref{lem:AperiodicCoboundary}), which opens the possibility to verify the hypotheses of the local limit theorem in specific examples (see Section~\ref{sec:exlclt}).
In addition, we establish a \emph{periodic version} of the LCLT in Theorem~\ref{thm:lcltp}.


In summary, a main contribution of the present work is the development of the spectral method for establishing limit theorems for quenched (or $\omega$ fibre-wise) random dynamics.
Our hypotheses are natural from a dynamical point of view, and we explicitly verify them in the framework of the random Lasota-Yorke maps, and more generally for random piecewise expanding maps in higher dimensions.
The new spectral approach for the quenched random setting we present here has been specifically designed for generalisation and we are hopeful that this method will afford the same broad flexibility that continues to be exploited by work in the deterministic setting.
While at present we have uniform-in-$\omega$ assumptions on \emph{time-asymptotic} expansion and decay properties of the random dynamics, we hope that in the future these assumptions can be relaxed to enable even larger classes of dynamical systems to be treated with our new spectral technique.
For example, limit theorems for dynamical systems beyond the uniformly hyperbolic setting continues to be an active area of research, e.g.~ \cite{gouezel05,gouezel10,G15,BahsounBose,DeSimoiLiverani-LimThms,nicol_torok_vaienti_2016,LeppanenStenlund},
and another interesting set of related results on limit theorems occur in the setting of homogenisation \cite{GottwaldMelbourne,MelbourneKelly16,MelbourneKelly17}.
Our extension to the quenched random case opens up a wide variety of potential applications and  future work will explore generalisation to random dynamical systems with even more complicated forms of behaviour.

\section{Preliminaries}
We begin this section by recalling several useful facts from multiplicative ergodic theory. We then introduce assumptions on the state space $X$;  $X$ will be a probability space equipped with a notion of variation for integrable functions.
This abstract approach will enable us to simultaneously treat the
cases where (i) $X$ is a unit interval (in the context of Lasota-Yorke maps) and (ii) $X$ is a subset of $\R^n$ (in the context of piecewise expanding maps in higher dimensions).
We introduce several dynamical assumptions for the cocycle $\mcl_\om$, $\om \in \Om$ of transfer operators under which our limit theorems apply.
This section is concluded by constructing large families of examples
of both Lasota-Yorke maps and piecewise expanding maps in $\R^n$ that satisfy
all of our conditions.
\subsection{Multiplicative ergodic theorem}\label{sectionMET}
In this subsection we recall the recently established versions of the multiplicative ergodic theorem which can be applied to the study of cocycles of transfer operators and  will play
an important role in the present paper. We begin by recalling some basic notions.

A tuple $\mc{R}=(\Om, \mc{F}, \bbp, \sigma, \B, \mcl)$ will be called a linear cocycle, or simply a\ \textit{cocycle}, if
 $\sigma$ is an invertible ergodic measure-preserving transformation on a probability space $(\Omega,\mathcal F,\mathbb P)$,  $(\B, \|\cdot \|)$ is a Banach space and
 $\mathcal L\colon \Omega\to L(\B)$ is a family of bounded linear operators such that
$\log^+\|\mathcal L(\omega)\|\in L^1(\mathbb P)$. 
Sometimes we will also use $\mc{L}$ to refer to the full cocycle $\mc{R}$.
In order to obtain sufficient measurability conditions in our setting of interest, we assume the following:
\begin{enumerate}[label=(C\arabic*), series=conditions]
\setcounter{enumi}{-1}
\item \label{cond:METCond}
$\sigma$ is a homeomorphism, $\Om$ is a Borel subset of a separable, complete metric space and $\mcl$ is $\bbp-$continuous (that is, $\mcl$ is  continuous on each of countably many Borel sets whose union is $\Omega$).
\end{enumerate}
For each $\om \in \Om$ and $n\geq 0$,  let $ \mcl_\om^{(n)}$ be the linear operator given by  \[ \mcl_\om^{(n)}:= \mcl_{\sig^{n-1}\om}\circ \dots \circ\mcl_{\sig\om} \circ \mcl_\om. \]
Condition~\ref{cond:METCond} implies that the maps $\om \mapsto \log \| \mcl_\om^{(n)}\|$ are measurable. Thus,
Kingman's sub-additive ergodic theorem ensures that the following limits exist and coincide for \paeom:
\begin{align*}
\Lam(\mc{R}) &:= \lim_{n\to \infty} \frac1n\log \| \mcl_\om^{(n)}\|\\
\ka(\mc{R}) &:=  \lim_{n\to \infty} \frac1n\log ic( \mcl_\om^{(n)}),
\end{align*}
where \[\text{ic}(A):=\inf\Big\{r>0 : \ A(B_\B) \text{ can be covered with finitely many balls of radius }r \Big\},\] and $B_\B$ is the unit ball of $\B$.
 The cocycle $\mc{R}$ is called \textit{quasi-compact} if
$\Lam(\mc{R})> \ka(\mc{R})$.
 The quantity $\Lam(\mc{R})$ is called the \textit{top Lyapunov exponent} of the cocycle and generalises the notion of (logarithm of) spectral radius of a linear operator. Furthermore,  $\ka(\mc{R})$ generalises the notion of essential spectral radius to the context of cocycles.
 Let
$(\B', |\cdot|)$ be a Banach space such that $\B\subset \B'$ and that the inclusion $(\B, \|\cdot\|) \hookrightarrow (\B',|\cdot|)$ is compact.
The following result, based on a theorem of Hennion \cite{Hennion}, is useful to establish quasi-compactness.
\begin{lemma} (\cite[Lemma C.5]{GTQuas1}) \label{lem:Hennion}
Let $(\Omega,\mathcal F,\mathbb P)$ be a probability space,  $\sigma$ an ergodic, invertible, $\bbp$-preserving transformation on $\Om$ and
 $\mc{R}=(\Om, \mc{F}, \bbp, \sigma, \B, \mcl)$  a cocycle. Assume $\mcl_\om$ can be extended continuously to $(\B', |\cdot|)$ for \paeom, and that
 there exist measurable functions $\al_\om, \be_\om, \ga_\om: \Om \to \R$ such that
the following strong and weak Lasota-Yorke type inequalities hold for every $f\in \B$,
\begin{align}
\| \mcl_\om f\| &\leq \al_\om \|f\| + \be_\om |f| \label{eq:RandomLY}
\quad \text{ and }\quad \\
\| \mcl_\om \| & \leq \ga_\om.
\end{align}
In addition, assume
\[
\int \log \al_\om \, d\bbp(\om) < \Lam(\mc{R}), \text{ and } \int \log \ga_\om \, d\bbp(\om)<\infty.
\]
Then, $\ka(\mc{R})\leq \int \log \al_\om\, d\bbp(\om)$.
In particular, $\mc{R}$ is quasi-compact.
\end{lemma}

Another result which will be useful in the sequel is the following comparison between Lyapunov exponents with respect to different norms.  In what follows, we denote by $\lex{\om, f}{\B}$ the Lyapunov exponent of $f$ with respect to the norm $\|\cdot\|_{\B}$. That is,
$\lex{\om, f}{\B}=\lim_{n\to \infty} \frac1n \log  \| \mcl_\om^{(n)}f\|_{\B}$, where $f\in \B$ and $(\B, \| \cdot\|_{\B})$ is a Banach space.
%
%

\begin{lemma}[Lyapunov exponents for different norms]\label{lem:RandomSameExp}
Under the notation and hypotheses of Lemma~\ref{lem:Hennion},
let
$r:=\int_\Om \log \al_\om\, d\bbp(\om)$ and assume that for some $f\in \B$, $\lex{\om, f}{\B'}> r$. Then, $\lex{\om, f}{\B}=\lex{\om, f}{\B'}$.
\end{lemma}

\begin{proof}
 The inequality
$\lex{\om, f}{\B}\geq \lex{\om, f}{\B'}$ is trivial, because $\|\cdot\|$ is stronger than
$|\cdot|$ (i.e. because the embedding $(\B, \|\cdot\|) \hookrightarrow (\B',|\cdot|)$ is compact).
In the other direction, the result essentially follows from Lemma~C.5(2) in~\cite{GTQuas1}. Indeed,  this lemma establishes that if $r<0$ and $\lex{\om, f}{\B'}\leq 0$ then $\lex{\om, f}{\B}\leq 0$. The choice of $0$ is irrelevant, because if the cocycle
is  rescaled by a constant $C>0$, all Lyapunov exponents and $r$ are shifted by $\log C$. Thus, we conclude that if $\lex{\om, f}{\B'}> r$ then, $\lex{\om, f}{\B}\leq \lex{\om, f}{\B'}$, as claimed.
%
\end{proof}
 A spectral-type decomposition for quasi-compact cocycles can be obtained via  a \textit{multiplicative ergodic theorem},
as follows.

\begin{thm}[Multiplicative ergodic theorem, MET \cite{FLQ2}] \label{thm:MET}
Let $\mathcal R=(\Omega,\mathcal F,\mathbb
P,\sigma,\B,\mathcal L)$ be a  quasi-compact cocycle and
suppose that condition~\ref{cond:METCond} holds.
Then, there exists $1\le l\le \infty$ and a sequence of exceptional
Lyapunov exponents
\[ \Lam(\RR)=\lambda_1>\lambda_2>\ldots>\lambda_l>\kappa(\RR) \quad \text{(if $1\le l<\infty$)}\]
or  \[ \Lam(\RR)=\lambda_1>\lambda_2>\ldots \quad \text{and} \quad \lim_{n\to\infty} \lambda_n=\kappa(\RR) \quad \text{(if $l=\infty$);} \]
 and for $\mathbb P$-almost every $\omega \in \Om$ there exists a unique splitting (called the \textit{Oseledets splitting}) of $\B$ into closed subspaces
\begin{equation}\label{eq:splitting}
\B=V(\omega)\oplus\bigoplus_{j=1}^l Y_j(\omega),
\end{equation}
depending measurably on $\om$ and  such that:
\begin{enumerate} [label=(\Roman*)]
\item  For each $1\leq j \leq l$, $Y_j(\omega)$ is finite-dimensional ($m_j:=\dim Y_j(\omega)<\infty$),  $Y_j$ is equivariant i.e. $\mcl_\om Y_j(\omega)= Y_j(\sigma\omega)$ and for every $y\in Y_j(\omega)\setminus\{0\}$,
\[\lim_{n\to\infty}\frac 1n\log\|\mathcal L_\omega^{(n)}y\|=\lambda_j.\]
(Throughout this work, we will  also refer to $Y_1(\om)$ as simply $Y(\om)$ or $Y_\om$.)
\item
$V$ is equivariant i.e. $\mcl_\om V(\omega)\subseteq V(\sigma\omega)$ and
for every $v\in V(\omega)$, \[\lim_{n\to\infty}\frac 1n\log\|\mathcal
L_\omega^{(n)}v\|\le \kappa(\RR).\]
\end{enumerate}
\end{thm}
The \textit{adjoint cocycle} associated to $\mc{R}$ is the cocycle  $\mc{R}^*:=(\Om, \mc{F}, \bbp, \sigma^{-1}, \B^*, \mcl^*)$, where $(\mcl^*)_\om := (\mcl_{\sig^{-1}\om})^*$.
In a slight abuse of notation which should not cause confusion, we will often write $\mcl^*_\om$ instead of $(\mcl^*)_\om$, so $\mcl^*_\om$ will denote the operator adjoint to $\mcl_{\sig^{-1}\om}$.
\begin{rmk}\label{rmk:METCond*}
It is straightforward to check that if \ref{cond:METCond} holds for $\mc{R}$, it also holds for $\mc{R}^*$. Furthermore,
$\Lam(\mc{R}^*)=\Lam(\mc{R})$ and $\ka(\mc{R}^*)=\ka(\mc{R})$. The last statement follows from the equality, up to a multiplicative factor (2),  $ic(A)$ and $ic(A^*)$ for every $A\in L(\B)$ \cite[Theorem 2.5.1]{AkhmerovEAl}.
\end{rmk}
The following result gives an answer to a natural question on whether one can relate the Lyapunov exponents and Oseledets splitting of the adjoint cocycle $\mc{R}^*$ with the Lyapunov
exponents and Oseledets decomposition of the original cocycle $\mc{R}$.
\begin{cor}\label{cor:METAdjoint}
Under the assumptions of Theorem~\ref{thm:MET}, the adjoint cocycle $\mc{R}^*$ has a unique, measurable, equivariant Oseledets splitting
\begin{equation}\label{eq:splitting_adj}
\B^*=V^*(\omega)\oplus\bigoplus_{j=1}^l Y^*_j(\omega),
\end{equation}
with the same exceptional Lyapunov exponents $\lam_j$ and multiplicities $m_j$ as $\mc{R}$.
\end{cor}
The proof of this result involves some technical properties about volume growth in Banach spaces, and is therefore deferred to Appendix~\ref{sec:usc}.

Next, we establish a relation between Oseledets splittings of $\mc{R}$ and $\mc{R}^*$, which will be used in the sequel.
Let the simplified Oseledets decomposition for the cocycle $\mcl$ (resp. $\mcl^*$) be
\begin{equation}\label{BVstar}
\mathcal B=Y(\om)\oplus H(\om) \quad
(\text{resp. } \mathcal{B}^*=Y^*(\om) \oplus H^*(\om) ),
 \end{equation}
where $Y(\om)$ (resp. $Y^*(\om)$) is the top Oseledets
subspace for $\mcl$ (resp. $\mcl^*$) and $H(\om)$ (resp. $H^*(\om)$) is a direct sum of all other Oseledets subspaces.

For a subspace $S\subset \mathcal B$, we set $ S^\circ=\{\phi \in \mathcal{B}^*: \phi(f)=0 \quad \text{for every $f\in S$}\}$
and similarly for a subspace $S^* \subset \mathcal{B}^*$ we define
$
 (S^*)^\circ =\{f\in \mathcal{B}: \phi(f)=0 \quad \text{for every $\phi \in S^*$}\}.
$
\begin{lemma}[Relation between Oseledets splittings of $\mc{R}$ and $\mc{R}^*$]\label{lem:AnnihilatorOsSplittings}
The following relations hold for \paeom:
 \begin{equation}\label{jj}
  H^*(\om)=Y(\om)^\circ \quad \text{and} \quad H(\om)=Y^*(\om)^\circ.
 \end{equation}

\end{lemma}

\begin{proof}
  We first  claim that
 \begin{equation}\label{0456}
  \limsup_{n\to \infty}\frac{1}{n} \log \lVert \mcl_\om^{*, (n)}\rvert_{Y(\om)^\circ}\rVert <\lambda_1, \quad \text{for $\paeom$.}
 \end{equation}
Let $\Pi_{\om}$ denote the projection onto $H(\om)$ along $Y(\om)$ and take an arbitrary $\phi \in Y(\om)^\circ$.
We have
\[
 \begin{split}
  \lVert \mcl_\om^{*, (n)} \phi \rVert_{\mathcal{B}^*} &=\sup_{\lVert f\rVert_{\mathcal B}\le 1} \lvert (\mcl_\om^{*, (n)} \phi)(f)\rvert =\sup_{\lVert f\rVert_{\mathcal B}\le 1} \lvert \phi(\mcl_{\sigma^{-n} \om}^{(n)}(f))\rvert \\
  &=\sup_{\lVert f\rVert_{\mathcal B}\le 1} \lvert \phi(\mcl_{\sigma^{-n} \om}^{(n)}(\Pi_{\sigma^{-n} \om}f))\rvert
  \le \lVert \phi\rVert_{\mathcal{B}^*} \cdot \lVert \mcl_{\sigma^{-n} \om}^{(n)}\Pi_{\sigma^{-n} \om}\rVert,
 \end{split}
\]
and thus
\[
 \lVert \mcl_\om^{*, (n)}\rvert_{Y(\om)^\circ}\rVert \le \lVert \mcl_{\sigma^{-n} \om}^{(n)}\Pi_{\sigma^{-n} \om}\rVert.
\]
Hence, in order to prove~\eqref{0456} it is sufficient to show that
\begin{equation}\label{mm4}
 \limsup_{n\to \infty}\frac{1}{n} \log \lVert \mcl_{\sigma^{-n} \om}^{(n)}\Pi_{\sigma^{-n} \om}\rVert <\lambda_1, \quad \text{for $\paeom$.}
\end{equation}
However, it follows from results in~\cite{DF} and \cite[Lemma 8.2]{FLQ1} that
\[
\lim_{n\to \infty} \frac 1 n \log \lVert \mcl_{\sigma^{-n} \om}^{(n)} \rvert_{H(\sigma^{-n} \om)}\rVert =\lambda_2
\]
and
\[
\lim_{n\to \infty} \frac 1 n \log \lVert \Pi_{\sigma^{-n} \om}\rVert=0,
\]
which readily  imply~\eqref{mm4}. We now claim that
\begin{equation}\label{BVstar2}
\mathcal{B}^*=Y(\om)^*\oplus Y(\om)^\circ, \quad \text{for $\paeom$.}
\end{equation}
We first note that the sum on the right hand side of~\eqref{BVstar2} is direct. Indeed, each nonzero vector in $Y(\om)^*$ grows at the rate $\lambda_1$, while by~\eqref{0456} all nonzero vectors in $Y(\om)^\circ$ grow at the rate $<\lambda_1$. Furthermore, since
 the codimension of $Y(\om)^\circ$ is the same as dimension of $Y(\om)^*$, we have that~\eqref{BVstar2} holds.

Finally, by comparing decompositions~\eqref{BVstar} and~\eqref{BVstar2}, we conclude that the first equality in~\eqref{jj} holds. Indeed, each $\phi \in H^*(\om)$ can be written as $\phi=\phi_1+\phi_2$, where
$\phi_1 \in Y(\om)^*$ and $\phi_2 \in Y(\om)^\circ$. Since $\phi$ and $\phi_2$ grow at the rate $<\lambda_1$ and $\phi_1$ grows at the rate $\lambda_1$, we obtain  that $\phi_1=0$ and thus $\phi=\phi_2 \in Y(\om)^\circ$. Hence,  $H^*(\om) \subset Y(\om)^\circ$ and similarly $Y(\om)^\circ \subset H^*(\om)$. The second assertion of the lemma  can be obtained similarly.
\end{proof}

\subsection{Notions of variation}\label{sec:var}
Let $(X, \mathcal G)$ be a measurable space endowed with a probability measure $m$ and a notion of a variation $\var \colon L^1(X, m) \to [0, \infty]$ which satisfies
the following conditions:
\begin{enumerate}
 \item[(V1)] $\var (th)=\lvert t\rvert \var (h)$;
 \item[(V2)] $\var (g+h)\le \var (g)+\var (h)$;
 \item[(V3)] $\lVert h\rVert_{L^\infty} \le C_{\var}(\lVert h\rVert_1+\var (h))$ for some constant $1\le C_{\var}<\infty$;
 \item[(V4)] for any $C>0$, the set  $\{h\colon X \to \mathbb R: \lVert h\rVert_1+\var (h) \le C\}$ is $L^1(m)$-compact;
 \item[(V5)] $\var(1_X) <\infty$, where $1_X$ denotes the function equal to $1$ on $X$;
 \item[(V6)] $\{h \colon X \to \mathbb R_+: \lVert h\rVert_1=1 \ \text{and} \ \var (h)<\infty\}$ is $L^1(m)$-dense in
 $\{h\colon X \to \mathbb R_+: \lVert h\rVert_1=1\}$.
 \item[(V7)] for any $f\in L^1(X, m)$ such that $\essinf f>0$, we have $\var(1/f) \le \frac{\var (f)}{(\essinf f)^2}$.
 \item[(V8)] $\var (fg)\le \lVert f\rVert_{L^\infty}\cdot \var(g)+\lVert g\rVert_{L^\infty}\cdot \var(f)$.
 \item[(V9)] for $M>0$, $f\colon X \to [-M, M]$ measurable and  every $C^1$ function $h\colon [-M, M] \to \C$, we have
 $\var (h\circ f)\le \lVert h'\rVert_{L^\infty} \cdot \var(f)$.
\end{enumerate} We define
\[
 \B:=BV=BV(X,m)=\{g\in L^1(X, m): \var (g)<\infty \}.
\]
Then, $\B$ is a Banach space with respect to the norm
\[
 \lVert g\rVert_{\B} =\lVert g\rVert_1+ \var (g).
 \]
From now on, we will use $\B$ to denote a Banach space of this type, and $ \lVert g\rVert_{\B} $, or simply $\|g\|$ will denote the corresponding norm.

Well-known examples of this notion correspond to the case where $X$ is a  subset of $\R^n$.
In the one-dimensional case we use the classical notion of variation given by
\begin{equation}\label{var1d}
 \var (g)=\inf_{h=g (mod \ m)} \sup_{0=s_0<s_1<\ldots <s_n=1}\sum_{k=1}^n \lvert h(s_k)-h(s_{k-1})\rvert
\end{equation}
for which it is well known that properties (V1)-(V9) hold.
On the other hand, in the multidimensional case, we let $m=Leb$ and define
\begin{equation}\label{varmd}
 \var (f)=\sup_{0<\epsilon \le \epsilon_0}\frac{1}{\epsilon^\alpha}\int_{\R^d}osc (f,  B_\epsilon (x)))\, dx,
\end{equation}
where
\[
 osc (f, B_\epsilon (x))=\esssup_{x_1, x_2 \in B_\epsilon (x)}\lvert f(x_1)-f(x_2)\rvert
\]
and where $\esssup$ is taken with respect to product measure $m\times m$. For this notion properties (V1)-(V9) have been verified by Saussol~\cite{Saussol} except for (V7) which is proved in~\cite{DFGTV}
and~(V9) which we prove now.
\begin{lemma}
 The notion of $\var$ defined by~\eqref{varmd} satisfies (V9).
\end{lemma}

\begin{proof}
 Take $M>0$, $f$ and $h$ as in the statement of~(V9). For arbitrary $x\in X$, $\epsilon>0$ and $x_1, x_2\in B_\epsilon (x)$, it follows from the mean value theorem that
 \[
  \lvert (h\circ f)(x_1)-(h\circ f)(x_2)\rvert \le \lVert h'\rVert_{L^\infty} \cdot \lvert f(x_1)-f(x_2)\rvert,
 \]
which immediately implies that
\[
 osc (h\circ f, B_\epsilon (x)) \le \lVert h'\rVert_{L^\infty} \cdot osc (f, B_\epsilon (x)),
\]
and we obtain the conclusion of the lemma.
\end{proof}

\subsection{Admissible cocycles of transfer operators}\label{COTO}

Let  $(\Omega, \mathcal{F}, \mathbb P, \sigma)$ be  as Section~\ref{sectionMET}, and $X$ and $\BV$ as in Section~\ref{sec:var}.
Let  $T_{\omega} \colon X \to X$, $\omega \in \Omega$ be a collection of non-singular  transformations (i.e.\ $m\circ T_\omega^{-1}\ll m$ for each $\omega$) acting   on $X$.
The associated skew product transformation  $\tau \colon  \Omega \times X \to  \Omega \times X$ is defined by
\begin{equation}
\label{eq:tau}
\tau(\omega, x)=( \sigma (\omega), T_{\omega}(x)), \quad \om \in \Om, \ x\in X.
\end{equation}
 Each transformation $T_{\omega}$ induces the corresponding transfer operator $\mathcal L_{\omega}$ acting on $L^1(X, m)$ and  defined  by the following duality relation
\[
\int_X(\mathcal L_{\omega} \phi)\psi \, dm=\int_X\phi(\psi \circ T_{\omega})\, dm, \quad \phi \in L^1(X, m), \ \psi \in L^\infty(X, m).
\]
For each $n\in \mathbb N$ and $\omega \in \Omega$, set
\[
T_{\omega}^{(n)}=T_{\sigma^{n-1} \omega} \circ \cdots \circ T_{\omega} \quad \text{and} \quad \mathcal L_{\omega}^{(n)}=\mathcal L_{\sigma^{n-1} \omega} \circ \cdots \circ \mathcal L_{\omega}.\]

\begin{defn}[Admissible cocycle] \label{def:admis}
We call the transfer operator cocycle $\mc{R}=(\Om, \mathcal F, \bbp, \sig, \B, \mathcal L)$ admissible if, in addition to \ref{cond:METCond}, the following conditions hold.
\begin{enumerate}[label=(C\arabic*), series=conditions]
\item \label{cond:unifNormBd}
there exists $K>0$ such that
\begin{equation*}
 \lVert \mathcal L_\om f\rVert_{\BV} \le K\lVert f\rVert_{\BV}, \quad \text{for every $f\in \BV$ and $\paeom$.}
\end{equation*}

\item \label{C0}
there exists $N\in \mathbb N$ and measurable $\alpha^N, \beta^N \colon \Om \to (0, \infty)$, with
$ \int_\Om \log \alpha^N (\om)\,  d\mathbb P(\om)<0$,
 such that for every $f\in \BV$ and $\paeom$,
\begin{equation*}
 \lVert \mcl_\om^{(N)} f\rVert_{\BV} \le \alpha^N(\om)\lVert f\rVert_{\BV}+\beta^N(\om)\lVert f\rVert_1.
\end{equation*}
\item  \label{cond:dec} there exist $K', \lambda >0$ such that
for every $n\ge 0$, $f\in \BV$ such that $\int f\, dm=0$ and $\paeom$.
\begin{equation*}
\lVert \mathcal L_{\om}^{(n)} (f)\rVert_{\BV} \le K'e^{-\lambda n}\lVert f\rVert_{\BV}.
\end{equation*}

\item  \label{Min} there exist $N\in \N, c>0$ such that for each $a>0$  and any sufficiently large $n\in \mathbb N$,
\begin{equation*}
\essinf  \mathcal L_\omega^{(Nn)} f\ge c \lVert f\rVert_1, \quad \text{for every $f\in C_a$ and \paeom,}
\end{equation*}
where $C_a:=\{ f \in \BV: f\ge 0 \text{ and } \var(f)\le a\int f\, dm \}.$
\end{enumerate}
\end{defn}
Admissible cocycles of transfer operators can be investigated via Theorem~\ref{thm:MET}. Indeed, the following holds.


\begin{lemma}\label{lem:qc+1dim}
 An admissible cocycle of transfer operators $\mathcal R=(\Om, \mathcal F, \mathbb P, \sig, \BV, \mcl)$  is quasi-compact.
 Furthermore, the top Oseledets space is one-dimensional. That is, $\dim Y(\om)=1$ for \paeom.
\end{lemma}
\begin{proof}
 The first statement follows readily from Lemma~\ref{lem:Hennion}, \ref{C0} and a simple observation that for a cocycle $\mc{R}$ of transfer operators we have that $\Lambda (\mc{R})=0$.
 The fact that $\dim Y(\om)=1$ follows from \ref{cond:dec}.
 \end{proof}

The following result shows that, in this context, the top Oseledets space is indeed the unique random acim.
That is, there exists a unique measurable function $v^0: \Om \times \X \to \R^+$ such that
for \paeom, $v^0_\om:= v^0(\om, \cdot) \in \B$, $\int v^0_\om(x)dm=1$ and
\begin{equation}\label{v0om}
 \mathcal L_\om v_\om^0=v_{\sigma \om}^0, \quad \text{for $\paeom$.}
\end{equation}

\begin{lemma}[Existence and uniqueness of a random acim]\label{lem:PF}
Let $\mathcal R=(\Omega,\mathcal F,\mathbb
P,\sigma,\B,\mathcal L)$ be an admissible cocycle of transfer operators,
satisfying the assumptions of Theorem~\ref{thm:MET}.
Then, there exists a unique random absolutely continuous invariant measure
for $\mathcal R$.

\end{lemma}
\begin{proof}

Theorem~\ref{thm:MET} shows that the map $\om \mapsto Y_\om$ is measurable, where
$Y_\om$ is regarded as an element of the Grassmannian $\mc{G}$ of $\B$.
Furthermore,  \cite[Lemma 10]{FLQ2} and an argument analogous to \cite[Lemma
10]{GTQuas1}
yields  existence of a measurable selection of bases for $Y_\om$.
Lemma~\ref{lem:qc+1dim} ensures that $\dim Y(\om)=1$.
Hence, there exists  a measurable map
$\om \mapsto h_\om$, with $h_\om \in \B$ such that $h_\om$ spans $Y_\om$ for $\paeom$.

Notice that Lebesgue measure $m$, when regarded as an element of $\B^*$, is a
conformal measure for $\mc{R}$. That is, $m$ spans $Y^*_\om$ for \paeom. In fact, it
is straightforward to verify $\mcl^*_\om m = m$, because the $\mcl_\om$ preserve
integrals.

Thus, the simplified Oseledets decomposition~\eqref{BVstar}
in combination with the duality relations of Lemma~\ref{lem:AnnihilatorOsSplittings} imply that
$m(h_\om)\neq 0$ for \paeom.
In particular we can consider the (still measurable) function $\om \mapsto v^0_\om:=
\frac{h_\om}{\int h_\om dm}$.

The equivariance property of Theorem~\ref{thm:MET} ensures that $\mcl_\om v^0_\om \in
Y_{\sig\om}$ and the fact that $\mcl_\om$ preserves integrals, combined with the
normalized choice of $v^0_\om$ and the assumption that $\dim Y_{\sig\om}=1$, implies
that $\mcl_\om v^0_\om =v^0_{\sig\om}$.

The fact that $v^0_\om \geq 0$ for \paeom \ follows from the positivity and linearity properties of
$\mcl_\om$, which ensure that the positive and negative parts, $v^+_\om$ and
$v^-_\om$,  are equivariant.
Recall that $v^+_\om$, $v^-_\om$, have non-overlapping supports. Thus, if $v_\om^+ \neq 0 \neq v_\om^-$ for a set of positive measure of $\om \in \Om$,
the spaces
$Y^+_\om, Y^-_\om$ spanned by $v^+_\om$, $v^-_\om$, respectively, are subsets of $Y(\om)$,
contradicting  the fact that $\dim Y(\om)=1$. Then, since the normalization condition implies $v_\om^+\neq 0$, we have $v^-_\om=0$ for \paeom.
The fact that the random acim is unique is also a direct consequence of the fact that $\dim Y(\om)=1$.
\end{proof}
%

For an admissible transfer operator cocycle $\mc{R}$, we let $\mu$ be the invariant probability measure given by
\begin{equation}\label{eq:defmu}
 \mu(A \times B)=\int_{A\times B} v^0(\om, x)\, d (\mathbb P \times m)(\om, x), \quad \text{for $A\in \mathcal F$ and $B\in \mathcal G$,}
\end{equation}
where $v^0$ is the unique random acim for $\mc{R}$ and $\mc{G}$ is the Borel $\sigma$-algebra of $X$.
We note that $\mu$ is $\tau$-invariant, because 
of \eqref{v0om}. Furthermore, for each $G\in L^1(\Omega \times X, \mu)$ we have that
\[
 \int_{\Omega \times X} G\, d\mu=\int_{\Omega} \int_X G(\om, x)\, d\mu_\om(x)\, d\mathbb P(\om),
\]
where $\mu_\om$ is a measure on $X$ given by $d\mu_\om=v^0(\om, \cdot)dm$.
We now list several  important consequences of conditions~\ref{C0}, \ref{cond:dec} and~\ref{Min} established in~\cite[\S2]{DFGTV}.
\begin{lemma}\label{lem:boundedv}
The unique random acim $v^0$ of an  admissible cocycle of transfer operators satiesfies the following:
\begin{enumerate}
\item  \label{it:boundedv}
\begin{equation}\label{eq:boundedv}
  \esssup_{\om \in \Om} \lVert v_\om^0\rVert_{\BV} <\infty;
 \end{equation}
\item  \begin{equation}\label{lowerbound} \essinf v_\om^0 (\cdot)\ge c>0, \quad \text{for  $\paeom$;}
\end{equation}
\item  there exists $K>0$ and $\rho \in (0, 1)$  such that
 \begin{equation}\label{buzzi}
 \bigg{\lvert} \int_X \mathcal L_\omega^{(n)}(f v_\omega^0)h\, dm -\int_X f \, d\mu_\omega \cdot \int_X h \, d\mu_{\sigma^n \omega} \bigg{\rvert} \le K\rho^n
 \lVert h\rVert_{L^\infty} \cdot \lVert f \rVert_{\BV} ,
\end{equation}
for $n\ge 0$, $h \in L^\infty (X, m)$, $f \in \BV$ and $\paeom$.
\end{enumerate}
\end{lemma}
We emphasize that~\eqref{buzzi} is a special case of a more general decay of correlations result proved by Buzzi~\cite{Buzzi}, but in this case with the stronger conclusion that the decay rates and coefficients $K$ are uniform over $\om\in \Om$.

 \subsubsection{Examples }\label{RLYM}
 In order to be able to be in the setting of admissible transfer operators cocycles, we need to ensure that
\ref{cond:METCond} holds. To fulfill this requirement (see~\cite[Section 4.1]{FLQ2} for a detailed discussion)  in the rest of the paper we will assume
\begin{enumerate}[label=(C\arabic*'), start=0]
\item \label{cond:c0'}
 $\sigma$ is a homeomorphism, $\Om$ is a Borel subset of a separable, complete metric space,
  the map $\omega \to T_\omega$ has a countable range $T_1, T_2, \dots$ and for each $j$, $\{\om \in \Om: T_\om=T_j\}$ is  measurable.
\end{enumerate}
Although this condition is somewhat restrictive,
 we emphasize that the assumptions on the structure of $\Omega$ are very mild and that the only requirements for $\sigma$ are that it has to be an ergodic, measure-preserving homeomorphism. In particular, no mixing conditions are required.
 Furthermore, the $T_\omega$ need only be chosen from a countable family.

Following \cite[\S2]{DFGTV}, we present two classes of examples,  one- and higher-dimensional piecewise smooth expanding maps, which yield admissible transfer operator cocycles.
 \paragraph{Random Lasota-Yorke maps.}
\label{sec:examples1}
Let $X=[0, 1]$, a Borel $\sigma$-algebra $\mathcal G$ on $[0, 1]$ and the Lebesgue measure $m$ on $[0, 1]$. Consider the notion of variation defined  in \eqref{var1d}.
For a piecewise $C^2$ map $T:[0,1]\to [0,1]$, set $\delta (T)=\essinf_{x\in [0,1]} \lvert T'\rvert$ and let $b(T)$ denote the number of intervals of monoticity (branches) of $T$.
Consider now a measurable map $\omega\mapsto T_\omega$, $\omega \in \Omega$  of piecewise $C^2$ maps on $[0, 1]$ such that
\begin{equation}\label{sd_consts}
 b:=\esssup_{\omega \in \Omega} b(T_\omega)<\infty, \  \delta:=\essinf_{\omega \in \Omega} \delta (T_\omega)>1, \ \mbox{ and } D:=\esssup_{\omega \in \Omega}\lVert T''_\omega \rVert_{L^\infty}<\infty.
\end{equation}

For each $\omega \in \Omega$, let $b_\omega=b(T_\om)$, so that there are essentially disjoint sub-intervals $J_{\omega,1}, \dots, J_{\omega, b_\omega}\subset I$, with $\cup_{k=1}^{b_{\omega}} J_{\omega, k}=I$, so that $T_\omega|_{J_{\omega,k}}$ is $C^2$ for each $1\leq k \leq b_\omega$. The minimal such partition $\mathcal{P}_\omega:=\{ J_{\omega,1}, \dots, J_{\omega,b_\omega} \}$ is called the \textit{regularity partition} for $T_\omega$.
It is well known that whenever $\delta>2$, and $\essinf_{\omega \in \Omega} \min_{1\leq k \leq b_\omega} m(J_{\omega,k})>0$, there exist  $\alpha \in (0, 1)$ and $K>0$  such that
\begin{equation*}
 \var (\mathcal L_\omega f)\le \alpha \var(f)+K \lVert f\rVert_1, \quad \text{for $ f\in BV$ and \paeom.}
\end{equation*}

More generally, when $\delta<2$, one can take an iterate $N \in \N$ so that $\delta^N>2$. If the regularity partitions $\mathcal{P}^N_\omega:=\{ J^N_{1,\omega}, \dots, J^N_{\omega,b_\omega^{(N)}} \}$ corresponding to the maps $T_\omega^{(N)}$ also satisfy
$\essinf_{\omega \in \Omega} \min_{1\leq k \leq b^{(N)}_\omega} m(J^N_{\omega,k})>0$, then
 there exist $\alpha^N \in (0, 1)$ and $K^N>0$  such that
\begin{equation}\label{RLY}
 \var (\mathcal L_\omega^N f)\le \alpha^N \var(f)+K^N \lVert f\rVert_1, \quad \text{for $ f\in BV$ and \paeom.}
\end{equation}
We  assume that \eqref{RLY} holds for some $N\in \N$.

Finally, we suppose the following uniform covering condition holds:
\begin{equation}\label{eq:UnifCov}
\text{For every subinterval } J\subset I, \exists k= k(J) \text{ s.t. for a.e. }  \omega \in \Omega, T_\omega^{(k)}(J) = I.
\end{equation}

The results of \cite[\S2]{DFGTV} ensure that random Lasota-Yorke maps which satisfy the conditions of this section plus \ref{cond:c0'} are admissible. (While~\ref{C0} is not explicitely required by \cite{DFGTV}, it is established in the process of showing the remaining conditions.)

\paragraph{Random piecewise expanding maps in higher dimensions.}
\label{sec:examples2}
We now discuss the case of piecewise expanding maps in higher dimensions. Let $X$ be a compact subset of $\mathbb{R}^N$
which is the closure of its non-empty interior. Let $X$ be  equipped with a Borel $\sigma$-algebra $\mathcal G$  and Lebesgue measure $m$. We consider the notion of variation defined in~\eqref{varmd} for suitable $\alpha$ and $\epsilon_0$.
We say that the map  $T:X\to X$ is \emph{piecewise expanding} if there exist
finite families   $\mathcal{A}=\{A_i\}_{i=1}^m$ and $\mathcal{\tilde A}=\{\tilde{A_i}\}_{i=1}^m$ of  open sets in $\mathbb R^N$, a family of maps
$T_i: \tilde{A_i}\to \mathbb{R}^N$, $i=1, \ldots, m$ and $\epsilon_1(T)>0$
 such that:
\begin{enumerate}
\item $\mathcal A$ is a disjoint family of sets, $m(X\setminus \bigcup_{i}A_i)=0$ and $\tilde A_i \supset \overline A_i$ for each $i=1, \ldots, m$;
\item there exists $0<\gamma(T_i)\le 1$ such that each $T_i$ is of class $C^{1+\gamma(T_i)}$;
\item
For every $1\leq i \leq m$, $T|_{A_i}=T_i|_{A_i}$ and
$T_i(\tilde{A_i})\supset B_{\eps_1(T)}(T(A_i))$, where
  $B_{\eps}(V)$ denotes a neighborhood of size $\eps$ of the set $V.$ We say that
   $T_i$ is  the local extension of $T$ to the $\tilde{A_i}$;

\item there exists a constant $C_1(T)>0$ so that for each $i$ and $x,y\in T(A_i)$ with
$\mbox{dist}(x,y)\leq\eps_1(T)$,
$$
|\det DT_i^{-1}(x)-\det DT_i^{-1}(y)|\leq C_1(T)|\det DT_i^{-1}(x)|\mbox{dist}(x,y)^{\gamma(T)};
$$

\item there exists $s(T)<1$ such that for every $x,y\in T(\tilde{A}_i) \textrm{ with } \mbox{dist}(x,y)\leq\eps_1(T)$, we have
$$\mbox{dist}(T_i^{-1}x,T_i^{-1}y)\leq s(T)\, \mbox{dist}(x,y);$$

\item each $\partial A_i$ is a codimension-one embedded compact
  piecewise $C^1$ submanifold and
  \[\label{sc}s(T)^{\gamma(T)}+\frac{4s(T)}{1-s(T)}Z(T)\frac{\Gamma_{N-1}}{\Gamma_N}<1, \]
  where $Z(T)=\sup\limits_{x}\sum\limits_i \#\{\textrm{smooth pieces
    intersecting } \partial A_i \textrm{ containing }x\}$ and $\Gamma_N$ is
  the volume of the unit ball in $\mathbb{R}^N$.
\end{enumerate}
Consider now a measurable map $\omega\mapsto T_\omega$, $\omega \in \Omega$  of piecewise expanding maps on $X$ such that
\[
 \epsilon_1:=\inf_{\omega \in \Omega} \epsilon_1(T_\omega)>0, \  \gamma:=\inf_{\omega \in \Omega} \gamma (T_\omega)>0,\  C_1:=\sup_{\om \in \Om} C_1(T_\omega)<\infty, \  s:=\sup_{\om \in \Om} s(T_\om)<1
\]
and
\[
\sup_{\om \in \Om} \Big( s(T_\om)^{\gamma(T_\om)}+\frac{4s(T_\om)}{1-s(T_\om)}Z(T_\om)\frac{\Gamma_{N-1}}{\Gamma_N} \Big)<1.
\]
Then, \cite[Lemma 4.1]{Saussol} implies that there exist $\nu \in (0, 1)$ and $K>0$ independent on $\omega$ such that
\begin{equation}\label{gfv}
 \var (\mathcal L_\omega f) \le \nu \var (f)+K \lVert f\rVert_1 \quad \text{for each $f\in \mathcal B$ and $\omega \in \Omega$,}
\end{equation}
where $\var$ is given by~\eqref{varmd} with $\alpha=\gamma$ and some $\epsilon_0>0$ sufficiently small (which is again independent on $\om$).  We note that~\eqref{gfv} readily implies that conditions~\ref{cond:unifNormBd} and~\ref{C0} hold.
Finally, we note that under additional assumption that
$$\mbox{for any open set $J \subset X$, there exists $k=k(J)$ such that for a.e. $\omega\in \Omega$, $T^k_{\omega}(J)=X,$}$$
the results in~\cite[\S2]{DFGTV} show that~\ref{cond:dec} and~\ref{Min} also hold.

\paragraph{Remark.}
We point out that while conditions \ref{cond:unifNormBd}, \ref{cond:dec} and \ref{Min} are stated in a uniform way, sometimes it is possible to recover them from non-uniform assumptions. For example, assuming that $\{T_\om\}_{\om \in \Om}$ takes only finitely many values, one can recover a uniform version of \ref{cond:dec} from a non-uniform one, for example by compactness arguments (see the proof of Lemma~\ref{lem:AperiodicCoboundary} for a similar argument). Also, our results apply to cases where conditions \ref{cond:unifNormBd}--\ref{Min}, or the hypotheses which imply them (e.g. \eqref{sd_consts}), are only satisfied eventually; that is, for some iterate $T_\om^{(N)}$, where $N$ is independent of $\om \in \Om$.

\section{Twisted transfer operator cocycles}
We begin by introducing the class of observables to which our limit theorems apply.
For a  fixed observable and each parameter $\theta \in \mathbb C$, we introduce the twisted cocycle $\mathcal {L}^\theta=\{\mcl_\om^\theta\}_{\om \in \Om}$.
We show that the cocycle $\mcl^\theta$ is quasicompact for $\theta$ close to $0$.
Most of this section is devoted to the study of   regularity properties of the map $\theta \mapsto \Lambda (\theta)$ on  a neighborhood of $0\in \mathbb C$, where $\Lambda (\theta)$ denotes the top Lyapunov exponent of the cocycle $\mcl^\theta$.
In particular, we show that this map is of class $C^2$ and that its restriction to a neighborhood of $0\in \R$ is  strictly  convex. This is achieved by combining  ideas from the perturbation theory of linear operators with our multiplicative ergodic theory machinery. As a byproduct of our approach, we explicitly construct the top Oseledets subspace of cocycle
$\mcl^\theta$ for $\theta$ close to $0$.

\subsection{The observable}\label{sec:obs}
\begin{defn}[Observable]\label{def:obs}
Let an \emph{observable} be a
 measurable map $g \colon \Omega \times X \to \mathbb R$ satisfying the following properties:
 \begin{itemize}
 \item Regularity:
\begin{equation}\label{obs}
 \lVert g(\om, x)\rVert_{L^\infty (\Omega \times X)}=: M<\infty \quad \text{and} \quad  \esssup_{\om \in \Om} \var (g_\om) <\infty,
\end{equation}
where $g_\omega=g (\omega, \cdot )$, $\omega \in \Omega$.
\item Fiberwise centering:
\begin{equation}\label{zeromean}
 \int g(\om, x) \, d\mu_\om(x)= \int g(\om, x)v^0_\om (x) \, dm(x)=0 \quad \text{for } \paeom,
\end{equation}
where $v^0$ is the density of the unique random acim, satisfying \eqref{v0om}.
\end{itemize}
\end{defn}
The main results of this paper will deal with establishing limit theorems for Birkhoff sums associated to $g$, $S_ng$, defined in \eqref{birkhoff}.

\subsection{Basic properties of twisted transfer operator cocycles}
Throughout this section, $\mc{R}=(\Om, \mc{F}, \bbp, \sig, \B, \mcl)$ will denote an admissible transfer operator cocycle.
For  $\theta \in \mathbb C$,
the \emph{twisted transfer operator cocycle}, or \emph{twisted cocycle}, $\mc{R}^\theta$  is defined as
 $\mc{R}^\theta=(\Om, \mc{F},  \bbp, \sig, \BV, \mcl^\theta)$,
where   for each $\omega \in \Omega$, we define
\begin{equation}\label{eq:twisted}
\mathcal L_\omega^{\theta}(f)=\mathcal L_\omega (e^{\theta g(\omega, \cdot )}f), \quad f\in \BV.
\end{equation}
  For convenience of notation, we will also use $ \mcl^\theta$ to denote the cocycle $\mc{R}^\theta$.
For each $\theta \in \C$, set $\Lam(\theta):=\Lam(\mc{R}^\theta)$, $\ka(\theta):=\ka(\mc{R}^\theta)$ and
\[
\mathcal L_\om^{\theta, \, (n)}=\mathcal{L}^{\theta}_{\sigma^{n-1}\omega}\circ\cdots\circ \mathcal{L}^{\theta}_\omega, \quad \text{for $\om \in \Om$ and $n\in \N$.}
\]
The next lemma provides basic information about the dependence of $\mcl_\omega^{\theta}$ on $\theta$.

\begin{lemma}[Basic regularity of $\theta \mapsto \mcl_\om^\theta$]\label{1214} \
\begin{enumerate}
\item
Assume~\ref{cond:unifNormBd} holds.
 Then, there exists a continuous function $K\colon \mathbb C \to (0, \infty)$ such that
 \begin{equation}\label{se2}
 \lVert \mathcal L_\om^\theta h\rVert_{\BV} \le K(\theta)\lVert h\rVert_{\BV}, \quad \text{for $h\in \BV$, $\theta \in \mathbb C$ and $\paeom$.}
\end{equation}

\item
For $\om\in \Om, \theta \in \C$, let  $M_\om^\theta$ be the linear operator on $\BV$ given by $M_\om^\theta(h(\cdot)):= e^{\theta g(\om, \cdot)} h(\cdot)$. Then, $\theta \mapsto M_\om^\theta$ is continuous in the norm topology of $\BV$. Consequently, $\theta \mapsto \mcl_\om^\theta$ is also continuous in the norm topology of $\BV$.
\end{enumerate}
\end{lemma}

\begin{proof}
Note that it follows from~\eqref{obs} that $\lvert e^{\theta g(\omega, \cdot)}h\rvert_1 \le e^{\lvert \theta \rvert M} \lvert h\rvert_1$.
Furthermore,  by~(V8) we have
\[
 \var(e^{\theta g(\omega, \cdot)}h) \le \lVert e^{\theta g(\omega, \cdot)}\rVert_{L^\infty}\cdot  \var(h)+\var(e^{\theta g(\omega, \cdot)}) \cdot \lVert h\rVert_{L^\infty}.
\]
On the other hand, it follows from Lemma~\ref{var} and (V9)  that
\[
 \lVert e^{\theta g(\omega, \cdot)}\rVert_{L^\infty}\le e^{\lvert \theta \rvert M} \quad \text{and} \quad \var(e^{\theta g(\omega, \cdot)})\le \lvert \theta\rvert e^{\lvert \theta \rvert M}
 \var (g(\om, \cdot))
\]
and thus using~(V3), 
\begin{equation}\label{eq:boundMultip}
 \begin{split}
  \lVert e^{\theta g(\omega, \cdot)}h\rVert_{\BV} &=\var (e^{\theta g(\omega, \cdot)}h)+\lvert e^{\theta g(\omega, \cdot)}h\rvert_1 \\
&\le e^{\lvert \theta \rvert M} \lVert h\rVert_{\BV}+\lvert \theta \rvert e^{\lvert \theta \rvert M}\var (g(\om, \cdot))\lVert h\rVert_{L^\infty}
 \\
 &\le (e^{\lvert \theta \rvert M}+C_{\var}\lvert \theta \rvert e^{\lvert \theta \rvert M}\esssup_{\om\in \Om} \var(g(\om, \cdot)))\lVert h\rVert_{\BV}.
 \end{split}
\end{equation}
We now establish part 1 of the Lemma. It follows from~\ref{cond:unifNormBd} that
\[
 \lVert \mathcal L_{\omega}^\theta (h)\rVert_{\BV} =\lVert \mathcal L_\omega
(e^{\theta g(\omega, \cdot)}h)\rVert_{\BV} \le K \lVert e^{\theta g(\omega, \cdot)}h\rVert_{\BV}.
\]
Hence, \eqref{eq:boundMultip} implies that~\eqref{se2} holds with
\begin{equation}\label{ktheta}
 K(\theta)=K(e^{\lvert \theta \rvert M}+C_{\var}\lvert \theta \rvert e^{\lvert \theta \rvert M}\esssup_{\om\in \Om} \var(g(\om, \cdot))).
\end{equation}

For part 2 of the Lemma, we observe that
 \[ |(M_\om^{\theta_1}-M_\om^{\theta_2})h\|_{\BV} \leq \|M_\om^{\theta_1}\|_{\BV} \|(I- M_\om^{\theta_2-\theta_1})\|_{\BV} \|h\|_{\BV}.\]
By~\eqref{obs} and  the mean value theorem for the map $z\mapsto e^{(\theta_1-\theta_2)z}$, we have
that for each $x\in X$,
\[
\lvert e^{(\theta_1-\theta_2) g(\om, x)}-1\rvert \le Me^{\lvert \theta_1-\theta_2\rvert M}\lvert \theta_1-\theta_2\rvert.
\]
Thus,
\begin{equation}\label{417}
\|1-e^{(\theta_2-\theta_1)g(\om, \cdot)}\|_{L^\infty} \le Me^{\lvert \theta_1-\theta_2\rvert M}\lvert \theta_1-\theta_2\rvert
\end{equation}
and
\begin{equation}\label{412}
\lvert (I- M_\om^{\theta_2-\theta_1})h\rvert _1 \leq  Me^{\lvert \theta_1-\theta_2\rvert M}\lvert \theta_1-\theta_2\rvert \cdot  \lvert h\rvert_1.
\end{equation}
 Assume that $|\theta_2-\theta_1|\leq 1$. We note that conditions (V3) and (V8) together with~\eqref{417} and Lemma~\ref{L5}  imply
\begin{equation}\label{407}
\begin{split}
\var ( (I- M_\om^{\theta_2-\theta_1})h) &\leq \big(\|1-e^{(\theta_2-\theta_1)g(\om, \cdot)}\|_{L^\infty} + C_{var}\var(1-e^{(\theta_2-\theta_1)g(\om, \cdot)}) \big)
 \|h\|_{\BV} \\
& \leq C'|\theta_2-\theta_1| \|h\|_{\BV},
\end{split}
\end{equation}
for some $C'>0$.
Hence, it follows from~\eqref{412}  and~\eqref{407}  that  $\theta \mapsto M_\om^\theta$ is continuous in the norm topology of $\BV$. Continuity of $\theta \mapsto \mcl_\om^\theta$ then follows immediately from continuity of $\mcl_\om$ and the definition of $\mcl_\om^\theta$, in \eqref{eq:twisted}.
\end{proof}

The following lemma shows that the twisted cocycle naturally appears in the study of Birkhoff sums~\eqref{birkhoff}.

\begin{lemma}\label{lem:exprLit}
The following statements hold:
\begin{enumerate}
\item for every $\phi \in \BV^*, f \in \BV$, $\om \in \Om$,  $\theta \in \mathbb C$ and $n\in \N$ we have that
\begin{equation}\label{eq:adjointPert}
\mcl_\om^{\theta, (n)}(f)=\mcl_\om^{(n)}(e^{\theta S_{n}g(\om, \cdot)}f), \quad \text{and} \quad
 \mcl_\om^{\theta*,(n)}(\phi)  =  e^{\theta S_ng(\om, \cdot)}  \mcl_\om^{*(n)}(\phi),
\end{equation}
where $(e^{\theta S_ng(\om, \cdot)} \phi) (f):= \phi (e^{\theta S_ng(\om, \cdot)} f)$;
\item for every $f\in \BV$, $\om \in \Om$ and $n\in \N$ we have that
\begin{equation}\label{intprop}\int \mathcal{L}^{\theta, \, (n)}_\omega (f)\ dm=\int e^{\theta S_ng(\omega, \cdot )}f\ dm. \end{equation}
\end{enumerate}
\end{lemma}

\begin{proof}

We establish the first identity in~\eqref{eq:adjointPert} by induction on $n$. The case $n=1$ follows from the definition of $\mcl_\om^{\theta}$.
We recall that for every $f, \tilde{f}\in \BV$,
\begin{equation}\label{eq:prodRuleP}
\mcl_\om^{(n)}((\tilde{f} \circ T_\om^{(n)})  \cdot f) = \tilde{f}\cdot  \mcl_\om^{(n)}( f).
\end{equation}
 Assuming the claim holds for some $n\geq 1$, we get
\begin{align*}
 \mcl_{\om}^{(n+1)} (e^{\theta S_{n+1}g(\om, \cdot)}f)
 &= \mcl_{\sig^n\om} \big(\mcl_{\om}^{(n)} ( e^{\theta g(\sig^{n} \om, \cdot)\circ T_\om^{(n)}}  e^{\theta S_{n}g(\om, \cdot)}f) \big) \\
 &=
 \mcl_{\sig^n\om} \big( e^{\theta g(\sig^{n} \om, \cdot)} \mcl_{\om}^{(n)} ( e^{\theta  S_{n}g(\om, \cdot)}f) \big)
 = \mcl_{\sig^n\om}^{\theta }\mcl_\om^{\theta , (n)}(f) = \mcl_\om^{\theta , (n+1)}(f).
\end{align*}
The second identity in~\eqref{eq:adjointPert}  follows directly from duality. Finally, we note that the second assertion of the lemma follows by integrating  the first equality in~\eqref{eq:adjointPert} with respect to $m$ and using the fact that $\mathcal L_\om^n$ preserves integrals with respect to $m$.
\end{proof}

\subsection{An auxiliary existence and regularity result}
In this section we establish a regularity result, Lemma~\ref{thm:IFT}, which generalises a theorem of~Hennion and Herv\'e \cite{HennionHerve} to the random setting.
This result will be used later  to show regularity of the top Oseledets space $Y_\om^\theta:= Y_1^{\theta}(\om)$ of the twisted cocycle, for $\theta$ near 0.

Let
\begin{equation}\label{defS}
\begin{split}
\mc{S}:=\Big\{ \mc{V} &: \Om \times X \to \mathbb C \ | \  \mc{V} \text{ is measurable}, \mc{V}(\om, \cdot)\in \BV,\\
 &\esssup_{\om \in \Om} \| \mc{V}(\om, \cdot) \|_{\BV}<\infty,
 \int \mc{V}(\om, x) dm =0 \text{ for } \paeom  \Big\},
 \end{split}
 \end{equation}
 endowed with the Banach space structure defined by the norm
\begin{equation}\label{infnorm}
\|\mc{V}\|_\infty := \esssup_{\om \in \Om} \| \mc{V}(\om, \cdot) \|_{\BV}.
\end{equation}
For $\theta \in \C$ and $\mc{W} \in \mc S$, set
\begin{equation}\label{defF}
F(\theta, \mc{W})(\om, \cdot)= \frac{\Lsot(\mc{W}(\sig^{-1}\om, \cdot) + v_{\sig^{-1}\om}^0(\cdot))}{\int \Lsot(\mc{W}(\sig^{-1}\om, \cdot) + v_{\sig^{-1}\om}^0(\cdot)) dm} - \mc{W}(\om,\cdot) - v_\om^0(\cdot).
\end{equation}

%

\begin{lemma}\label{lem:FwellDef}
There exist  $\ep, R>0$ such that
  $F \colon \mc{D} \to \mc{S}$ is a well-defined map on $\mc{D}:=\{ \theta \in \C : |\theta|<\ep \} \times B_{\mc{S}}(0,R)$, where $B_{\mc{S}}(0,R)$ denotes the ball of radius $R$ in $\mc{S}$ centered at $0$.
\end{lemma}
\begin{proof}
We define a map $H$  by
\[
 H(\theta, \mc W)(\om)=\int \Lsot(\mc{W}(\sig^{-1}\om, \cdot) + v_{\sig^{-1}\om}^0(\cdot)) \, dm=\int e^{\theta g(\sigma^{-1}\om, \cdot)}(\mc{W}(\sig^{-1}\om, \cdot) + v_{\sig^{-1}\om}^0(\cdot)) \, dm.
\]
It is proved in Lemmas~\ref{1108} and \ref{L6} of Appendix~\ref{regofF} that $H$ is a well-defined and differentiable function on a neighborhood of $(0,0)$  (and thus in particular continuous) with values in $L^\infty(\Om, \bbp)$.  Moreover,  we observe that $H(0,0)(\om)=1$ for each $\om \in \Om$ and
therefore
\[
\begin{split}
 \lvert H(\theta, \mc W)(\om)\rvert \ge 1-\lvert H(0,0)(\om)-H(\theta, \mc W)(\om)\rvert \ge 1-\lVert H(0,0)-H(\theta, \mc W)\rVert_{L^\infty},
 \end{split}
\]
for \paeom. Continuity of $H$ implies that  $\lVert H(0,0)-H(\theta, \mc W)\rVert_{L^\infty} \le 1/2$
for all $(\theta, \mc W)$ in a neighborhood of $(0,0)$ and hence, in such neighborhood,
\[
 \essinf_{\om} \lvert H(\theta, \mc W)(\om)\rvert \ge 1/2.
\]
The above inequality together with~\eqref{eq:boundedv} and~\eqref{se2}  yields the desired conclusion.
\end{proof}

\begin{lemma}\label{thm:IFT}
Let $\mc{D}=\{ \theta \in \C : |\theta|<\ep \} \times B_{\mc{S}}(0,R)$ be as in Lemma~\ref{lem:FwellDef}. Then, by shrinking $\epsilon>0$ if necessary, we have that 
$F:\mc{D} \to \mc{S}$ is $C^1$  and
the equation
\begin{equation}
F(\theta, \mc{W})=0
\end{equation}
has a unique solution $O(\theta) \in \mc{S}$, for every $\theta$ in a neighborhood  of 0.
Furthermore, $O(\theta)$ is a $C^2$ function of $\theta$.
\end{lemma}
\begin{proof}
We notice that  $F(0,0)=0$. Furthermore, Proposition~\ref{prop:F-C2} of Appendix~\ref{sec:regofF} ensures  that $F$ is $C^2$
on a neighborhood $(0, 0)\in \mathbb C \times \mc{S}$, and
\[
 (D_2F(0,0) \mathcal X)(\om, \cdot)=\mathcal L_{\sigma^{-1} \om}(\mathcal X(\sigma^{-1}\om, \cdot))-\mathcal X(\om, \cdot), \quad \text{for $\om \in \Om$ and $\mathcal X\in \mathcal S$.}
\]
 We now prove that $D_2 F(0,0)$ is bijective operator.

For injectivity, we have that if $D_2F(0, 0)\mathcal X=0$ for some nonzero $\mc{X}\in \mc{S}$, then $\mcl_\om \mc{X}_\om = \mc{X}_{\sig\om}$ for \paeom. Notice that $\mc{X}_\om \notin \langle v^0_\om \rangle$ because $\int \mc{X}_\om(\cdot) dm=0$ and  $ \mc{X}_\om\neq 0$. Hence, this yields a
contradiction with the one-dimensionality of the top Oseledets space of the cocycle $\mathcal L$, given by Lemma~\ref{lem:qc+1dim}. Therefore,  $D_2F(0,0)$ is injective.
To prove surjectivity, take $\mc{X}\in \mc{S}$ and
 let \begin{equation}\label{inverse} \tld{\mc{X}}(\om, \cdot):= - \sum_{j=0}^\infty  \mcl_{\sig^{-j}\om}^{(j)} \mc{X}(\sig^{-j}\om, \cdot). \end{equation} It follows from~\ref{cond:dec} that $\tld{\mc{X}}
 \in \mathcal S$ and it is easy to verify that
$D_2 F(0,0)\tld{\mc{X}}=\mc{X}$.
Thus, $D_2F(0,0)$ is surjective.

Combining the previous arguments, we conclude that $D_2F(0,0)$ is bijective. The conclusion of the lemma  now follows directly from the  implicit function theorem for Banach spaces (see, e.g.\ Theorem 3.2 \cite{avez}).
\end{proof}

%

We end this section with a specialisation of the previous results to real valued $\theta$.
\begin{proposition}\label{prop:Density}
 There exists $\delta >0$ such that for each $\theta \in (-\delta, \delta)$, $O(\theta)(\om, \cdot)+v^0_\om$ is a density for \paeom.
\end{proposition}

 We first show the following auxiliary result.
 \begin{lemma}\label{lem:WReal}
  For $\theta \in \R$ sufficiently close to $0$, $O(\theta)$ is real-valued.
 \end{lemma}

 \begin{proof}
 We consider the space
 \[
 \begin{split}
\tilde{\mc{S}}:=\Big\{ \mc{V} &: \Om \times \X \to \mathbb R \ | \  \mc{V} \text{ is measurable}, \mc{V}(\om, \cdot)\in \BV,\\
 &\esssup_{\om \in \Om} \| \mc{V}(\om, \cdot) \|_{\BV}<\infty,
 \int \mc{V}(\om, x) dm =0 \text{ for } \paeom  \Big\}.
 \end{split}
 \]
Hence, $\tilde{\mc{S}}$ consists of real-valued functions $\mc{V} \in \mc{S}$.
 We note that $\tilde{\mc {S}}$ is a Banach space with the norm $\lVert \cdot \rVert_\infty$ defined by~\eqref{infnorm}. Moreover, we can define  a map $\tilde F$ on a neighborhood of $(0,0)$ in
 $\mathbb R \times \tilde{\mc{S}}$ with values in $\tilde{\mc{S}}$ by the RHS of~\eqref{defF}. Proceeding as in  Appendix~\ref{regofF}, one can show that $\tilde F$ is a
 differentiable map on a neighborhood of $(0,0)$. Moreover, arguing as in the proof of Lemma~\ref{thm:IFT} one can conclude that for $\theta$ sufficiently close to $0$, there exists a unique
 $\tilde{O}(\theta )\in \tilde{\mc{S}}$ such that $\tilde F(\theta, \tilde{O}(\theta))=0$ and that $\tilde{O}(\theta)$ is differentiable with respect to $\theta$.
 Since $\tilde{\mc{S}} \subset \mc{S}$ and from the uniqueness property in the implicit function
 theorem, we conclude that $O(\theta)=\tilde{O}(\theta)$ for $\theta$ sufficiently close to $0$ which immediately implies the conclusion of the lemma.
\end{proof}

\begin{proof} [Proof of Proposition~\ref{prop:Density}]
By Lemma~\ref{lem:WReal}, for $\theta$ sufficiently close to $0$, $O(\theta)(\om, \cdot) +v_\om^0(\cdot)$ is real-valued. Moreover, $\int (O(\theta)(\om, \cdot) +v_\om^0(\cdot))\, dm=1$ for a.e.
$\om \in \Om$. It remains to show that $O(\theta)(\om, \cdot) +v_\om^0(\cdot) \ge 0$ for \paeom. Since the map $\theta \mapsto O(\theta)$ is continuous,  there
exists $\delta >0$ such that for all $\theta \in (-\delta, \delta)$, $O(\theta)$ belongs to a ball of radius $c/(2 C_{\var})$ centered at $0$ in $\mc{S}$. In particular,
\[
 \esssup_{\om \in \Om}\lVert O(\theta) (\om, \cdot) \rVert_{\BV} < c/(2 C_{\var})
\]
and therefore,
\[
 \esssup_{\om \in \Om}\lVert O(\theta) (\om, \cdot) \rVert_{L^\infty} < c/2.
\]
By~\eqref{lowerbound},
\[
 \essinf (O(\theta) (\om, \cdot)+v^0_\om(\cdot)) \ge c/2, \quad \text{for a.e. $\om \in \Omega$,}
\]
which completes the proof of the proposition.
\end{proof}

\subsection{A lower bound on $\Lam(\theta)$}
\label{LMB}

The goal of this section is to establish a differentiable lower  bound ($\hat\Lam(\theta)$)  on $\Lam(\theta)$, the top Lyapunov exponent of the twisted cocycle, for $\theta\in \C$ in a neighborhood of $0$. In Section~\ref{sec:quasicompactnessTwisted}, we will show that this lower bound in fact coincides with $\Lam(\theta)$, and hence all the results of this section will immediately translate into properties of $\Lam$.

Let $0<\ep<1$ and $O(\theta)$ be as in Lemma~\ref{thm:IFT}.
Let
\begin{equation}\label{eq:vomt}
v_\om^\theta(\cdot):= v_\om^0(\cdot) +O(\theta)(\om,\cdot).
\end{equation}
We notice that $\int v_\om^\theta(\cdot)\ dm =1$ and by Lemma~\ref{thm:IFT}, $\theta \mapsto v^\theta$ is continuously differentiable.
 Let us define
\begin{equation}\label{eq:hatLam}
 \hat\Lambda (\theta) :=  \int \log \Big|\int e^{\theta g(\om, x)} v_\om^\theta(x) \,d m(x) \Big|\, d\bbp(\om),
\end{equation}
and
\begin{equation}\label{eq:int}
\lot :=  \int e^{\theta g(\om, x)} v_\om^\theta(x) \,d m(x)
= \int \mcl_\om^\theta v_\om^\theta(x) \,d m(x),
\end{equation}
where the last identity follows from~\eqref{intprop}. Notice also that $\om \mapsto \lot$ is an integrable function.

\begin{lemma}\label{lem:lowerBoundLam}
For every $\theta \in B_\C(0,\ep):= \{ \theta \in \C : |\theta|<\ep \}$, $ \hat\Lambda (\theta)\leq \Lambda (\theta)$.
\end{lemma}
\begin{proof}
Recall that $O(\theta)$ satisfies the equation $F(\theta, O(\theta))=0$, for $\theta \in \{ \theta \in \C : |\theta|<\ep \}$.
 Hence, for \paeom,  $v_\om^\theta(\cdot)$ satisfies the equivariance equation
 $\mcl^\theta_\om v_\om^\theta(\cdot) = \lam^\theta_{\om} v_{\sig\om}^\theta(\cdot)$.
Thus, using Birkhoff's ergodic theorem to go from the first to the second line below, we get
\begin{equation*}
\begin{split}
\Lam(\theta) & \geq \lni \frac1n \log \|\mcl_\om^{\theta,(n)} v_\om^\theta \|_{\mathcal B} \geq \lni \frac1n \log \|\mcl_\om^{\theta,(n)} v_\om^\theta \|_1\geq \lni  \frac1n  \sum_{j=0}^{n-1}\log |\lam^\theta_{\sig^{j}\om}|  \\
&  = \int \log |\lam^\theta_{\om}| d\bbp(\om)=
 \int \log \Big|\int e^{\theta g(\om, x)} v_\om^\theta(\cdot) \,d m(x)\Big| \, d\bbp(\om) = \hat\Lambda (\theta).
 \end{split}
\end{equation*}
\end{proof}

The rest of the section deals with differentiability properties of  $\hat\Lambda (\theta)$.
 From now on  we  shall also use the notation $O(\theta)_\om$ for $O(\theta)(\om, \cdot)$.

%

\begin{lemma}\label{difflambda}
We have that $\hat\Lambda$ is differentiable on a neighborhood of 0, 
and
\[
\hat \Lambda' (\theta)=
\Re \Bigg( \int \frac{ \overline{\lot}  ( \int g(\om, \cdot)e^{\theta g(\om, \cdot)}v_\om^\theta(\cdot)+ e^{\theta g(\om, \cdot)}O'(\theta)_\om (\cdot)\, dm )}{|\lot |^2}\, d\bbp(\om) \Bigg),
\]
where $\Re (z)$ denotes the real part of $z$ and $\overline{z}$ the complex conjugate of $z$.
\end{lemma}

\begin{proof}
 Write
 \[
  \hat \Lambda(\theta)=\int Z(\theta, \omega) \, d\mathbb P(\omega),
 \]
where
\[
 Z(\theta, \omega):= \log | \lot|=\log \Big| \int e^{\theta g(\om, x)}v_\om^\theta(x)\,d m(x)\Big|.
\]
Note that $Z(\theta, \om)=\log |H(\theta, O(\theta))(\sigma \om)|$, where $H$ is as in Lemma~\ref{lem:FwellDef}. Since $H(0,0)=1$ and both $H$ and $O$ are continuous (by Lemma~\ref{thm:IFT}), there
is a neighborhood $U$ of $0$ in $\C$ on which $\lVert H(\theta, O(\theta))-H(0,0)\rVert_{L^\infty} <1/2$. In particular, $Z$ is well defined  and
$Z(\theta, \om)\in [\log\frac12, \log \frac32]$ for every $\theta\in U \cap B_\C(0,\ep)$ and \paeom.
Thus, the map $\omega \mapsto Z(\theta, \omega)$ is $\mathbb P$-integrable for every $\theta\in U \cap B_\C(0,\ep)$.

%

It follows from Lemma~\ref{lem:Z'} below that for \paeom, the map $\theta \mapsto Z_\om(\theta):=Z(\theta, \omega)$ is differentiable in a neighborhood of 0, and
\[
 Z_\om'(\theta)=\frac{\Re \Big( \overline{\lot}   ( \int g(\om, \cdot)e^{\theta g(\om, \cdot)}v_\om^\theta(\cdot)+ e^{\theta g(\om, \cdot)}O'(\theta)_\om (\cdot)\, dm )\Big)}{| \lot|^2},
\]
where $\Re (z)$ denotes the real part of $z$ and $\overline{z}$ the complex conjugate of $z$.
In particular,
\[
 |Z_\om'(\theta)| \leq \frac{\big|\int (g(\om, x)e^{\theta g(\om, x)}v_\om^\theta(\cdot)+ e^{\theta g(\om, x)}O'(\theta)_\om (x))\, dm(x) |}{ |\int e^{\theta g(\om, x)}v_\om^\theta(x) \,d m(x)|}.
\]
We claim that there exists an integrable function $C\colon \Omega \to \mathbb R$  such that
\begin{equation}\label{nn}
 \lvert Z_\om'(\theta)\rvert \le C(\om), \quad \text{for all $\theta$ in a neighborhood of 0 and \paeom.}
\end{equation}
Once this is established, the conclusion of the lemma follows from Leibniz rule for exchanging the order of differentiation and integration.

To complete the proof, let us show \eqref{nn}. For $\theta \in U$ we have
\[
\Big|\int e^{\theta g(\om, x)}v_\om^\theta(x)\, dm(x) \Big|\geq \frac12.
 \]
Also, recall  that  $\ep<1$, so that for  $\theta\in B_\C(0,\ep)$ one has
\[
\begin{split}
\bigg{ \lvert} \int g(\om, x)e^{\theta g(\om,x)}&v_\om^\theta(x) \, dm(x) \bigg{\rvert} \le \int \big \lvert g(\om, x)  e^{\theta g(\om, x)}v_\om^\theta(x) \big \rvert \, dm(x) \\
&\le Me^M \lvert O(\theta)_\om + v^0_\om \rvert_1 \leq Me^M (1+ \| O(\theta)_\om \|_{\BV})
\leq Me^M (1+ \| O(\theta) \|_\infty).
\end{split}
\]
Finally,
\[
 \bigg{\lvert} \int e^{\theta g(\om,x)}O'(\theta)_\om (x)\, dm(x) \bigg{\rvert} \le e^M \lvert O'(\theta)_\om \rvert_1 \le e^M \lVert O'(\theta)_\om\rVert_{\BV}\le e^M \lVert O'(\theta)\rVert_\infty,
\]
for \paeom.
Since  $O$ and $O'$ are continuous by Lemma~\ref{thm:IFT}, the terms on the RHS of the above inequalities are uniformly bounded for $\theta$ in a (closed) neighborhood of 0. Hence, \eqref{nn} holds for a constant function $C$.
\end{proof}

\begin{lemma}\label{lem:Z'}
For \paeom, and $\theta$ in a neighborhood of 0, the map $\theta \mapsto Z_\om(\theta):=Z(\theta, \omega)$ is differentiable. Moreover,
\[
 Z_\om'(\theta)=\frac{\Re \Big( \overline{\lot}   ( \int g(\om, \cdot)e^{\theta g(\om, \cdot)}v_\om^\theta(\cdot)+ e^{\theta g(\om, \cdot)}O'(\theta)_\om (\cdot)\, dm )\Big)}{| \lot|^2},
\]
where $\Re (z)$ denotes the real part of $z$ and $\overline{z}$ the complex conjugate of $z$.
\end{lemma}

\begin{proof}
First observe that if $\theta\mapsto f(\theta)\in \C$, has polar decomposition $f(\theta)=r(\theta) e^{i\phi(\theta)}$,
then, whenever $|f|(\theta)\neq 0$, $\frac{d |f|(\theta)}{d \theta} = \frac{\Re(\bar{f}(\theta)f'(\theta) )}{r(\theta)}$, where $f'$ denotes differentiation with respect to $\theta$.
Thus, by the chain rule, it is sufficient to prove that the map $\lot$
is differentiable with respect to $\theta$ and that
\begin{equation}\label{eq:Iprime}
D_\theta \lot=\int \Big( g(\om,x)e^{\theta g(\om, x)}v_\om^\theta(x)+ e^{\theta g(\om, x)}O'(\theta)_\om (x)\Big)\, dm(x).
\end{equation}
Using the same notation as in Lemma~\ref{lem:FwellDef}, we can write
\[
\lot=H(\theta, O(\theta))(\sigma \om)=:P(\theta)(\sigma \om).
\]
We note that $P$ is a differentiable map with values in $L^\infty (\Omega)$. Indeed, this follows directly from the regularity properties of $H$ established in
Lemmas~\ref{1108} and~\ref{L6} and the differentiability of $O$ (see Lemma~\ref{thm:IFT}) together with the chain rule. Since
\[
 \frac{\lvert \lam_\om^{\theta+t}-\lam_\om^\theta -P'(\theta)(\sigma \om)\rvert}{\lvert t\rvert} \le \frac{\lVert P(\theta+t)-P(\theta)-P'(\theta)\rVert_{L^\infty(\Om)}}{\lvert t\rvert},
\]
for $\mathbb P$-a.e. $\om \in \Om$ and  $t, \theta$ close to $0\in \mathbb C$, we conclude that $\lot$ is differentiable  with respect to $\theta$ in a neighborhood of $0\in \mathbb C$.

\end{proof}

\begin{lemma}\label{zero}
 We have that $\hat \Lambda'(0)=0$.
\end{lemma}

\begin{proof}
Let $F$ be as in Lemma~\ref{thm:IFT}.
 By identifying $D_1 F(0,0)$ with its value at $1$, it follows from the implicit function theorem that
 \[
  O'(0)=-D_2F(0,0)^{-1}( D_1F(0,0)).
 \]
It is shown in Lemma~\ref{thm:IFT} that $D_2F(0,0) \colon \mathcal S \to \mathcal S$ is bijective. Thus, $D_2F(0,0)^{-1} \colon \mathcal S \to \mathcal S$ and therefore  $O'(0) \in \mathcal S$ which implies that
\begin{equation}\label{zerod}
\int O'(0)_\om \, dm(x)=0 \text{ for \paeom.}
\end{equation}
 The conclusion of the lemma
follows directly from Lemma~\ref{difflambda} and the centering condition~\eqref{zeromean}.
\end{proof}

\subsection{Quasicompactness of twisted cocycles and differentiability of $\Lam(\theta)$}\label{sec:quasicompactnessTwisted}
In this section we establish quasicompactness of the twisted transfer operator cocycle, as well as differentiability of the top Lyapunov exponent with respect to $\theta$, for $\theta \in \C$ near $0$.
\begin{thm}[Quasi-compactness of twisted cocycles, $\theta$ near 0]\label{cor:quasicompactness}
Assume that the cocycle $\mc{R}=(\Om, \mathcal F, \bbp, \sig, \B, \mathcal L)$ is admissible. 
For $\theta \in \C$ sufficiently close to $0$, we have that the twisted cocycle $\mathcal L^\theta$ is quasi-compact.
Furthermore, for such $\theta$, the top Oseledets space of $\mathcal L^\theta$ is one-dimensional. That is, $\dim Y^\theta(\om)=1$ for \paeom.
\end{thm}

The following Lasota-Yorke type estimate will be useful in the proof.
\begin{lemma}\label{lem:LYtwisted}
Assume conditions~\ref{cond:unifNormBd} and \ref{C0} hold.
 Then, we have
 \[
  \lVert \mcl_\om^{\theta, (N)} f\rVert_{\BV} \le \tilde \alpha^{\theta, N}(\om)\lVert f\rVert_{\BV} +\beta^N(\om) \lVert f\rVert_1,
 \]
where
\[
 \tilde \alpha^{\theta, N}(\om)= \alpha^N(\om)+C\lvert \theta\rvert e^{\lvert \theta \rvert M}\sum_{j=0}^{N-1}K^{N-1-j}K(\theta)^j,
\]
for some constant  $C>0$ where $K(\theta)$ is given by Lemma~\ref{1214} and K is given by~\ref{cond:unifNormBd}.
\end{lemma}
\begin{proof}
 It follows from~\ref{C0} that
 \[
 \begin{split}
  \lVert \mcl_\om^{\theta, (N)} f\rVert_{\BV} &\le \lVert  \mcl_\om^{ (N)} f\rVert_{\BV}+\lVert \mcl_\om^{\theta, (N)}-\mcl_\om^{(N)}\rVert_{\BV} \cdot \lVert f\rVert_{\BV} \\
  &\le
  \alpha^N(\om)\lVert f\rVert_{\BV}+\beta^N(\om)\lVert f\rVert_1+\lVert \mcl_\om^{\theta, (N)}-\mcl_\om^{(N)}\rVert_{\BV} \cdot \lVert f\rVert_{\BV}.
  \end{split}
 \]
On the other hand, we have that
\[
 \mcl_\om^{\theta, (N)}-\mcl_\om^{(N)}=\sum_{j=0}^{N-1} \mcl_{\sigma^{N-j} \om}^{\theta, (j)}(\mcl_{\sigma^{N-1-j} \om}^\theta -\mcl_{\sigma^{N-1-j} \om})\mcl_\om^{(N-1-j)}.
\]
It follows from~\ref{cond:unifNormBd} and~\eqref{se2} that
\[
 \lVert \mcl_\om^{(N-1-j)}\rVert_{\BV} \le K^{N-1-j} \quad \text{and} \quad \lVert \mcl_{\sigma^{N-j} \om}^{\theta, (j)}\rVert_{\BV} \le K(\theta)^j.
\]
Furthermore, using (V3) and (V8), we have that for any $h\in \BV$,
\[
\begin{split}
 \lVert (\mcl_\om^\theta -\mcl _\om)(h) \rVert_{\BV} &=\lVert \mcl_\om (e^{\theta g(\om, \cdot)}h-h)\rVert_{\BV} \\
 &\le K\lVert (e^{\theta g(\om, \cdot)}-1)h\rVert_{\BV}\\
 &=K\var((e^{\theta g(\om, \cdot)}-1)h)+K\lVert (e^{\theta g(\om, \cdot)}-1)h\rVert_1 \\
 &\le K\lVert e^{\theta g(\om, \cdot)}-1\rVert_{L^\infty} \cdot \var (h)+K\var (e^{\theta g(\om, \cdot)}-1)\cdot \lVert h\rVert_{L^\infty}\\
 &\phantom{\le}+K\lVert e^{\theta g(\om, \cdot)}-1\rVert_{L^\infty} \cdot \rVert h\rVert_1 \\
 &\le K\lVert e^{\theta g(\om, \cdot)}-1\rVert_{L^\infty}\lVert h\rVert_{\BV}+KC_{\var}\var (e^{\theta g(\om, \cdot)}-1)\cdot \lVert h\rVert_{\BV}.
 \end{split}
\]
 By applying the mean-value theorem for the map $z\mapsto e^{\theta z}$ and using~\eqref{obs}, we obtain that $\lVert e^{\theta g(\om, \cdot)}-1\rVert_{L^\infty}\le \lvert \theta \rvert e^{\lvert \theta \rvert M}
 M$. Furthermore, it follows from (V9) (applied to $h(z)=e^{\theta z}-1$ and $f=g(\om, \cdot)$) together with~\eqref{obs} that $\var (e^{\theta g(\om, \cdot)}-1)\le 
 \lvert \theta\rvert e^{\lvert \theta \rvert M}\var(g(\om, \cdot))$.
Therefore,
\[
 \lVert \mcl_\om^{\theta, (N)}-\mcl_\om^{(N)}\rVert_{\BV} \le C\lvert \theta\rvert e^{\lvert \theta \rvert M}\sum_{j=0}^{N-1}K(\theta)^j K^{N-1-j},
\]
where
\[
 C=KM+KC_{\var} \esssup_{\om \in \Om} (\var g(\om, \cdot))
\]
and the conclusion of the lemma follows by combining the above estimates.
\end{proof}

Theorem~\ref{cor:quasicompactness} may now be established as follows.
\begin{proof}[Proof of Theorem~\ref{cor:quasicompactness}]
 It follows from Lemma~\ref{lem:LYtwisted} and the dominated convergence theorem that
 \[
  \int_\Om \log \tilde \alpha^{\theta, N} (\om)\,  d\mathbb P(\om) \to  \int_\Om \log \alpha^N (\om)\,  d\mathbb P(\om)<0 \quad \text{when $\theta \to 0$.}
 \]
Thus, there exists $\delta >0$ such that
\[
 \int_\Om \log \tilde \alpha^{\theta, N} (\om)\,  d\mathbb P(\om) \le \frac{1}{2}\int_\Om \log \alpha^N (\om)\,  d\mathbb P(\om), \quad \text{for $\theta \in B_{\C}(0, \delta)$.}
\]
Lemma~\ref{lem:lowerBoundLam} implies that $\Lambda$ is bounded below by a continuous function $\hat\Lambda$ in a neighborhood of 0, and $\Lambda(0)=\hat\Lambda(0)=0$. 
Hence, by decreasing $\delta$ if necessary, we can assume that \[N\Lambda(\theta)> \frac{1}{2}\int_\Om \log \alpha^N (\om)\,  d\mathbb P(\om)
\quad \text{for $\theta \in B_{\C}(0, \delta)$.}\]
Therefore,
\begin{equation}\label{6io}
 N\Lambda(\theta)>\int_\Om \log \tilde \alpha^{\theta, N} (\om)\,  d\mathbb P(\om) \quad \text{for $\theta \in B_{\C}(0, \delta)$.}
\end{equation}
Let $\mc{R}^{\theta(N)}$ denote the cocycle 
 over $\sigma^N$ with generator $\om \mapsto \mcl_\om^{\theta, (N)}$. We claim that
 \begin{equation}\label{7io}
  \Lambda (\mc{R}^{\theta(N)})=N\Lambda (\theta) \quad \text{and} \quad \kappa(\mc{R}^{\theta(N)})=N\kappa (\mc R^\theta).
 \end{equation}
Indeed, we have that
\[
 \Lambda (\mc{R}^{\theta(N)})=\lim_{n\to \infty}\frac 1 n \log \lVert \mcl_{\sigma^{(n-1)N}\om}^{\theta, (N)} \cdot \ldots \cdot \mcl_{\sigma^N \om}^{\theta, (N)} \cdot \mcl_\om^{\theta, (N)}\rVert=
 N\lim_{n\to \infty}\frac{1}{nN} \log \lVert \mcl_\om^{\theta, (nN)}\rVert=N\Lambda(\theta),
\]
which proves the first equality in~\eqref{7io}. Similarly, one can establish the second identity in~\eqref{7io}. We now note that Lemmas~\ref{lem:Hennion} and~\ref{lem:LYtwisted} together 
with~\eqref{6io} and the first identity in~\eqref{7io} imply that the cocycle $\mc{R}^{\theta(N)}$ is quasicompact, i.e. $\Lambda(\mc{R}^{\theta(N)})>\kappa(\mc{R}^{\theta(N)})$.
Hence, \eqref{7io} implies  that $\Lambda(\mc{R}^\theta)>\kappa (\mc{R}^\theta)$ and we conclude that $\mc{R}^\theta$ is a quasicompact cocycle.

Now we show $\dim Y^\theta:=\dim Y_1^\theta=1$.
Let $\lam_1^\theta=\mu_1^\theta \geq \mu_2^\theta \geq \dots \geq \mu_{L_\theta}^\theta>\ka(\theta)$ be the exceptional Lyapunov exponents of twisted cocycle $\mathcal L_\om^\theta$, enumerated with multiplicity.
That is,  $m_j^\theta= \dim Y_j^{\theta}(\om)$ denotes the multiplicity of the Lyapunov exponent $\lam_j^\theta$. As in Theorem~\ref{thm:MET}, let $M_j^\theta:=m^\theta_1+\dots +m^\theta_{j}$. Therefore,
$\Lam(\theta)=\lam_1^\theta=\mu_i^\theta$ for every $1\leq i \leq M^\theta_1$
and
$\lam_{j}^\theta=\mu_i^\theta$ for every $M^\theta_{j-1}+1 \leq i \leq M^\theta_j$ and for every finite $1<j\leq l_\theta$.
By Lemma \ref{1214}(2) the map $\theta \mapsto \mcl_\om^\theta$ is continuous in the norm topology of $\BV$ for every $\om \in \Om$ and also that the functions $\om \mapsto \log^+\|\mcl_\om^\theta\|$ are dominated by an integrable function whenever $\theta$ is restricted to a compact set.
Thus, Lemma~\ref{lem:usc} of Appendix~\ref{sec:usc} shows that $\theta \mapsto \mu^\theta_1 + \mu^\theta_2$ is upper-semicontinuous.
 Hence,
$$
0>\mu^0_1 + \mu_2^0  \geq  \limsup_{\theta\to0} ( \mu^\theta_1 + \mu^\theta_2),
$$
where the first inequality follows from the one-dimensionality of the top Oseledets subspace of the cocycle $\mathcal L_\om$. We note that
Lemmas~\ref{lem:lowerBoundLam} and \ref{difflambda}, ensure that $\limsup_{\theta\to0} \mu^\theta_1 \geq \hat\Lam(0)=0$. Therefore
$\limsup_{\theta\to0} \mu^\theta_2 <0$ and $\dim Y_1^\theta=1$, as claimed.
\end{proof}

\begin{cor}\label{cor:LamHatLam}
For  $\theta\in \C$ near 0, we have that $\Lam(\theta)=\hat\Lam(\theta)$.
In particular, $\Lam(\theta)$ is differentiable near $0$ and $\Lam'(0)=0$.
\end{cor}
\begin{proof}
We recall that $\hat\Lam(0)=0$ and $\hat\Lam$ is differentiable near 0, by Lemma~\ref{difflambda}.
In addition, $v_\om^\theta(\cdot)$, defined in \eqref{eq:vomt}, gives a one-dimensional measurable equivariant subspace of $\BV$ which grows at rate $\hat\Lam(\theta)$ (see~\eqref{eq:hatLam}).
Theorem~\ref{cor:quasicompactness} shows that $\limsup_{\theta\to0} \mu^\theta_2 <0$.
In particular, $\mu^\theta_2 < \hat\Lam(\theta)$ for $\theta$ sufficiently close to 0.
Combining this information with the multiplicative ergodic theorem (Theorem~\ref{thm:MET}) and Lemma~\ref{lem:lowerBoundLam}, we get that $\Lam(\theta)=\hat\Lam(\theta)$ and  $Y_1^{\theta}(\om)= \langle v_\om^\theta\rangle$, for all $\theta\in \C$ near 0.  Thus, lemma~\ref{zero} implies that $\Lam'(0)=0$.
\end{proof}

\subsection{Convexity of $\Lambda(\theta)$}
We continue to denote by  $\mu$ the invariant measure for the skew product transformation $\tau$ defined in \eqref{eq:defmu}. Furthermore, let $S_ng $ be given by~\eqref{birkhoff}.
By expanding the term $[S_n g(\om, x)]^2$ it is straightforward to verify using standard computations and~\eqref{buzzi} that
\[
 \lim_{n\to \infty} \frac 1 n \int_{\Om \times X} [S_n g(\om, x)]^2 \, d\mu(\om, x)=\int_{\Om \times X}  g(\om, x)^2\, d\mu(\om, x)+2 \sum_{n=1}^\infty \int_{\Om \times X}  g(\om, x) g(\tau^n (\om, x))\, d\mu(\om, x)
\]
and that the right-hand side of the above equality is finite.
Set \begin{equation}\label{variance}
     \Sig^2 :=\int _{\Om \times X} g(\om, x)^2\, d\mu(\om, x)+2 \sum_{n=1}^\infty \int_{\Om \times X}  g(\om, x) g(\tau^n (\om, x))\, d\mu(\om, x).
    \end{equation}
Obviously, $\Sig^2 \ge 0$ and from now on we shall assume that $\Sig^2>0$.
This is equivalent to a non-coboundary condition on $g$; we refer the interested reader to \cite{DFGTV} for a precise statement characterising the degenerate case $\Sig^2=0$.

\begin{lemma}\label{lem:Lam''0}
We have that $\Lambda$ is of class $C^2$ on a neighborhood of $0$ and  $\Lambda''(0)=\Sig^2$.
\end{lemma}

\begin{proof}
Using the notation in subsection~\ref{LMB}, it follows from Lemma~\ref{difflambda} and Corollary~\ref{cor:LamHatLam} that
\[
\Lambda' (\theta)=
\Re \Big( \int \frac{ \overline{\lot}  ( \int g(\om, \cdot)e^{\theta g(\om, \cdot)}(O(\theta)_\om (\cdot)+v^0_\om (\cdot))+ e^{\theta g(\om, \cdot)}O'(\theta)_\om (\cdot)\, dm )}{|\lot|^2}\, d\bbp(\om) \Big) .
\]
Proceeding as in the proof of Lemma~\ref{difflambda}, one can show that $\Lambda$ is of class $C^2$ on a neighborhood of $0$ and that
\begin{equation}\label{sd}
\Lambda''(\theta)= \Re \Big( \int \frac{\lot \,''}{\lot} - \frac{(\lot \, ')^2}{(\lot)^2} \, d\mathbb P(\om) \Big),
\end{equation}
where we have used $'$ to denote derivative with respect to $\theta$.
We recall that $\lam_\om^0=1$, $\lot\,'$ is given by \eqref{eq:Iprime}, and in particular $\lot\,' |_{\theta=0}=0$ for \paeom.
It is then straightforward,  using~\eqref{sd}, the chain rule and the formulas in  Appendices~\ref{regofF} and~\ref{regofF2}, to verify that
\[
 \Lambda''(0)= \Re \Big( \int \int g(\om, x)^2 v_\om^0(x)+ 2g(\om, x)O'(0)_\om(x)+O''(0)_\om (x)\, dm(x)\, d\mathbb P(\om) \Big).
\]
Moreover, since $\theta\mapsto O'(\theta)$ is a map on a neighborhood of $0$ with values in $\mathcal S$ we can regard $O''(0)$ as an element of (the tangent space of) $\mathcal S$, which implies that
\[
 \int O''(0)_\om (x)\, dm(x)=0 \quad \text{for a.e. $\om$}
\]
and thus
\begin{equation}\label{0521}
 \Lambda''(0)= \Re \Big( \int \int (g(\om, x)^2 v_\om^0(x)+ 2g(\om, x)O'(0)_\om(x))\, dm(x)\, d\mathbb P(\om) \Big).
\end{equation}
On the other hand, by the implicit function theorem,
\[
 O'(0)_\om=-(D_2 F(0,0)^{-1}(D_1 F(0,0)))_\om.
\]
Furthermore, \eqref{inverse} implies that
\[
(D_2F(0,0)^{-1}\mathcal W)_\om=-\sum_{j=0}^\infty \mathcal L_{\sigma^{-j}\om}^{(j)} (\mc W_{\sigma^{-j} \om}),
\]
for each $\mc W\in \mc S$.
This together with Proposition~\ref{difF} gives that
 \begin{equation}\label{derO}
 O'(0)_\om=\sum_{j=1}^\infty \mathcal L_{\sigma^{-j}  \om}^{(j)} (g(\sigma^{-j} \om, \cdot) v_{\sigma^{-j} \om}^0(\cdot)).
\end{equation}
Using~\eqref{0521}, \eqref{derO}, the duality property of transfer operators, as well as the fact that   $\sigma$ preserves $\mathbb P$, we have that
\[
\begin{split}
\Lambda''(0) &=\int \bigg{[} \int g(\om, x)^2v_\om^0 \, dm (x)+2\sum_{j=1}^\infty \int g(\om, x)\mathcal L_{\sigma^{-j}  \om}^{(j)} (g(\sigma^{-j} \om, \cdot) v_{\sigma^{-j} \om}^0) \, dm(x) \bigg{]}\, d\mathbb P(\om) \\
&=\int \bigg{[} \int g(\om, x)^2 \, d\mu_\om (x)+2\sum_{j=1}^\infty  \int g(\om, T_{\sigma^{-j} \om}^{(j)} x) g(\sigma^{-j} \om, x)\, d\mu_{\sigma^{-j} \om} (x) \bigg{]}\, d\mathbb P(\om) \\
&=\int g(\om, x)^2 \, d\mu (\om, x)+2 \sum_{j=1}^\infty \int   \int g(\sigma^j \om, T_{ \om}^{(j)} x) g( \om, x)\, d\mu_{ \om} (x) \, d\mathbb P(\om) \\
&=\int g(\om, x)^2 \, d\mu (\om, x)+2\sum_{j=1}^\infty \int g(\om, x) g(\tau^j (\om, x))\, d\mu(\om, x) =\Sig^2.
\end{split}
\]
\end{proof}

The following result is a direct consequence of the previous lemma.

\begin{cor}\label{convex}
$\Lambda$ is strictly convex on a neighborhood of $0$.
\end{cor}

\subsection{Choice of bases for top Oseledets spaces $Y_\omega ^\theta$ and $Y_\omega ^{*\theta}$}
\label{sec:choiceOsBases}

We recall that $Y_\omega^\theta$ and $Y_\omega^{*\theta}$ are top Oseledets subspaces for twisted and adjoint twisted cocycle, $\mcl^\theta$ and $\mcl^{\theta*}$, respectively.
 The Oseledets decomposition for these cocycles  can be written in the form
\begin{equation}\label{sod}
\BV=Y^\theta_\om \oplus H^\theta_\om  \quad \text{ and } \quad
\BV^* = Y^{*\,\theta}_\om \oplus H^{*\,\theta}_\om,
\end{equation}
where $H^\theta_\om=V^\theta(\omega)\oplus\bigoplus_{j=2}^{l_\theta} Y^\theta_j(\omega)$ is the equivariant complement to $Y^\theta_\om:= Y_1^\theta(\om)$, and $H^{*\,\theta}_\om$ is defined similarly.
Furthermore, Lemma~\ref{lem:AnnihilatorOsSplittings} shows that the following duality relations hold:
\begin{equation}\label{eq:DualityRel}
\begin{split}
\psi(y)&=0 \text{ whenever } y \in Y^\theta_\om \text{ and } \psi \in H^{*\,\theta}_\om,\quad \text{ and }\\
\phi(f) &=0 \text{ whenever } \phi  \in Y^{*\, \theta}_\om \text{ and } f \in  H^{\theta}_\om.
\end{split}
\end{equation}

Let us fix convenient choices for elements of the one-dimensional top Oseledets spaces $Y^\theta_\om$ and $Y^{*\,\theta}_\om$, for $\theta \in \C$ close to $0$.
Let  $v_\omega^\theta\in Y^\theta_\om$ be as in \eqref{eq:vomt}, so that $\int v_\om^\theta(\cdot)dm=1$.
(In view of Proposition~\ref{prop:Density}, when $\theta\in \R$ close to $0$, the operators $\mcl_\om^\theta$ are positive, so we can additionally assume $\vot \geq 0$ and so $\lVert \vot\rVert_1=1$).

Since $\dim Y^\theta_\om=1$, $\vot$ is defined uniquely for \paeom.
Theorem~\ref{thm:MET} ensures that, for \paeom, there exists $\lot\in \C$ ($\lot> 0$ if $\theta \in \R$) such that
\begin{equation}
\label{eq:def_lambdas}
\mathcal{L}^{\theta}_\omega v^\theta_\omega=\lambda^\theta_\omega v^\theta_{\sigma\omega}.
\end{equation}
Integrating~\eqref{eq:def_lambdas}, and using (\ref{eq:int}), we obtain
\begin{equation}\label{eq:lam}
\lambda^\theta_\omega= \int e^{\theta g(\om,x)} v_\om^\theta(x) \, dm(x),
\end{equation}
and thus $\lambda_\omega^\theta$ coincides with the quantity introduced in~\eqref{eq:int}.
By~\eqref{eq:hatLam} and Corollary~\ref{cor:LamHatLam},
\begin{equation}\label{eq:LambdaFromlambda}
\Lam(\theta)= \int \log |\lot| \, d\bbp(\om).
\end{equation}
Next, let us fix $\phi^\theta_\omega \in Y^{*\,\theta}_\om$ so that $\phi^\theta_\omega(v^\theta_\omega)=1$. This selection is again possible and unique, because of~\eqref{eq:DualityRel}. Furthermore, this choice implies that
\begin{equation} \label{eq:DualEig}
(\mathcal{L}^{\theta}_\omega)^*\phi^\theta_{\sigma\omega}=\lambda^\theta_\omega \phi^\theta_{\omega},
\end{equation}
because $Y_\om^{*\theta}$ is one-dimensional and equivariant.
Indeed, if $C_\om^\theta$ is the constant such that $(\mathcal{L}^{\theta}_\omega)^*\phi^\theta_{\sigma\omega}=C_\om^\theta \phi^\theta_{\omega}$, then
\begin{equation*}
\begin{split}
\lot &=  
 \lot \phi_{\sig\om}^\theta (v_{\sig\om}^\theta)=  \phi_{\sig\om}^\theta (\mcl_\om^\theta \vot) =  (  (\mathcal{L}^{\theta}_\omega)^*\phi_{\sig\om}^\theta)  (\vot)=    C_\om^\theta \phi_{\om}^\theta (\vot) = C_\om^\theta.
\end{split}
\end{equation*}

\section{Limit theorems}
In this section we establish the main results of our paper.
To obtain the large deviation principle (Theorem~\ref{thm:ldt}), we first
 link the asymptotic behaviour of moment generating (and characteristic) functions associated to Birkhoff sums with the Lyapunov exponents $\Lambda (\theta)$. Then, we combine the strict convexity of the map $\theta \mapsto \Lambda (\theta)$ on a neighborhood of $0\in \mathbb R$ with the classical G\"artner-Ellis theorem.
We establish the central limit theorem (Theorem~\ref{thm:clt}) by applying Levy's continuity theorem and using the  $C^2$-regularity
of the map $\theta \mapsto \Lambda (\theta)$ on a neighborhood of $0\in \mathbb C$.
Finally, we demonstrate the full power of our approach by proving for the first time random versions of the local central limit theorem,  both under the so-called aperiodic and periodic assumptions (Theorems~\ref{thm:lclt} and~\ref{thm:lcltp}).
In addition, we present several equivalent formulations of the aperiodicity condition.
\subsection{Large deviations property}
In this section we establish Theorem~\ref{thm:ldt}.
%
The main tool in establishing this large deviations property will be the following classical result.
\begin{thm}\label{GET}(G\"artner-Ellis \cite{HennionHerve})
 For  $n\in \N$, let $\mathbb P_n$ be a probability measure on a measurable space $(Y, \mathcal T)$ and let $\mathbb E_n$ denote the corresponding expectation operator. Furthermore,
 let $S_n$ be a real random variable on $(\Omega, \mathcal T)$ and assume that on some interval $[-\theta_{+}, \theta_{+}]$, $\theta_{+}>0$, we have
 \begin{equation}\label{GETE}
  \lim_{n\to \infty} \frac 1n \log \mathbb E_n (e^{\theta S_n})=\psi (\theta),
 \end{equation}
where $\psi$ is a strictly convex continuously differentiable function satisfying $\psi'(0)=0$. Then, there exists $\epsilon_+>0$
such that the function $c$ defined by
\begin{equation}\label{GETE2}
 c(\epsilon)=\sup_{\lvert \theta \rvert \le \theta_+ }\{ \theta \epsilon-\psi (\theta) \}
\end{equation}
is
nonnegative, continuous, strictly convex  on $[-\epsilon_+, \epsilon_+]$,
vanishing only at $0$ and  such that
\[
 \lim_{n\to \infty} \frac 1 n \log \mathbb P_n (S_n >n\epsilon)=-c(\epsilon), \quad \text{for every $\epsilon \in (0, \epsilon_+)$.}
\]
\end{thm}
We will also need  the following results, linking the asymptotic behaviour of characteristic functions associated to Birkhoff sums with the numbers $\Lambda (\theta)$.
\begin{lemma}\label{L:growthExpSums}
Let $\theta\in \C$ be sufficiently close to 0, so that the results of Section~\ref{sec:choiceOsBases} apply.
Let $f\in \BV$ be such that $f\notin H_\om^\theta$. That is, $\phi^\theta_\omega (f) \neq 0$.
Then,
\[
\lim_{n\to\infty}\frac{1}{n} \log \Big| \int e^{\theta S_ng(\omega,x)}f\ dm \Big| =  \Lam(\theta) \quad \text{for \paeom.}
\]
\end{lemma}

\begin{proof}
Given $f\in \B$, we may write (see~\eqref{sod}) $f=\phi^\theta_\omega (f) v^\theta_\omega+h^\theta_\omega$, where $h^\theta_\omega\in H^\theta_\omega$.
Using this decomposition and applying repeatedly \eqref{eq:def_lambdas}, we get
\begin{equation}
\label{decomp}\mathcal{L}^{\theta,(n)}_\omega f=\left(\prod_{i=0}^{n-1}\lambda_{\sigma^i\omega}^\theta\right) \phi^\theta_\omega (f)  v^\theta_{\sigma^{n-1}\omega}+ \mathcal{L}^{\theta,(n)}_\omega h_\omega^\theta.
\end{equation}
Theorem~\ref{thm:MET} ensures that
\begin{equation}
\label{decay}
\lim_{n\to\infty}\frac{1}{n}\log\|\mathcal{L}^{\theta,(n)}_{\omega}|_{H_\omega^\theta}\|<\Lambda(\theta).
\end{equation}
Thus, the second term in (\ref{decomp}) grows asymptotically with $n$ at an exponential rate strictly slower than $\Lambda(\theta)$.
By~\eqref{intprop} and \eqref{decomp}, we have that for \paeom
\begin{align*}
\lim_{n\to\infty}\frac{1}{n}& \log \Big| \int e^{\theta S_ng(\omega,x)}f\ dm \Big| = \lim_{n\to\infty}\frac{1}{n}\log \Big| \int \mathcal{L}^{\theta,(n)}_\omega f\ dm \Big| \\
&=\lim_{n\to\infty}\frac{1}{n}\sum_{i=0}^{n-1}\log |\lambda_{\sigma^i\omega}^\theta|+
\lim_{n\to\infty}\frac{1}{n}\log
\Big| \int\left[ \phi^\theta_\omega (f) v^\theta_{\sigma^{n-1}\omega}+\frac{\mathcal{L}^{\theta,(n)}_\omega h_\omega^\theta}{\prod_{i=0}^{n-1}|\lambda_{\sigma^i\omega}^\theta|}\right]\ dm \Big|,
\end{align*}
whenever the RHS limits exist.
 The first limit in the previous line equals $\Lambda(\theta)$ by \eqref{eq:LambdaFromlambda}.
The second limit is zero, because the choice of $v^\theta_{\sigma^{n-1}\omega}$ ensures the integral of the first term in the square brackets is $\phi^\theta_\omega (f) \neq 0$ (by assumption), which is independent of $n$, and
the second term in the square brackets goes to zero as $n \to \infty$ by (\ref{decay}).
The conclusion follows.
\end{proof}

\begin{lemma}\label{need}
For all complex $\theta$ in a neighborhood of 0, and \paeom, we have that
\[
 \lim_{n\to \infty} \frac 1 n \log  \Big|\int e^{\theta S_n g(\om, x)} \, d\mu_\om(x) \Big|=\Lambda (\theta).
\]
\end{lemma}

\begin{proof}
Since
\[
\lim_{n\to \infty} \frac 1 n \log  \Big| \int e^{\theta S_n g(\om, x)} \, d\mu_\om(x) \Big|=\lim_{n\to \infty} \frac 1 n \log  \Big|\int e^{\theta S_n(\om, x)}v_\om^0(x) \, dm(x) \Big|,
\]
by Lemma~\ref{L:growthExpSums} it is sufficient to show that $ \phi_\om^\theta (v_\om^0) \neq 0$ for $\theta$ near 0.
We know that $\phi_\om^0 (v_\om^0)=\int v_\om^0 dm=1$. Hence, the differentiability of $\theta \mapsto \phi^\theta$ at $\theta=0$, established in Appendix~\ref{0245}, together with the uniform bound on $\|v_\om^0\|_{\B}$ provided by \eqref{eq:boundedv}, ensure that for  $\theta \in \C$ sufficiently close to 0 and \paeom, $\phi_\om^\theta (v_\om^0)\neq 0$ as required.
\end{proof}

\begin{proof}[Proof of Theorem~\ref{thm:ldt}]
The proof follows directly from Theorem~\ref{GET} when applied to the case when
\[
 (Y, \mathcal T)=(X, \mathcal B), \quad \mathbb P_n=\mu_\om  \quad S_n=S_ng(\om, \cdot) \quad \text{and} \quad \psi(\theta)=\Lam(\theta).
\]
Indeed, we  note that~\eqref{GETE} holds by Lemma~\ref{need} (the absolute values are irrelevant when $\theta \in \R$). Furthermore, it follows from Corollary~\ref{cor:LamHatLam} that $\Lambda$ is continuously differentiable on a neighborhood
of $0$ in $\R$ satisfying $\Lambda'(0)=0$ and  by Corollary~\ref{convex}, we have that $\Lambda$ is strictly convex on a neighborhood
of $0$ in $\R$. Finally, $c$ does not depend on $\om$ by~\eqref{GETE2}.
\end{proof}

\subsection{Central limit theorem}\label{sec:pfclt}
The goal of section is to establish Theorem~\ref{thm:clt}.
We start with the following lemma, which will be useful in the proofs of the both central limit theorem and local central limit theorem.
\begin{lemma}\label{lem:UnifExpDecayY2}
There exist $C>0, 0<r<1$ such that for every $\theta \in \mathbb{C}$ sufficiently close to 0, every $n\in \N$ and $\paeom$, we have
\begin{equation}
\Big| \int\mathcal L_\om^{\theta, (n)}(v_\om^0 -\phi_\om^{\theta}(v_\om^0) v_{\om}^{\theta})\, dm \Big|
\leq Cr^n .
\end{equation}
\end{lemma}
\begin{proof}
The following argument generalises \cite[Lemma III.9]{HennionHerve} to the random setting.
For each $\theta$ near 0 and $\om \in \Om$, let
\[
\mcn_\om^{\theta} f :=  \mathcal L_\om^{\theta}(f - \phi_\om^\theta (f) v_{\om}^{\theta}).
\]
Note that, in view of  Lemma~\ref{1214} and  differentiability of $\theta\mapsto v^\theta$ and $\theta\mapsto \phi^\theta$
(established in Lemma~\ref{thm:IFT} (see \eqref{eq:vomt}) and Appendix~\ref{0245}, respectively),
 we get that there exists $N>1$ such that $\|\mcn_\om^{\theta}\|<N$ for every $\om \in \Om$, provided $\theta$ is sufficiently close to 0.

In addition, since $f - \phi_\om^\theta (f) v_{\om}^{\theta} $ is the projection of $f$ onto $H_\om^{\theta}$ along the top Oseledets space $Y^{\theta}_\om$, we get that, for every $n\geq 1$,
\[
\mcn_\om^{\theta, (n)} f =  \mathcal L_\om^{\theta,(n)}(f - \phi_\om^\theta (f) v_{\om}^{\theta}).
\]
Furthermore, since $f - \phi_\om^0 (f) v_{\om}^{0}= f -(\int f dm) v^0_\om$, condition \ref{cond:dec} and Lemma~\ref{lem:boundedv}\eqref{it:boundedv} ensure that there exist $K', \lam>0$ such that for every $n\geq 0$ and \paeom,
 $\|\mcn _\om^{0, (n)}\| \leq K'e^{-\lam n}$.

 Let $1>r>e^{-\lam}$, and let $n_0\in \N$ be such that $K'e^{-\lam n_0}< r^{n_0}$.
Lemma~\ref{1214} together with differentiability of $\theta\mapsto v^\theta$ and $\theta\mapsto \phi^\theta$ ensure that $\theta \mapsto \mcn_\om^\theta$ is continuous in the norm topology of $\BV$. In fact, the uniform control over $\om\in \Om$, guaranteed by the aforementioned differentiability conditions, along with
 Condition~\ref{cond:unifNormBd}, ensure that  one can choose $\ep>0$ so that if $|\theta|<\ep$, then
 $\| \mcn_\omega^{\theta,(n_0)}\|<r^{n_0}$ for every $\om \in \Om$.
Writing $n=kn_0+\ell$, with $0\leq \ell<n_0$, we get
\begin{eqnarray*}
\| \mcn_\omega^{\theta,(n)}\|&\le &
\prod_{j=0}^{k-1} \| \mcn_{\sig^{jn_0}\omega}^{\theta,(n_0)}\| (\| \mcn_{\sig^{kn_0}\omega}^{\theta,(\ell)}\|)
<  r^n\left(N/r \right)^\ell\le cr^n,
\end{eqnarray*}
with $c=\big(\frac{N}{r}\big)^{n_0}$.
 Thus,
 \begin{align*}
\Big| \int \mathcal L_\om^{\theta, (n)}(v_\om^0 &-\phi_\om^{\theta}(v_\om^0) v_{\om}^{\theta})\, dm \Big|
\leq \| \mathcal L_\om^{\theta, (n)}(v_\om^0 -\phi_\om^{\theta}(v_\om^0) v_{\om}^{\theta}) \|_{1}\\
&\leq  \| \mathcal L_\om^{\theta, (n)}(v_\om^0 -\phi_\om^{\theta}(v_\om^0) v_{\om}^{\theta}) \|_{\mathcal B}
= \|\mcn_\om^{\theta, (n)} (v_\om^0) \|_{\mathcal B}
\leq cr^n \|v_\om^0\|_{\mathcal B}.
 \end{align*}
By~\eqref{eq:boundedv}, there exists $\tilde K>0$ such that   $\|v_\om^0\|_{\mathcal B}\leq \tilde{K}$ for \paeom, so the proof of the lemma is complete.
 \end{proof}

\begin{proof}[Proof of Theorem~\ref{thm:clt}]
We recall that $\Sig^2>0$ is given by~\eqref{variance}.
It follows from Levy's continuity theorem that it is sufficient to prove that, for every $t\in \R$,
\[
 \lim_{n\to \infty}\int e^{it \frac{S_n g(\om, \cdot)}{\sqrt n}} \, d\mu_\om=e^{-\frac{t^2 \Sig^2}{2}}, \quad \text{for \paeom.}
\]
Assume $n$ is sufficiently large so that $\dim Y_1^{\frac{it}{\sqrt{n}}}=1$ and $v_\om^{\frac{it}{\sqrt{n}}}$ can be chosen as in \eqref{eq:vomt}. In particular, $\int_0^1 v_\om^{\frac{it}{\sqrt{n}}} dm=1$ and $\mathcal L_{\om}^{\frac{it}{\sqrt{n}}, (n)}v_\om^{\frac{it}{\sqrt{n}}}= (\prod_{j=0}^{n-1} \lam_{\sig^j\om}^{\frac{it}{\sqrt{n}}}) v_{\sig^{n}\om}^{\frac{it}{\sqrt{n}}}$, for \paeom. Furthermore,
using~\eqref{intprop},
\begin{align*}
 \int e^{it \frac{S_n g(\om, \cdot)}{\sqrt n}} \, d\mu_\om &= \int e^{it \frac{S_n g(\om, \cdot)}{\sqrt n}}v_\om^0 \, dm=\int \mathcal L_\om^{\frac{it}{\sqrt n}, (n)} v_\om^0 \, dm\\
 &=\int \mathcal L_\om^{\frac{it}{\sqrt n}, (n)}
 \Big(  \phi_\om^{\frac{it}{\sqrt n}}( v_\om^0 ) v_{\om}^{\frac{it}{\sqrt n}} +
 (v_\om^0 -  \phi_\om^{\frac{it}{\sqrt n}}( v_\om^0 ) v_{\om}^{\frac{it}{\sqrt n}}) \Big)\, dm\\
 & = \phi_\om^{\frac{it}{\sqrt n}}( v_\om^0 ) \cdot \prod_{j=0}^{n-1} \lambda_{\sigma^j \om}^{\frac{it}{\sqrt n}} + \int \mathcal L_\om^{\frac{it}{\sqrt n}, (n)}(v_\om^0 -  \phi_\om^{\frac{it}{\sqrt n}}( v_\om^0 ) v_{\om}^{\frac{it}{\sqrt n}})\, dm.
\end{align*}
 Lemma~\ref{lem:UnifExpDecayY2} shows that the second term converges to 0 as $n\to \infty$.
Also, differentiability of $\theta \mapsto \phi^\theta$, established in Appendix~\ref{0245}, ensures that
$\lim_{n\to \infty}  \phi_\om^{\frac{it}{\sqrt n}}( v_\om^0 )=  \phi_\om^{0}( v_\om^0 )=1$.
Thus, to conclude the proof of the theorem, we need to prove that
\begin{equation}\label{QQQ}
\lim_{n\to \infty} \prod_{j=0}^{n-1} \lambda_{\sigma^j \om}^{\frac{it}{\sqrt n}} = e^{-\frac{t^2 \Sig^2}{2}}, \quad \text{for \paeom,}
\end{equation}
which is equivalent to
\[
\lim_{n \to \infty} \sum_{j=0}^{n-1} \log  \lambda_{\sigma^j \om}^{\frac{it}{\sqrt n}} = -\frac{t^2 \Sig^2}{2}, \quad \text{for \paeom.}
\]
Using the notation of Lemmas~\ref{lem:FwellDef} and~\ref{thm:IFT}, we have that $\lambda_\om^\theta=H(\theta, O(\theta))(\sigma \om)$ and thus we need to prove that
\begin{equation}\label{taylor}
\lim_{n\to \infty} \sum_{j=0}^{n-1} \log H \bigg{(}\frac{it}{\sqrt n}, O(\frac{it}{\sqrt n})\bigg{)}(\sigma^{ j+1} \om) = -\frac{t^2 \Sig^2}{2} \quad \text{for \paeom.}
\end{equation}
Let $\tilde{H}$ be a map defined in a neighborhood of $0$ in $\mathbb C$ with values in $L^\infty (\Om)$ by $\tilde{H}(\theta)= \log H(\theta, O(\theta))$.
It will be shown in Lemma~\ref{TC2} that $\tilde{H}$ is of class $C^2$, $\tilde{H}(0)(\om)=0$, $\tilde{H}'(0)(\om)=0$ and
\[
\tilde{H}''(0)(\om)=\int (g(\sigma^{-1} \om, \cdot )^2v_{\sigma^{-1} \om}^0+2g(\sigma^{-1} \om, \cdot)O'(0)_{\sigma^{-1} \om})\, dm.
\]
Developing $\tilde{H}$  in a Taylor series around $0$, we have that
\[
\tilde{H}(\theta)(\om)= \log H(\theta, O(\theta))(\om) =\frac{\tilde{H}''(0)(\om)}{2} \theta^2+ R(\theta)(\om),
\]
where $R$ denotes the remainder. Therefore,
\[
 \log H \bigg{(}\frac{it}{\sqrt n}, O(\frac{it}{\sqrt n})\bigg{)}(\sigma^{ j+1} \om)=-\frac{t^2 \tilde{H}''(0)(\sigma^{j+1} \om)}{2n}+R(it/\sqrt n)(\sigma^{j+1} \om),
\]
which implies that
\begin{equation}\label{ii}
 \sum_{j=0}^{n-1} \log H \bigg{(}\frac{it}{\sqrt n}, O(\frac{it}{\sqrt n})\bigg{)}(\sigma^{ j+1} \om)=-\frac{t^2 }{2}\cdot \frac 1 n\sum_{j=0}^{n-1}\tilde{H}''(0)(\sigma^{j+1} \om) +\sum_{j=0}^{n-1}R(it/\sqrt n)(\sigma^{j+1} \om).
 \end{equation}
The asymptotic behaviour of the first term is governed by Birkhoff's ergodic theorem,
so using \eqref{0521} in the second equality and Lemma~\ref{lem:Lam''0} in the third one, we get:
\begin{equation}
\label{Lambda2prime}
\lim_{n\to\infty} -\frac{t^2 }{2}\frac 1 n\sum_{j=0}^{n-1}\tilde{H}''(0)(\sigma^{j+1} \om ) = -\frac{t^2 }{2}\int \tilde{H}''(0)(\om) \, d\mathbb P(\om)=-\frac{t^2 }{2}\Lambda''(0)=-\frac{t^2 }{2}\Sig^2 \quad \text{for \paeom.}
\end{equation}
Now we deal with the last term of \eqref{ii}. Writing $R(\theta)=\theta^2 \tilde R(\theta)$ with $\lim_{\theta \to 0} \tilde R(\theta)=0$, we conclude that for each $\epsilon >0$ and $t\in \R\setminus \{0\}$, there exists $\delta >0$ such that $\lVert \tilde R(\theta)\rVert_{L^\infty} \le
\frac{\epsilon}{ t^2}$
for all $\lvert \theta\rvert \le \delta$. We note that  there exists $n_0\in \mathbb N$ such that $\lvert it/\sqrt n\rvert \le \delta$ for each $n\ge n_0$. Hence,
\[
\bigg{\lvert}  \sum_{j=0}^{n-1}R(it/\sqrt n)(\sigma^{j+1} \om) \bigg{\rvert} \le \frac{ t^2}{n}\sum_{j=0}^{n-1}\lvert \tilde R(it/\sqrt n)(\sigma^{j+1} \om)\rvert \le \frac{ t^2}{n}\cdot \frac{n\epsilon}{ t^2}
=\epsilon,\]
for every $n\ge n_0$, which implies that
 the second term on the right-hand side of~\eqref{ii} converges to $0$ and
 thus~\eqref{taylor} holds. The proof of the theorem is complete.
\end{proof}

\begin{lemma}\label{TC2}
 The map $\tilde{H}(\theta)= \log H(\theta, O(\theta))$ is of class $C^2$. Moreover,
 $\tilde{H}(0)(\om)=0$, $\tilde{H}'(0)(\om)=0$ and
\[
\tilde{H}''(0)(\om)=\int (g(\sigma^{-1} \om, \cdot )^2v_{\sigma^{-1} \om}^0+2g(\sigma^{-1} \om, \cdot)O'(0)_{\sigma^{-1} \om})\, dm.
\]
\end{lemma}

\begin{proof}
The regularity of $\tilde H$ follows directly from the results in Appendices~\ref{regofF} and~\ref{regofF2}. Moreover,
we have $\tilde{H}(0)(\om)=\log H(0, O(0))(\om)=\log 1=0$. Furthermore,
\[
 \tilde{H}'(\theta)(\om)=\frac{1}{H(\theta, O(\theta))(\om)}[D_1H(\theta, O(\theta))(\om)+(D_2H(\theta, O(\theta))O'(\theta))(\om)].
\]
Taking into account formulas in Appendix~\ref{regofF}, \eqref{zeromean} and~\eqref{zerod}, we have
\[
 \tilde{H}'(0)(\om)=\int g(\sigma^{-1} \om, \cdot) v_{\sigma^{-1} \om}^0 \, dm +\int O'(0)_{\sigma^{-1} \om}\, dm=0.
\]
Finally, taking  into account that $D_{22}H=0$ (see Appendix~\ref{regofF2}) we have
\[
 \begin{split}
  \tilde{H}''(\theta)(\om) &=\frac{-D_1 H(\theta, O(\theta))(\om)}{[H(\theta, O(\theta))(\om)]^2}[D_1H(\theta, O(\theta))(\om)+(D_2H(\theta, O(\theta))O'(\theta))(\om)] \\
  &\phantom{=}+\frac{1}{H(\theta, O(\theta))(\om)}[D_{11}H(\theta, O(\theta))(\om)+(D_{21}H(\theta, O(\theta))O'(\theta))(\om)] \\
  &\phantom{=}+\frac{1}{H(\theta, O(\theta))(\om)}[(D_{12}H(\theta, O(\theta))O'(\theta))(\om)+(D_2H(\theta, O(\theta))O''(\theta))(\om)].
 \end{split}
\]
Using formulas in Appendices~\ref{regofF} and~\ref{regofF2}, we obtain the desired expression for $\tilde{H}''(0)$.
\end{proof}

\subsection{Local central limit theorem}
In order to obtain a local central limit theorem, we introduce an additional assumption related to aperiodicity, as follows.
  \begin{enumerate}[label=(C\arabic*), series=conditions, start=5]
 \item \label{cond:aper}
 For \paeom \ and for every compact interval $J\subset \R \setminus\{0\}$ there exist $C=C(\om)>0$ and
$\rho \in (0,1)$ such that
\begin{equation}\label{aper}
\lVert \mathcal L_{\om}^{it, (n)}\rVert_{\B} \le C\rho^n, \quad \text{for $t\in J$ and $n\ge 0$.}
\end{equation}
\end{enumerate}

The proof of Theorem~\ref{thm:lclt} is presented in Section~\ref{sec:pfLCLT}.
In Section~\ref{sec:aper}, we show that \ref{cond:aper}
  can be phrased as a so-called aperiodicity condition, resembling a usual requirement for autonomous versions of the local CLT.
Examples are presented in Section~\ref{sec:exlclt}.

\subsubsection{Proof of Theorem~\ref{thm:lclt}}\label{sec:pfLCLT}
 Using the density argument (see~\cite{Morita}), it is sufficient to show that
 \begin{equation}
 \label{66a}
 \sup_{s\in \R} \bigg{\lvert} \Sig \sqrt{n}\int h(s+S_ng(\om, \cdot))\, d\mu_\om-\frac{1}{\sqrt{2\pi}}e^{-\frac{s^2}{2n\Sig^2}}\int_{\R} h(u)\, du\bigg{\rvert} \to 0,
 \end{equation}
when $n\to \infty$ for every $h\in L^1(\R)$ whose Fourier transform $\hat{h}$ has compact support.
Moreover, we recall the following inversion formula
\begin{equation}\label{invf}
 h(x)=\frac{1}{2\pi}\int_{\R} \hat h(t)e^{itx}\, dt.
\end{equation}
By~\eqref{intprop}, \eqref{invf} and Fubini's theorem,
\[
 \begin{split}
  \Sig \sqrt{n}\int h(s+S_ng(\om,\cdot))\, d\mu_\om &=\frac{\Sig \sqrt{n}}{2\pi} \int \int_{\R}\hat h(t)e^{it(s+S_n g(\om, \cdot))}\, dt \, d\mu_\om \\
  &=\frac{\Sig \sqrt{n}}{2\pi} \int_{\R} e^{its}\hat h(t)\int e^{it S_n g(\om, \cdot)}\, d\mu_\om \, dt \\
  &=\frac{\Sig \sqrt{n}}{2\pi} \int_{\R} e^{its}\hat h(t)\int e^{it S_n g(\om, \cdot)}v_\om^0\, dm \, dt \\
  &=\frac{\Sig \sqrt{n}}{2\pi} \int_{\R} e^{its}\hat h(t)\int \mathcal L_{\om}^{it, (n)}v_\om^0\, dm \, dt \\
  &=\frac{\Sigma}{2\pi}\int_{\R} e^{\frac{its}{\sqrt n}}\hat h(\frac{t}{\sqrt n})\int \mathcal L_{\om}^{\frac{it}{\sqrt n}, (n)}v_\om^0\, dm \, dt.
 \end{split}
\]
Recalling that the Fourier transform of  $f(x)=e^{-\frac{\Sig^2x^2}{2}}$ is given by $\hat f(t)=\frac{\sqrt{2\pi}}{\Sigma}e^{- t^2/2\Sig^2}$ we have
\begin{eqnarray*}
\frac{1}{\sqrt{2\pi}}e^{-\frac{s^2}{2n\Sig^2}}\int_{\R} h(u)\, du&=&\frac{\hat h(0)}{\sqrt{2\pi}}e^{-\frac{s^2}{2n\Sig^2}}\\
&=&\frac{\hat h(0)\Sigma}{2\pi} \hat f(-s/\sqrt n) \\
&=&\frac{\hat h(0)\Sigma}{2\pi} \int_{\R} e^{\frac{its}{\sqrt n}} \cdot e^{-\frac{\Sig^2 t^2}{2}}\, dt.
 \end{eqnarray*}
Hence, we need to prove that
\begin{equation}\label{CLTL1}
 \sup_{s\in \R}\bigg{\rvert} \frac{\Sigma}{2\pi}\int_{\R} e^{\frac{its}{\sqrt n}}\hat h(\frac{t}{\sqrt n})\int \mathcal L_{\om}^{\frac{it}{\sqrt n}, (n)}v_\om^0\, dm \, dt -
 \frac{\hat h(0)\Sigma}{2\pi} \int_{\R} e^{\frac{its}{\sqrt n}} \cdot e^{-\frac{\Sig^2 t^2}{2}}\, dt \bigg{\rvert} \to 0,
\end{equation}
when $n\to \infty$, for \paeom.
Choose $\delta >0$ such that the support of $\hat h$ is contained in $[-\delta, \delta]$.
Recall that  $ \mathcal L_{\om}^{\theta, (n)}v_\om^{\theta}= (\prod_{j=0}^{n-1} \lam_{\sig^j\om}^{\theta}) v_{\sig^{n}\om}^{\theta}$ for \paeom, and for all $\theta$ near 0.
Then, for any $\tilde \delta \in (0, \delta)$, we have,
\begin{align*}
 &\frac{\Sigma}{2\pi}\int_{\R} e^{\frac{its}{\sqrt n}}\hat h(\frac{t}{\sqrt n})\int \mathcal L_{\om}^{\frac{it}{\sqrt n}, (n)}v_\om^0\, dm \, dt -
 \frac{\hat h(0)\Sigma}{2\pi} \int_{\R} e^{\frac{its}{\sqrt n}} \cdot e^{-\frac{\Sig^2 t^2}{2}}\, dt \displaybreak[0] \\
 &=\frac{\Sigma}{2\pi} \int_{\lvert t\rvert < \tilde \delta \sqrt n} e^{\frac{its}{\sqrt n}} \Big(\hat h(\frac{t}{\sqrt n})\prod_{j=0}^{n-1}\lambda_{\sigma^j \om}^{\frac{it}{\sqrt n}}-\hat h(0)e^{-\frac{\Sig^2 t^2}{2}} \Big)\, dt \displaybreak[0] \\
&\phantom{=}+\frac{\Sigma}{2\pi}\int_{\lvert t\rvert < \tilde \delta \sqrt n}e^{\frac{its}{\sqrt n}}\hat h(\frac{t}{\sqrt n})
\int
\prod_{j=0}^{n-1}\lambda_{\sigma^j \om}^{\frac{it}{\sqrt n}} \Big( \phi_\om^{\frac{it}{\sqrt n}}( v_\om^0 ) v_{\sig^{n}\om}^{\frac{it}{\sqrt n}}-1 \Big)\,dm \, dt \displaybreak[0] \\
&\phantom{=}+\frac{\Sig \sqrt{n}}{2\pi}\int_{\lvert t\rvert <\tilde \delta}e^{its}\hat h(t)\int  \mathcal L_{\om}^{it, (n)} (v_\om^0 - \phi_\om^{it}( v_\om^0 ) v_{\om}^{it}) \, dm\, dt \displaybreak[0] \\
&\phantom{=}+\frac{\Sig \sqrt{n}}{2\pi}\int_{\tilde \delta \le \lvert t\rvert < \delta}e^{its}\hat h(t)\int \mathcal L_\om^{it, (n)}v_\om^0\, dm\, dt \displaybreak[0] \\
&\phantom{=}-\frac{\Sigma}{2\pi}\hat h(0) \int_{\lvert t\rvert \ge \tilde \delta \sqrt n}e^{\frac{its}{\sqrt n}} \cdot e^{-\frac{\Sig^2 t^2}{2}}\, dt=: (I)+(II)+(III)+(IV)+(V).
\end{align*}

The proof of the theorem will be complete once we
 show that each of the terms $(I)$--$(V)$ converges to zero as $n\to \infty$.

\paragraph{Control of (I).}
We claim that for \paeom,
\[
\lim_{n\to \infty}
\sup_{s\in \R} \bigg{\lvert}\int_{\lvert t\rvert < \tilde \delta \sqrt n}e^{\frac{its}{\sqrt n}}
 (\hat h(\frac{t}{\sqrt n})\prod_{j=0}^{n-1} \lambda_{\sigma^j \om}^{\frac{it}{\sqrt n}}-\hat h(0)e^{-\frac{\Sig^2 t^2}{2}})\, dt \bigg{\rvert} = 0.
 \]

 Indeed, it is clear that
 \[
  \sup_{s\in \R} \bigg{\lvert}\int_{\lvert t\rvert < \tilde \delta \sqrt n}e^{\frac{its}{\sqrt n}}
 (\hat h(\frac{t}{\sqrt n})\prod_{j=0}^{n-1} \lambda_{\sigma^j \om}^{\frac{it}{\sqrt n}}-\hat h(0)e^{-\frac{\Sig^2 t^2}{2}})\, dt \bigg{\rvert}
 \le \int_{\lvert t\rvert < \tilde \delta \sqrt n}\bigg{\lvert} \hat h(\frac{t}{\sqrt n})\prod_{j=0}^{n-1} \lambda_{\sigma^j \om}^{\frac{it}{\sqrt n}}-\hat h(0)e^{-\frac{\Sig^2 t^2}{2}}
 \bigg{\lvert} \, dt.
 \]
It follows from the continuity of $\hat h$ and~\eqref{QQQ} that for \paeom \ and every $t$,
\begin{equation}
\label{intermediate}
 \hat h(\frac{t}{\sqrt n})\prod_{j=0}^{n-1} \lambda_{\sigma^j \om}^{\frac{it}{\sqrt n}}-\hat h(0)e^{-\frac{\Sig^2 t^2}{2}}\to 0, \quad \text{when $n\to \infty$.}
\end{equation}
The desired conclusion will  follow from the dominated convergence theorem once we establish the following lemma.

\begin{lemma}\label{lem:boundProdLam}
 For $\tilde \delta >0$ sufficiently small, there exists $n_0\in \N$ such that for all $n\ge n_0$ and $t$ such that $\lvert t\rvert < \tilde \delta \sqrt n$,
 \[
  \bigg{\lvert} \prod_{j=0}^{n-1} \lambda_{\sigma^j \om}^{\frac{it}{\sqrt n}} \bigg{\rvert} \le e^{-\frac{t^2 \Sigma^2}{8}}.
 \]

\end{lemma}
\begin{proof}
 We use the same notation as in the proof of Lemma~\ref{TC2}. As before, $\Re (z)$ denotes the real part of a complex number $z$.  We note that
 \[
  \bigg{\lvert} \prod_{j=0}^{n-1} \lambda_{\sigma^j \om}^{\frac{it}{\sqrt n}} \bigg{\rvert}=e^{-\frac{t^2}{2} \Re (\frac 1n \sum_{j=0}^{n-1} \tilde{H}''(0)(\sigma^j \om))}\cdot
  e^{\Re (\sum_{j=0}^{n-1} R(it/\sqrt n)(\sigma^j \om))}.
 \]
Since, by (\ref{Lambda2prime}), $\frac{1}{n} \sum_{j=0}^{n-1} \tilde{H}''(0)(\sigma^j \om) \to \Sigma^2$ for \paeom, we also have that \[ \Re \bigg{(}\frac 1n \sum_{j=0}^{n-1} \tilde{H}''(0)(\sigma^j \om) \bigg{)}\to \Sigma^2, \quad \paeom \]
and therefore for \paeom \ there exists $n_0=n_0(\om) \in \N$ such that
\[
  \Re \bigg{(}\frac 1n \sum_{j=0}^{n-1} \tilde{H}''(0)(\sigma^j \om) \bigg{)} \ge \Sigma^2/2, \quad \text{for $n\ge n_0$.}
\]
Hence,
\[
 e^{-\frac{t^2}{2} \Re (\frac 1n \sum_{j=0}^{n-1} \tilde{H}''(0)(\sigma^j \om))} \le e^{-\frac{t^2 \Sigma^2}{4}}, \quad \text{for $n\ge n_0$ and every $t\in \R$.}
\]
We now choose $\tilde \delta$ such that $\lVert \tilde R(\theta)\rVert_{L^\infty} \le \Sigma^2/8$ whenever $\lvert \theta \rvert \le \tilde \delta$. Hence, for $t$ such that $\lvert t\rvert< \tilde \delta \sqrt n$
, we have
\[
 \bigg{\lvert} \sum_{j=0}^{n-1} R(it/\sqrt n)(\sigma^j \om)) \bigg{\rvert} \le \frac{t^2}{n}\sum_{j=0}^{n-1} \lvert \tilde R(it/\sqrt n)(\sigma^j \om)\rvert \le \frac{t^2 \Sigma^2}{8}
\]
and therefore
\[
 e^{\Re (\sum_{j=0}^{n-1} R(it/\sqrt n)(\sigma^j \om))}\le e^{-\frac{t^2 \Sigma^2}{8} },
\]
which implies the statement of the lemma.
\end{proof}

\paragraph{Control of (II).}
We recall that for $\theta$ sufficiently close to 0,  $v^\theta_\om$ as defined in \eqref{eq:vomt} satisfies  $\int_0^1 v^{\theta}_{\om}\,dm=1$ for \paeom.
Thus, to control (II) we must show that for \paeom
\begin{equation}\label{eq:controlII}
\lim_{n\to \infty} \sup_{s\in \R} \bigg{\lvert} \frac{\Sigma}{2\pi}\int_{\lvert t\rvert < \tilde \delta \sqrt n}e^{\frac{its}{\sqrt n}}\hat h(\frac{t}{\sqrt n})\prod_{j=0}^{n-1}\lambda_{\sigma^j \om}^{\frac{it}{\sqrt n}}( \phi_\om^{\frac{it}{\sqrt n}}( v_\om^0 ) -1)\, dt
\bigg{\rvert} = 0.
 \end{equation}

Using the fact that $ \phi_\om^0( v_\om^0) =1$ and the differentiability of $\theta \mapsto \phi^\theta$ (see Appendix~\ref{0245}), we conclude that there exists $C>0$ such that $\lvert\phi_\om^{\theta}(v_\om^0) -1\rvert \le C\lvert \theta \rvert$ for $\theta$ in a
neighborhood of $0$ in $\mathbb C$. Taking into account Lemma~\ref{lem:boundProdLam}, we conclude that
\[
\sup_{s\in \R} \bigg{\lvert} \frac{\Sigma}{2\pi}\int_{\lvert t\rvert < \tilde \delta \sqrt n}e^{\frac{its}{\sqrt n}}\hat h(\frac{t}{\sqrt n})\prod_{i=0}^{n-1}\lambda_{\sigma^i \om}^{\frac{it}{\sqrt n}}( \phi_\om^{\frac{it}{\sqrt n}}( v_\om^0 ) -1)\, dt
\bigg{\rvert} \le \frac{1}{\sqrt n}C\frac{\Sigma}{2\pi}\lVert \hat h\rVert_{L^\infty}\int_{\lvert t\rvert < \tilde \delta \sqrt n}\lvert t\rvert e^{-\frac{\Sig^2t^2}{8}}\, dt,
\]
which readily implies \eqref{eq:controlII}.

\paragraph{Control of (III).}\label{par:controlIII}
We must show that
\[
\lim_{n\to \infty}
\sup_{s\in \R} \bigg{\lvert}\frac{\Sig \sqrt{n}}{2\pi}\int_{\lvert t\rvert <\tilde \delta}e^{its}\hat h(t)\int_0^1 \mathcal L_\om^{it, (n)}(v_\om^0 - \phi_\om^{it}( v_\om^0 ) v_{\om}^{it})\, dm\,dt \bigg{\rvert} = 0.
\]
Lemma~\ref{lem:UnifExpDecayY2} shows that there exist $C>0$ and $0<r<1$ such that for every sufficiently small $t$, every $n\in \N$ and \paeom,
 \begin{align*}
\Big| \int \mathcal L_\om^{it, (n)}(v_\om^0 &- \phi_\om^{it}( v_\om^0 ) v_{\om}^{it})\, dm \Big|
\leq Cr^n.
 \end{align*}
 Hence, provided $\tilde{\delta}$ is sufficiently small,
 \[
\lim_{n\to \infty}
\sup_{s\in \R} \bigg{\lvert}\frac{\Sig \sqrt{n}}{2\pi}\int_{\lvert t\rvert <\tilde \delta}e^{its}\hat h(t)\int \mathcal L_\om^{it, (n)}(v_\om^0 - \phi_\om^{it}( v_\om^0 ) v_{\om}^{it})\, dm\,dt \bigg{\rvert}
\leq \lim_{n\to \infty} \frac{\Sig \sqrt{n}}{2\pi} \|\hat h\|_{L^\infty} Cr^n=0.
\]

\paragraph{Control of (IV).}
By the aperiodicity condition~\ref{cond:aper},
\[
 \sup_{s\in \R}\frac{\Sig \sqrt{n}}{2\pi} \bigg{\lvert}\int_{\tilde \delta \le \lvert t\rvert \le \delta} e^{its}\hat h(t)\int \mathcal L_{\om}^{it, (n)}v_\om^0\, dm \, dt
 \bigg{\rvert} \le 2C(\delta -\tilde \delta)\frac{\Sig \sqrt{n}}{2\pi}\lVert \hat h\rVert_{L^\infty} \cdot \rho^n \cdot \lVert v^0\rVert_\infty  \to 0,
\]
when $n\to \infty$ by \eqref{eq:boundedv} and the fact that $\hat h$ is continuous.

\paragraph{Control of (V).}
It follows from the dominated convergence theorem and the integrability of the map $t\mapsto e^{-\frac{\Sig^2t^2}{2}}$ that
\[
 \sup_{s\in \R} \bigg{\lvert}\frac{\hat h(0)\Sigma}{2\pi} \int_{\lvert t\rvert \ge \tilde \delta \sqrt n} e^{\frac{its}{\sqrt n}} \cdot e^{-\frac{\Sig^2 t^2}{2}}\, dt\bigg{\rvert}
 \le \frac{\lvert \hat h(0)\rvert \Sigma}{2\pi}\int_{\lvert t\rvert \ge \tilde \delta \sqrt n}e^{-\frac{\Sig^2 t^2}{2}}\, dt \to 0,
\]
when $n\to \infty$.
\qed

\subsubsection{Equivalent versions of the aperiodicity condition}\label{sec:aper}

In this subsection we show the following equivalence result.

\begin{lemma}\label{lem:AperiodicCoboundary}
Assume $\dim Y^0=1$ and condition~\ref{cond:METCond} holds. Suppose, in addition, that $\Om$ is compact and that the map $\mcl: \Om \to L(\B), \om \mapsto  \mcl_\om$, is continuous on each of finitely many pairwise disjoint open sets $\Om_1, \dots, \Om_q$ whose union is  $\Omega$, up to a set of $\bbp$ measure 0. Furthermore, assume that for each $1\leq j \leq q$, $\mcl: \Om_j \to L(\B)$ can be extended continuously to the closure $\bar \Om_j$.
 Then, each of the following conditions is equivalent to Condition~\ref{cond:aper}:
\begin{enumerate}
\item \label{it:cobi}
For every $t\in \R \setminus \{0\}$, $\Lam(it)< 0$.
\item \label{it:cobii}
For every $t\in \R$,
either (i) $\Lam(it)< 0$ or (ii) the cocycle $\mc{R}^{it}$ is quasicompact and
the equation
\begin{equation}\label{eq:coboundary}
e^{itg(\om,x)} \mcl^{0*}_\om \psi_{\sig \om} = \ga_\om^{it} \psi_\om,
\end{equation}
where $\ga_\om^{it}\in S^1$ and $\psi_\om \in \B^*$   only has a
measurable non-zero solution $\psi:=\{\psi_\om\}_{\om\in\Om}$ when $t=0$. Furthermore, in this case $\ga_\om^{0}=1$ and $\psi_\om(f)=\int f dm$ (up to a scalar multiplicative factor).
\end{enumerate}
\end{lemma}
Before proceeding with the proof, we present an auxiliary result for the cocycle $\mc{R}^{it}$.

\begin{lemma}\label{lem:simpleY1it}
Assume $\dim Y^0=1$ and $\mc{R}^{it}$ is quasi-compact for every $t\in \R$ for which $\Lam(it)=0$.
Then, for each $t\in \R$, either $\Lam(it)<0$ or $\dim Y^{it}=1$.
\end{lemma}
\begin{proof}
Assume $\dim Y^0=1$.
It follows from the definition of $\mcl_\om^{it}$ that $\Lam(it)\leq 0$ for every $t\in \R$.
Indeed, for every $v\in \B$, $\|\mcl_\om^{it} v\|_1= \|\mcl_\om (e^{itg(\om, \cdot)}v) \|_1 \leq \| e^{itg(\om, \cdot)}v \|_1=\|v\|_1$.
Hence,  $\lim_{n\to \infty} \frac1n \log  \| \mcl_\om^{it,(n)}v\|_{1}\leq 0$.
Lemma~\ref{lem:RandomSameExp} then implies that $\Lam(it)\leq 0$.

Suppose $\Lam(it)=0$ for some $t\in \R$. Let $d=\dim Y^{it}$. Then $d<\infty$ by the quasi-compactness assumption. Our proof proceeds in three steps:
\begin{enumerate}[label=(\arabic*)]
\item \label{it:L1Norm}
Let $S_1=\{x\in \BV : \|x\|_1=1 \}$. Then, for \paeom \  and every $v \in Y_\om^{it}\cap S_1$, $\|\mcl_\om^{it} v\|_1 =1$.
\item \label{it:vitNorm}
 Assume $v \in Y_\om^{it}$ is such that $\|v\|_1=1$. Then $|v|=v^0_\om$. In words, the magnitude of $v$ is given by $v^0_\om$, the generator of $Y_\om^{0}$.
\item  \label{it:vitAngle}
Assume $u , v  \in Y_\om^{it}$ are such that $\|v \|_1=\|u \|_1=1$. Then, there exists a constant $a \in \R$ such that $u =e^{ia} v $. In particular, $d=\dim Y^{it}=1$.
\end{enumerate}

The proof of step~\ref{it:L1Norm} involves some technical aspects of Lyapunov exponents and volume growth
and it is deferred until Appendix~\ref{sec:pfStep1}. Assuming this step has been established, we proceed to show the remaining two.
\paragraph{Proof of step \ref{it:vitNorm}.}
Let $v \in Y_\om^{it}$ be such that $\|v\|_1=1$. Consider the polar decomposition of $v$,
\[
v(x)= e^{i\phi(x)} r(x),
\]
where $\phi, r: \X \to \R$ are functions such that $r\geq 0$. Notice that the choice of $r(x)$ is unique. The choice of $\phi(x) \pmod{2\pi}$ is unique whenever $r(x)\neq 0$, and arbitrary otherwise.
Because of step \ref{it:L1Norm}, for \paeom \ and $n\in \N$, we have $\| \mcl_\om^{it,(n)} v\|_1=1$.
Also, $\|\mcl^{(n)}_\om |v| \|_1 = \| \mcl^{(n)}_\om r \|_1=1$, where we use $|v|$ to denote the magnitude (radial component) of $v$.
Notice that $\mcl_\om^{(n)} r(x)= \sum_{T^{(n)}_\om y=x} \frac{r(y) }{|(T^{(n)}_\om)'(y)|}$ and by Lemma~\ref{lem:exprLit}(1),
\begin{equation}\label{eq:polarLitv}
\mcl_\om^{it,(n)} v (x)= \sum_{T^{(n)}_\om y=x} e^{it S_ng(\om,y)+ i \phi(y)}\frac{ r(y) }{|(T^{(n)}_\om)'(y)|}.
\end{equation}
In particular, for each $x\in \X$, we have $|\mcl_\om^{it,(n)} v (x)| \leq \mcl^{(n)}_\om r(x)$.
Since $1=\| \mcl_\om^{it,(n)} v\|_1 = \int |\mcl_\om^{it,(n)} v (x)| dx$
and $1=\| \mcl^{(n)}_\om r\|_1 = \int \mcl^{(n)}_\om r(x) dx$, it must be that for a.e. $x\in \X$,
\begin{equation} \label{eq:NormLitv}
|\mcl_\om^{it,(n)} v (x)| = \mcl^{(n)}_\om r (x).
\end{equation}
In view of the triangle inequality, equality in \eqref{eq:polarLitv} holds
if and only if for a.e. $x\in \X$ such that $\mcl^{(n)}_\om r(x)\neq 0$, the phases coincide on all preimages of $x$. That is, if and only if $e^{it S_n g(\om,y)+ i \phi(y)}=e^{it S_n g(\om,y')+ i \phi(y')}$ for all $y, y' \in (T_\om^{(n)})^{-1}(x)$
(if for some preimage $y$ of $x$ the modulus $\frac{ r(y) }{|(T^{(n)}_\om)'(y)|}$ is zero, we may redefine $\phi(y)$ in such a way that it satisfies this requirement).
Thus, there exists $\phi_n: \X \to \R$ such that $e^{it S_n g(\om,y)+ i \phi(y)}= e^{i\phi_n \circ T^{(n)}_\om(y)}$,
for every $y$ such that $\mcl^{(n)}_\om r(y)\neq 0$. Thus, for all such $y$, we have
\begin{equation}\label{eq:loitv}
\mcl_\om^{it,(n)} v (y) = \mcl^{(n)}_\om (e^{it S_n g(\om,y)+ i \phi(y)} r(y)) = \mcl^{(n)}_\om (e^{i\phi_n \circ T^{(n)}_\om(y)} r(y)) = e^{i\phi_n (y)} \mcl^{(n)}_\om r(y).
\end{equation}
Note that if  $\mcl^{(n)}_\om r(y)= 0$, then $\mcl_\om^{it,(n)} v (y)=0$ as well, so indeed equality between LHS and RHS of \eqref{eq:loitv} holds for a.e. $y \in \X$.

Notice that, by equivariance of  $Y_{\om}^{it}$, $\mcl_\om^{it,(n)} v \in Y_{\sig^n\om}^{it}$, and the polar decomposition of $\mcl_\om^{it,(n)} v$ is precisely given by the RHS of \eqref{eq:loitv}.
Recall that for every $n$ and \paeom, $\mcl_{\sig^{-n}\om}^{it,(n)}: Y_{\sig^{-n}\om}^{it} \to Y_\om^{it}$ is a bijection.
Let $v_{-n} \in Y_{\sig^{-n}\om}^{it}$ be such that $\mcl_{\sig^{-n}\om}^{it,(n)} v_{-n}=v$, and let $r_{-n}=|v_{-n}|$.
We recall that by step~\ref{it:L1Norm} of the proof, $\|r_{-n}\|_1=1$.
Also, \cite[Lemma 20]{FLQ2} implies that for every $\ep>0$ there exists $C_\ep>0$ such that $\|v_{-n}\| \leq C_\ep e^{\ep n} \|v\|$.
Hence, $\|r_{-n}\| \leq \|v_{-n}\| \leq C_\ep e^{\ep n} \|v\|$, where we have used the facts that $\var(|v|)\leq \var(v)$ and $\|\, |v|\, \|_1 = \|v\|_1$ for every $v\in \BV$.
Notice that $\int r_{-n} - v^0_{\sig^{-n}\om} dm=0$, as both $r_{-n}$ and $v^0_{\sig^{-n}\om}$ are non-negative and normalized in $L^1$. Thus, \eqref{eq:NormLitv} applied to $v_{-n}$  and $\sig^{-n}\om$, together with \ref{cond:dec} yields
\begin{equation}\label{eq:r-v0}
\begin{split}
\| r - v^0_\om \| &= \| |\mcl_{\sig^{-n}\om}^{it,(n)} v_{-n} |- v^0_\om \| = \|\mcl_{\sig^{-n}\om}^{(n)} (r_{-n} - v^0_{\sig^{-n}\om}) \| \\
& \leq
K' e^{-\lam n} (\| r_{-n}\| + \|v^0_{\sig^{-n}\om}\|) \leq K' e^{-\lam n} (C_\ep e^{\ep n}\|v\| +  \esssup_{\om \in \Om} \lVert v_\om^0\rVert).
\end{split}
\end{equation}
Let $\ep < \lam$. Then, the quantity on the RHS of \eqref{eq:r-v0} goes to zero as $n\to \infty$ and therefore $r=v_\om^0$, as claimed.

\paragraph{Proof of step \ref{it:vitAngle}.} Let $u, v \in Y_\om^{it}$ be such that $\|v\|_1=\|u\|_1=1$. In view of step \ref{it:vitNorm}, there exist functions $\phi, \psi: \X \to \R$ such that $v=e^{i\phi}v^0_\om$ and $u=e^{i\psi}v^0_\om$.
Since $Y_\om^{it}$ is a vector space, we have $u+v \in Y_\om^{it}$, although $u+v$ may  not be normalized in $L^1$. Hence, again using step \ref{it:vitNorm}, there exist $\rho \in \R$ and $\xi: X \to \R$ such that $v+u = \rho e^{i\xi}v^0_\om$.
Therefore,
\[
v+u=e^{i\phi}v^0_\om + e^{i\psi}v^0_\om = \rho e^{i\xi}v^0_\om.
\]
Recalling that $v^0_\om$ is bounded away from 0, we can divide by $v^0_\om$, and take magnitudes (norms) to get
\[
|e^{i\phi} + e^{i\psi}| = \rho.
\]
Elementary plane geometry shows that this implies $|\phi-\psi|$ is essentially constant (modulo $2\pi$). In particular, $\phi-\psi$ can take at most two values, say $\pm a$.
A similar argument, considering $v$ and $u'=e^{ia}u$ shows that $\phi-\psi-a$ can also take at most two values, say $\pm b$.
Putting this together, we have on the one hand that $\phi-\psi-a  \in \{0, -2a\}$, and on the other hand that $\phi-\psi-a  \in \{b,-b\}$. Thus, either (i) $b=0$, and therefore $v=e^{ia}u$, or (ii) $b\neq 0$ and then $\phi-\psi-a=-2a$, and therefore $\phi=\psi-a$ and $v=e^{-ia}u$.

\end{proof}

\begin{proof}[Proof of Lemma~\ref{lem:AperiodicCoboundary}] \
\paragraph{Equivalence between Assumption~\eqref{aper} and item~\eqref{it:cobi}.}
It is straightforward to check that \eqref{aper} directly implies  item  \eqref{it:cobi}.
To show the converse, assume the hypotheses of Lemma~\ref{lem:AperiodicCoboundary} and item \eqref{it:cobi}. An immediate consequence of upper semi-continuity of $t\mapsto \Lam(it)$, as established in Lemma~\ref{lem:usc}, is that
if $J\subset \R$ is a compact interval not containing 0, then there exists $r<0$ such that $\sup_{t\in J} \Lam(it)<r$. Let $\rho:=e^r$. Then, for
 $\paeom$ and $t\in J$, there exists $C_{\om,t}>0$ such that for every for $n\ge 0$,
\begin{equation}\label{eq:aperNear0}
\lVert \mathcal L_{\om}^{it, (n)}\rVert \le C_{\om,t} \rho^n.
\end{equation}
In order to show \eqref{aper}, we will in fact ensure the constant $C_{\om,t}$ can be chosen independently of $(\om,t)$, provided $(\om,t)\in \hat{\Om} \times J$ for some full $\bbp$-measure subset $\hat{\Om}\subset \Om$.
We will establish this result for $\om \in \hat{\Om} := \cap_{k \in \Z} \, \sig^{k} (\cup_{l=1}^q \Om_l)$.
 Notice that $\hat{\Om}$ is $\sig$-invariant and, since $\sig$ is a $\bbp$-preserving homeomorphism of $\Om$, then $\bbp(\hat\Om)=1$.
 For technical reasons regarding compactness, let us consider $\tilde{\Om}:= \amalg_{1\leq l \leq q} \bar{\Om}_l$;
 where $\amalg$ denotes disjoint union, with the associated disjoint union topology (so $\tilde{\Om}$ may be thought of as $\cup_{1 \leq l \leq q} \big( \{ l \}\times \bar{\Om}_l \big)$, with the finest topology such that each injection $\bar \Om_l \hookrightarrow \{ l \}\times \bar{\Om}_l \subset \tilde{\Om}$ is continuous).   In this way, each  $\{l\} \times \bar \Om_l \subset \tilde{\Om}$ is a clopen set and, since $\Om$ is compact, so is $\tilde \Om$.

  For notational convenience, but in a slight abuse of notation, we drop the `$\{l\}$' component, and identify elements of $\tilde\Om$ with elements of $\Om$, although points on the boundaries between $\bar\Om_l$'s may appear with multiplicity in $\tilde{\Om}$.
 For each $\om_0 \in \bar{\Om}_l  \subset \tilde{\Om}$, we denote $\mcl_{\om_0}$ the (unique) value making $\om \mapsto \mcl_\om$ continuous on $\bar\Om_l \subset \tilde\Om$. This is possible by the assumptions of the lemma and the universal property of the disjoint union topology.
 In addition, notice that each element of $\hat{\Om}$ belongs to exactly one of $\bar{\Om}_1, \dots, \bar{\Om}_q$ and therefore it has a unique representative in $\tilde\Om$. Hence, there is no ambiguity in the definition of $\mcl_\om$ for $\om \in \hat\Om \subset \tilde \Om$.

Let $1\leq l \leq q$
and note that  for every $(\om_0, t_0)\in \bar{\Om}_l \times J \subset \tilde{\Om}\times J$, there is an open neighborhood $U_{(\om_0, t_0)}\subset \tilde{\Om} \times J$
(we emphasize that the topology of $\tilde\Om$ is used here)
and $\bar n= \bar n(\om_0,t_0)<\infty$ such that if $(\om,t)\in U_{(\om_0, t_0)}$ 
 then
$\|\mathcal{L}_{\omega}^{it,(\bar n)}\|\le \rho^{\bar n}$.
Indeed, let $\bar n=\bar n(\om_0, t_0)$ be such that $\|\mathcal{L}_{\omega_0}^{it_0,(\bar n)}\|\leq \rho^{\bar n}/2$.
Recall that Lemma~\ref{1214} ensures that  $t\mapsto M_\om^{it}:= (f(\cdot)\mapsto e^{itg(\om, \cdot)} f(\cdot))$ is continuous in the norm topology of $\B$, so that $(\om, t) \mapsto \mcl_{\om}^{it}$ can be extended continuously to $\bar{\Om}_l\times J$ for each $1\leq l \leq q$, and therefore to all $\tilde{\Om}\times J$. Thus, one can choose
an open neighborhood $U_{(\om_0, t_0)}\subset \tilde{\Om} \times J$
so that if $(\om,t)\in U_{(\om_0, t_0)}$, then $\|\mathcal{L}_\omega^{it,(\bar n)}\|\leq \rho^{\bar n}$, as claimed.

By compactness, there are finite collections (of cardinality, say, $N^l$)  $A^l_1, \dots, A^l_{N^l} \subset \bar \Om_l \times J$ and $n^l_1, \dots, n^l_{N^l}\in \N$  such that $\cup_{j=1}^{N^l}  A^l_j \supset  (\hat{\Om}\cap \Om_l) \times J$ and for every $(\om,t) \in A^l_j \cap \big( (\hat{\Om}\cap \Om_l) \times J \big)$,
$\|\mathcal{L}_\omega^{it, (n^l_j)}\|\leq \rho^{n^l_j}$.

Let  $n_0:=\max_{1\leq l \leq q}\max_{1\leq j \leq N^{l}}  n^l_j<\infty$.
For each $\om \in \hat{\Om}$, let $1\leq l(\om) \leq q$ be the index such that $\om \in \Om_{l(\om)}$.
Let $(\om, t) \in \hat{\Om} \times J$, and let $1\leq j(\om,t) \leq N^{l(\om)}$ be such that $(\om, t)\in A^{l(\om)}_{j(\om,t)}$.
Let us recursively define two sequences $\{m_k(\om,t)\}_{k\geq 0}, \{M_k(\om,t)\}_{k\geq 0} \subset \N$ as follows: $ M_0(\om,t)=0, m_0(\om,t)=n^{l(\om)}_{ j(\om,t)}, M_{k+1}(\om,t)= M_k(\om,t) + m_{k}(\om,t)$ and $m_{k+1}(\om,t) = n^{l(\sig^{M_k(\om,t)}\om)}_{j(\sig^{M_k(\om,t)}\om, t)}$.

Notice that for every $(\om, t)\in \hat{\Om} \times J$ and $k\in \N$, $m_{k}(\om,t) \leq n_0$.
Then, each $n\in \N$ can be decomposed as $n= \big(\sum_{k=0}^{\tilde n-1}  m_k(\om,t) \big ) + \ell$,
where $\tilde n = \tilde n (\om, t, n)\geq 0$ is taken to be as large as possible while ensuring  that $0\leq \ell =\ell(\om,t, n)<n_0$. Choosing $M>1$ such  that $\|\mcl_\om^{it}\|\leq M$ for every $(\om, t) \in \hat{\Om} \times J$  (possible by Lemma \ref{1214}), we get
\begin{eqnarray*}
\|\mathcal{L}_\omega^{it,(n)}\| \leq
\Big( \prod_{k=0}^{\tilde n-1} \|\mathcal{L}_{\sig^{M_{k}(\om,t)}\omega}^{it,(m_{k}(\om,t))}\| \Big) (\|\mathcal{L}_{\sig^{M_{\tilde n}(\om,t)}\omega}^{it,(\ell)}\|) \leq \rho^n\left(M/\rho \right)^\ell\le C\rho^n,
\end{eqnarray*}
for every $(\om, t) \in \hat{\Om} \times J$, where  $C=(M/\rho)^{n_0}$, and \eqref{aper} holds.

\paragraph{Equivalence of items \eqref{it:cobi} and \eqref{it:cobii}.}
Assume item \eqref{it:cobi} holds, and suppose there exists $t\in \R\setminus\{0\}$ such that \eqref{eq:coboundary} has a non-zero, measurable solution. By iterating \eqref{eq:coboundary} $n$ times, and recalling identity \eqref{eq:prodRuleP}, we get
\begin{equation}\label{eq:adjointExprN}
e^{itS_ng(\om,\cdot)}  \mcl_\om^{0*(n)} (\psi_{\sig^{n} \om}) = \ga_\om^{it,n} \psi_\om,
\end{equation}
with $\ga_\om^{it,n} \in S^1$.
Lemma~\ref{lem:exprLit} ensures $\mcl_\om^{it*(n)}(\psi)=  e^{itS_ng(\om, \cdot)} \mcl_\om^{0*(n)}(\psi)$, so \eqref{eq:adjointExprN} implies that
$\|\mcl_\om^{it* (n)}\psi_{\sig^{n}\om} \|_{\B^*}=\|\psi_\om\|_{\B^*}$. Thus, invoking again \cite[Lemma 8.2]{FLQ1},
$\lim_{n\to \infty} \frac{1}{n}\log \|\mcl_\om^{it* (n)}\psi\|_{\B^*}=0$, contradicting item \eqref{it:cobi}.
Hence, \eqref{eq:coboundary} only has solutions when $t=0$. It is direct to check that the choice  $\ga_\om^{0}=1$ and $\psi^0_\om(f)=\int f dm$ provide a solution. Since by hypothesis $\dim Y^0=1$, no other solution may exist, except for constant scalar multiples of $\psi^0_\om$.

Let us show item \eqref{it:cobii} implies item \eqref{it:cobi} by contradiction.
Assume item  \eqref{it:cobii}  holds, and $\Lam(it)=0$ for some nonzero $t\in \R$.
Then, by assumption $\mcl^{it}$ is quasi-compact and
by Lemma~\ref{lem:simpleY1it}, $\dim Y^{it}=1$.
An argument similar to that in Section~\ref{sec:choiceOsBases} implies that there exist non-zero measurable solutions
$v$ to $\mcl_\om^{it}v_{\om} = \hat{\lam}_\om^{it} v_{\sig\om}$ and $\psi$ to $\mcl_\om^{it*}\psi_{\sig\om} = \hat{\lam}_\om^{it} \psi_\om$, chosen so that  $\|v_\om\|_1=1$ and $ \psi_\om( v_\om )=1$ for \paeom.
Thus,  $\int \log |\hat{\lam}_\om^{it}| d\bbp=\Lam(it)=0$.
Recalling that $\|\mcl_\om^{it}\|_1 \leq 1$, we get $|\hat{\lam}_\om^{it}|\leq 1$ for \paeom. Combining the last two statements we get that $|\hat{\lam}_\om^{it}|=1$ for \paeom.
In view of Lemma~\ref{lem:exprLit}(1), $\psi$ yields a solution to \eqref{eq:coboundary}.  Hence, Condition \eqref{it:cobii} implies that $t=0$.
\end{proof}

\subsubsection{Application to random Lasota--Yorke maps}\label{sec:exlclt}

\begin{thm}[Local central limit theorem for random Lasota--Yorke maps]
Assume $\mc{R}=(\Om, \mc{F}, \bbp, \sig, \B, \mcl)$ is an admissible random Lasota-Yorke map (see Section~\ref{RLYM})
such that there exists $1\leq q<\infty$, essentially disjoint compact sets $\Om_1, \dots, \Om_q\subset \Om$ with $\cup_{j=1}^q \Om_j=\Om$, and maps $\{ T_j : I \to I\}_{1\leq j \le q}$ such that $T_{\om}= T_j$ for $\bbp$ a.e. $\om \in \Om_j$. Let $g \colon \Omega \times X \to \mathbb R$ be an observable satisfying the regularity and centering conditions  \eqref{obs} and \eqref{zeromean}.
Then one of the two following conditions holds:
\begin{enumerate}
\item
$\mc{R}$ satisfies the local central limit theorem (Theorem~\ref{thm:lclt}), or
\item
The observable is periodic, that is, \eqref{eq:coboundary} has a
measurable non-zero solution $\psi:=\{\psi_\om\}_{\om\in\Om}$ with $\psi_\om \in \B^*$, for some $t\in \R\setminus \{0\}$, $\ga_\om^{it}\in S^1$. (See Section~\ref{sec:plclt} for further information in this setting.)
\end{enumerate}
\end{thm}


\begin{proof}
 Lemma~\ref{lem:exprLit} ensures that for any $n\in \N$ and $f\in \BV$,
 \[
  \mcl_\om^{it, (n)}f=\mcl_\om^{(n)}(e^{it S_{n}g(\om, \cdot)}f).
 \]
 In order to verify the quasicompactness condition for $\mc{R}^{it}$ for $t\in \R$, we adapt an argument of Morita \cite{Morita,MoritaCorrection}.
 First note that since the $T_\om$ take only finitely many values, then $\mc{R}$ has a uniform big-image property.
That is, for every $n\in \N$,
\[
\essinf_{\om \in \Om} \min_{1\leq j \leq b^{(n)}_\om} m(T^{(n)}_\om(J^{(n)}_{\om,j}))>0,
\]
where $J^{(n)}_{\om,1}, \dots, J^{(n)}_{\om,b^{(n)}_\om}$,
 are the regularity intervals of  $T^{(n)}_\om$. Indeed, the infimum is taken over a finite set.
Then, the argument of
  \cite[Proposition 1.2]{Morita} (see also \cite{MoritaCorrection}), with straightforward changes to fit the random situation, ensures that
\begin{equation}\label{eq:MoritaLY}
\var (\mcl_\om^{it,(n)}(f)) = \var(\mcl_\om^{(n)}(e^{it S_{n}g(\om, \cdot)}f)) \leq (2+ n \var(e^{itg(\om, \cdot)})) (\del^{-n} \var(f) + I_n(\om)\|f\|_{1}),
\end{equation}
 for some measurable function $I_n$.

Let $n_0$ be sufficiently large so that $a_{n_0}:=(2+ n_0 \var(e^{itg(\om, \cdot)})) \del^{-n_0} <1$. Then,
\[
\|\mcl_\om^{it,(n_0)}(f)\|_{\BV} \leq a_{n_0} \|f\|_{\BV} + J_{n_0}(\om)\|f\|_1,
\]
for some  measurable function $J_{n_0}$.
Lemma~\ref{lem:Hennion} implies that $\ka(it) \leq \log (a_{n_0})/n_0 <0 = \Lam(it)$. Thus, the cocycle $\mc{R}^{it}$ is quasicompact.
The result now follows directly from Theorem~\ref{thm:lclt} and Lemma~\ref{lem:AperiodicCoboundary}, which is applicable since $\om \mapsto \mcl_\om$ is essentially constant on each of the $\Om_j$.
\end{proof}

\subsection{Local central limit theorem: periodic case}\label{sec:plclt}

We now discuss the version of local central limit theorem for a certain class of observables for which  the aperiodicity condition~\ref{cond:aper} fails to hold. More precisely, we are
interested in observables of the form
\begin{equation}\label{no}
 g(\omega, x)=\eta_\om+k(\om, x), \quad \text{where $\eta_\om \in \R$ and $k(\om, \cdot)$ takes integer values for \paeom,}
\end{equation}
 that cannot be written in the form
\begin{equation}\label{form}
  g(\om, \cdot)=\eta_\om' +h(\om, \cdot)-h(\sigma \om, T_\om(\cdot))+p_\om k'(\om, \cdot),
 \end{equation}
for $\eta_\om' \in \R$, $p_\om \in \mathbb N \setminus \{1\}$ and $k'(\om, x) \in \Z$.
Furthermore, we will continue to assume that $g$ satisfies assumptions~\eqref{obs} and~\eqref{zeromean}. We note that in this setting~\eqref{eq:coboundary} holds with $t=2\pi$,
$\gamma_\om^{it}=e^{it\eta_\om}$ and $\psi_\om (f) = \int f dm$. Consequently, Lemma~\ref{lem:AperiodicCoboundary} implies that~\ref{cond:aper} does not hold.

Let $G$ denote the set of all $t\in \R$ with the property that there exists a measurable function $\Psi \colon \Om \times X \to S^1$ and a collection of numbers $\gamma_\om \in S^1$, $\om \in \Om$ such that:
\begin{enumerate}
\item $\Psi_\omega \in \mathcal B$ for \paeom,  where $\Psi_\om:=\Psi(\om, \cdot)$;
\item  for \paeom, \begin{equation}\label{group}
 e^{-it g(\om, \cdot)} \Psi_{\sigma \om}\circ T_\om=\gamma_\om \Psi_\om.
\end{equation}
\end{enumerate}
\begin{lemma}
 $G$ is a subgroup of $(\R, +)$.
\end{lemma}

\begin{proof}
 Assume that $t_1, t_2 \in G$ and let $\Psi^j \colon \Om \times X \to S^1$, $j=1, 2$ be measurable functions satisfying $\Psi^j_\om \in \mathcal B$ for \paeom,  $j=1, 2$ and $\gamma^j_\om \in S^1$, $\om \in \Om$, $j=1, 2$ collections of numbers such that
 \[
   e^{-it_j g(\om, \cdot)} \Psi^j_{\sigma \om}\circ T_\om=\gamma_\om^j \Psi^j_\om \quad \text{for \paeom \ and $j=1, 2$.}
\]
By multiplying those two identities, we obtain that
\[
 e^{-i(t_1+t_2)g(\om, \cdot)}\Psi_{\sigma \om}\circ T_\om=\gamma_\om \Psi_\om \quad \paeom,
\]
where $\Psi(\om, x)=\Psi^1(\om, x)\Psi^2(\om, x)$ and $\gamma_\om=\gamma^1_\om \cdot \gamma^2_\om$ for $\om \in \Om$ and $x\in X$. Noting that $\Psi$ takes values in $S^1$, $\Psi_\om \in \mathcal B$ for \paeom \  and that $\gamma_\om \in S^1$
for each $\om \in \Om$, we conclude that $t_1+t_2 \in G$.

Assume now that $t\in G$ and let $\Psi\colon \Om \times X \to S^1$ be a measurable function satisfying $\Psi_\om\in \mathcal B$ for \paeom \   and $\gamma_\om \in S^1$, $\om \in \Om$ a collection of numbers such that~\eqref{group} holds.
Conjugating the identity~\eqref{group}, we obtain that
\[
 e^{it g(\om, \cdot)} \overline{\Psi_{\sigma \om}}\circ T_\om=\overline{\gamma_\om } \overline{\Psi_\om} \quad \paeom,
\]
which readily implies that $-t\in G$.
\end{proof}

\begin{lemma}\label{g2}
 If $\Lambda (it)=0$ for $t\in \R$, then $t\in G$.
\end{lemma}

\begin{proof}
 Assume that $\Lambda (it)=0$ for some $t\in \mathbb R$. In Section~\ref{sec:aper}, we have showed that
 in this case, $\dim Y^{it}=1$ and if $v_\om \in \mathcal {B}$ is a generator of $Y^{it}_\om$ satisfying $\lVert v_\om \rVert_1=1$, then, for \paeom, $\lvert v_\om\rvert=v_\om^0$  and
 \begin{equation}\label{g1}
  \mcl_\om (e^{it g(\om, \cdot)} v_\om)=\gamma_\om v_{\sigma \om},
 \end{equation}
for some $\gamma_\om \in S^1$. For $\om \in \Om,  \ x\in X$, set
\[
 \Psi(\om, x)=\frac{v_\om (x)}{v_\om^0(x)}.
\]
Then, $\Psi$ is $S^1$-valued and $\Psi_\om \in \mathcal B$ for \paeom.
Set
\[
\varphi_\om :=\overline{\gamma_\om} e^{it g(\om, \cdot)} \quad \text{and} \quad
\Phi_\om:=\overline{\varphi_\om}  \Psi_{\sigma \om} \circ T_\om, \quad \om \in \Om.
\]
Then, we have that
\[
\begin{split}
\int \lvert \Phi_\om-\Psi_\om \rvert^2\, d\mu_\om  &=\int (\overline{\varphi_\om}  \Psi_{\sigma \om} \circ T_\om -\Psi_\om)(\varphi_\om \overline{\Psi_{\sigma \om}} \circ T_\om-\overline{\Psi_\om})\, d\mu_\om \\
&=\int \lvert \varphi_\om\rvert^2 \cdot (\lvert \Psi_{\sigma \om}\rvert^2 \circ T_\om) \, d\mu_\om +\int \lvert \Psi_\om \rvert^2\, d\mu_\om \\
&\phantom{=}-\int \varphi_\om \Psi_\om (\overline{\Psi_{\sigma \om}}\circ T_\om)\, d\mu_\om-\int \overline{\varphi_\om} \overline{ \Psi_\om} (\Psi_{\sigma \om}\circ T_\om)\, d\mu_\om.
\end{split}
\]
Since $\Psi_\om$ and $\varphi_\om$ take values in $S^1$ for each $\om \in \Om$, we obtain that
\[
\int \lvert \varphi_\om\rvert^2 \cdot (\lvert \Psi_{\sigma \om}\rvert^2 \circ T_\om) \, d\mu_\om=\int \lvert \Psi_\om \rvert^2\, d\mu_\om=1.
\]
On the other hand, by using~\eqref{g1} we have that
\[
\begin{split}
\int \varphi_\om \Psi_\om (\overline{\Psi_{\sigma \om}}\circ T_\om)\, d\mu_\om &=\int \varphi_\om v_\om^0 \Psi_\om (\overline{\Psi_{\sigma \om}}\circ T_\om)\, dm \\
&=\int \varphi_\om v_\om (\overline{\Psi_{\sigma \om}}\circ T_\om)\, dm \\
&=\int \mcl_\om (\varphi_\om v_\om (\overline{\Psi_{\sigma \om}}\circ T_\om))\, dm  \\
&=\int \overline{ \Psi_{\sigma \om}}\mcl_\om (\varphi_\om v_\om )\, dm  \\
&=\int \overline{\Psi_{\sigma \om}}v_{\sigma \om}\, dm \\
&=\int \frac{\lvert v_{\sigma \om} \rvert^2}{v_{\sigma \om}^0}  \, dm \\
&=\int v_{\sigma \om}^0 \, dm\\
&=1.
\end{split}
\]
Consequently, we also have that
\[
\int \overline{\varphi_\om} \overline{ \Psi_\om} (\Psi_{\sigma \om}\circ T_\om)\, d\mu_\om=1,
\]
and thus
\[
\int \lvert \Phi_\om-\Psi_\om \rvert^2\, d\mu_\om=0.
\]
Therefore,
\[
 e^{-it g(\om, \cdot)} \Psi_{\sigma \om}\circ T_\om=\overline{\gamma_\om} \Psi_\om \quad \paeom,
\]
which implies that $t\in G$.
\end{proof}

We now establish the converse of Lemma~\ref{g2}.
\begin{lemma}\label{dda}
If $t\in G$, then $\Lambda (it)=0$.
\end{lemma}

\begin{proof}
Assume that $t\in G$ and let $\Psi\colon \Om \times X \to S^1$ be a measurable function satisfying $\Psi_\om\in \mathcal B$ for \paeom \   and $\gamma_\om \in S^1$, $\om \in \Om$ a collection of numbers such that~\eqref{group} holds. It follows from~\eqref{group} that
\[
v_{ \om}^0 (\Psi_{\sigma \om} \circ T_\om)=\gamma_\om e^{it g(\om, \cdot)}\Psi_\om v_\om^0 \quad \text{for \paeom,}
\]
and thus
\[
\mcl_\om (v_{ \om}^0 (\Psi_{\sigma \om} \circ T_\om))=\gamma_\om \mcl_\om^{it} (\Psi_\om v_\om^0) \quad \text{for \paeom.}
\]
Consequently,
\begin{equation}\label{bnj}
\Psi_{\sigma \om}v_{\sigma \om}^0=\gamma_\om \mcl_\om^{it} (\Psi_\om v_\om^0) \quad \text{for \paeom.}
\end{equation}
Setting $v_\om :=\Psi_\om v_\om^0$, $\om \in \Om$, we have that
\[
v_\om \in \mathcal B, \quad v_{\sig\om} = \ga_\om  \mcl_\om^{it} (v_{\om})  \quad \text{and} \quad \lVert v_\om \rVert_1=1, \quad \paeom.
\]
Hence,  \eqref{bnj} implies that
\[
\lVert  \mcl_\om^{it} v_\om \rVert_1=1,  \quad \text{for \paeom.}
\]
Therefore,
\[
\lim_{n\to \infty} \frac 1 n \log \lVert \mcl_\om^{it, (n)} v_\om \rVert_1=0 \quad \text{for \paeom,}
\]
 and thus it follows from Lemma~\ref{lem:RandomSameExp} that $\Lambda (it)=0$.
\end{proof}
It follows directly from~\eqref{no} that $2\pi \in G$ since in this case~\eqref{group} holds with $\Psi(\om, x)=1$ and $\gamma_\om=e^{i 2\pi \eta_\om} \in S^1$. Furthermore, we will show that
our additional assumption that $g$ cannot be written in a form~\eqref{form} implies that $G$ is generated by $2\pi$. We begin by proving that $G$ is discrete.
\begin{lemma}
There exists $a>0$ such that
\begin{equation}\label{groupf}
G=\{ak: k\in \Z\}.
\end{equation}
\end{lemma}

\begin{proof}
Assume that $G$ is not of the form~\eqref{groupf} for any $a>0$ . Since $G$ is non-trivial (recall that $2\pi \in G$), we conclude  that $G$ is dense. On the other hand, it follows easily from Corollary~\ref{cor:LamHatLam} and Lemma~\ref{lem:Lam''0} that $\Lambda (it)<0$ for all $t\neq 0$, $t$ sufficiently close to $0$. This yields a contradiction with Lemma~\ref{dda}.
\end{proof}

\begin{lemma}\label{g3}
$G$ is of the form~\eqref{groupf} with $a=2\pi$.

\end{lemma}

\begin{proof}
 Assume that the group  $G$ is not generated by $2\pi$ and denote its generator by $t\in (0, 2\pi)$. In particular, $\frac{2\pi}{t} \in \mathbb N \setminus \{1\}$. Since $t\in G$, there
 exists a measurable function $\Psi\colon \Om \times X \to S^1$ and a collection of numbers $\gamma_\om \in S^1$, $\om \in \Om$ such that~\eqref{group} holds. Writing $\gamma_\om=e^{i r_\om}$, $r_\om \in \R$ and
$\Psi(\om, x)=e^{i H(\om, x)}$ for some measurable $H\colon \Om \times X \to \R$, it follows from~\eqref{group} that
\[
 -tg(\om, x)=r_\om+H(\om, x)-H(\sigma \om, T_\om x)+2\pi k'(\om, x) \quad \text{for $\om \in \Om$ and $x\in X$,}
\]
where $k' \colon \Om \times X \to \mathbb Z$. This implies that $g$ is of the form~\eqref{form} which yields a  contradiction.
\end{proof}

We are now in a position to establish the  periodic version of local central limit theorem.

\begin{thm}\label{thm:lcltp}
 Assume that $g$ has the form~\eqref{no}. In addition, we assume that $g$ cannot be written in the form~\eqref{form}.
  Then, for \paeom\  and every bounded interval $J\subset \R$, we have:
 \[
  \lim_{n\to \infty}\sup_{s\in \R} \bigg{\lvert} \Sig \sqrt{n} \mu_\om (s+S_n g(\om, \cdot)\in J)-\frac{1}{\sqrt{2\pi}}e^{-\frac{s^2}{2n\Sig^2}}\sum_{l=-\infty}^{+\infty}{\bf 1}_J(\overline{\eta}_{\om}(n)+s+l) \bigg{\rvert}=0,
 \]
 where $\overline{\eta}_{\om}(n)=\sum_{i=0}^{n-1}\eta_{\sigma^i\om}$.
\end{thm}

\begin{proof}
 Using again the density argument (see~\cite{Morita}), it is sufficient to show that
 \[
 \sup_{s\in \R} \bigg{\lvert} \Sig \sqrt{n}\int h(s+S_ng(\om, \cdot))\, d\mu_\om-\frac{1}{\sqrt{2\pi}}e^{-\frac{s^2}{2n\Sig^2}}\sum_{l=-\infty}^{+\infty}h(\overline{\eta}_{\om}(n)+s+l)\bigg{\rvert} \to 0.
 \]
when $n\to \infty$ for every $h\in L^1(\R)$ whose Fourier transform $\hat{h}$ has compact support. As in the proof of Theorem~\ref{thm:lclt}, we have that
\[
   \Sigma \sqrt{n}\int_0^1 h(s+S_ng(\om,\cdot))\, d\mu_\om =\frac{\Sigma \sqrt{n}}{2\pi}\int_{\R} e^{its}\hat h(t)\int_0^1 \mathcal L_{\om}^{it, (n)}v_\om^0\, dm \, dt,
\]
and therefore (using Lemma~\ref{lem:exprLit})
\begin{align*}
 & \Sigma \sqrt{n}\int_0^1 h(s+S_ng(\om,\cdot))\, d\mu_\om  \displaybreak[0] \\
 &=\frac{\Sigma \sqrt{n}}{2\pi}\sum_{l=-\infty}^{\infty}\int_{-\pi+2l\pi}^{\pi+2l\pi} e^{its}\hat h(t) e^{it\overline{\eta}_{\om}(n)} \int_0^1 \mathcal L_{\om}^{
(n)}(e^{itS_n k(\om, \cdot)}v_\om^0)\, dm \, dt\displaybreak[0] \\
&=\frac{\Sigma \sqrt{n}}{2\pi}\sum_{l=-\infty}^{\infty}\int_{-\pi}^{\pi} \hat h(t+2l\pi) e^{i(t+2l\pi)(\overline{\eta}_{\om}(n)+s)} \int_0^1 \mathcal L_{\om}^{
(n)}(e^{itS_n k(\om, \cdot)}v_\om^0)\, dm \, dt \displaybreak[0] \\
&=\frac{\Sigma \sqrt{n}}{2\pi}\int_{-\pi}^{\pi}H_s(t)e^{its}\int_0^1 \mathcal L_{\om}^{it,
(n)}v_\om^0\, dm \, dt \displaybreak[0] \\
&=\frac{\Sigma }{2\pi}\int_{-\pi \sqrt n}^{\pi \sqrt n}H_s(\frac{t}{\sqrt n})e^{\frac{its}{\sqrt n}}\int_0^1 \mathcal L_{\om}^{\frac{it}{\sqrt n},
(n)}v_\om^0\, dm \, dt,
\end{align*}
where
\[
 H_s(t):=\sum_{l=-\infty}^{+\infty}\hat{h}(t+2l\pi)e^{i2l\pi(\overline{\eta}_{\om}(n)+s)}
\]
Proceeding as in~\cite[p. 787]{RE83}, we have
\[
 \frac{1}{\sqrt{2\pi}}e^{-\frac{s^2}{2n\Sig^2}}\sum_{l=-\infty}^{\infty}h(\overline{\eta}_{\om}(n)+s+l)=\frac{H_s(0)\Sigma}{2\pi}\int_{\R} e^{\frac{its}{\sqrt n}}\cdot e^{\frac{-\Sigma^2 t^2}{2}}\, dt.
\]
Hence, we need to prove that
\[
 \sup_{s\in \mathbb R} \bigg{\lvert} \frac{\Sigma }{2\pi}\int_{-\pi \sqrt n}^{\pi \sqrt n}H_s(\frac{t}{\sqrt n})e^{\frac{its}{\sqrt n}}\int_0^1 \mathcal L_{\om}^{\frac{it}{\sqrt n},
(n)}v_\om^0\, dm \, dt-\frac{H_s(0)\Sigma}{2\pi}\int_{\R} e^{\frac{its}{\sqrt n}}\cdot e^{\frac{-\Sigma^2 t^2}{2}}\, dt \bigg{\rvert} \to 0,
\]
when $n\to \infty$. For $\tilde \delta >0$ sufficiently small, we have (as in the proof of Theorem~\ref{thm:lclt}) that
\begin{align*}
 &\frac{\Sigma}{2\pi}\int_{-\pi \sqrt n}^{\pi \sqrt n} e^{\frac{its}{\sqrt n}}H_s(\frac{t}{\sqrt n})\int_0^1 \mathcal L_{\om}^{\frac{it}{\sqrt n}, (n)}v_\om^0\, dm \, dt -
 \frac{H_s(0)\Sigma}{2\pi} \int_{\R} e^{\frac{its}{\sqrt n}} \cdot e^{-\frac{\Sig^2 t^2}{2}}\, dt \displaybreak[0] \\
 &=\frac{\Sigma}{2\pi} \int_{\lvert t\rvert < \tilde \delta \sqrt n} e^{\frac{its}{\sqrt n}} \Big(H_s(\frac{t}{\sqrt n})\prod_{j=0}^{n-1}\lambda_{\sigma^j \om}^{\frac{it}{\sqrt n}}-H_s(0)e^{-\frac{\Sig^2 t^2}{2}} \Big)\, dt \displaybreak[0] \\
&\phantom{=}+\frac{\Sigma}{2\pi}\int_{\lvert t\rvert < \tilde \delta \sqrt n}e^{\frac{its}{\sqrt n}}H_s(\frac{t}{\sqrt n})
\int_0^1
\prod_{j=0}^{n-1}\lambda_{\sigma^j \om}^{\frac{it}{\sqrt n}} \Big( \phi_\om^{\frac{it}{\sqrt n}}( v_\om^0 ) v_{\sig^{n}\om}^{\frac{it}{\sqrt n}}-1 \Big)\,dm \, dt \displaybreak[0] \\
&\phantom{=}+\frac{\Sig \sqrt{n}}{2\pi}\int_{\lvert t\rvert <\tilde \delta}e^{its}H_s (t)\int_0^1  \mathcal L_{\om}^{it, (n)} (v_\om^0 - \phi_\om^{it}( v_\om^0 ) v_{\om}^{it}) \, dm\, dt \displaybreak[0] \\
&\phantom{=}+\frac{\Sig \sqrt{n}}{2\pi}\int_{\tilde \delta \le \lvert t\rvert \le \pi}e^{its}H_s(t)\int_0^1 \mathcal L_\om^{it, (n)}v_\om^0\, dm\, dt \displaybreak[0] \\
&\phantom{=}-\frac{\Sigma}{2\pi}H_s(0) \int_{\lvert t\rvert \ge \tilde \delta \sqrt n}e^{\frac{its}{\sqrt n}} \cdot e^{-\frac{\Sig^2 t^2}{2}}\, dt=: (I)+(II)+(III)+(IV)+(V).
\end{align*}
Now the arguments follow closely the proof of Theorem~\ref{thm:lclt} with some appropriate modifications. In orter to illustrate those, let us restrict to dealing with the terms (I) and (IV).
Regarding (I), we can control it as in the proof of Theorem~\ref{thm:lclt} once we show the following lemma.
\begin{lemma}
 For each $t$ such that $\lvert t\rvert <\tilde \delta \sqrt n$, we have that $H_s(\frac{t}{\sqrt n})\to H_s(0)$ uniformly over $s$.
\end{lemma}
\begin{proof}[Proof of the lemma]
This follows from a simple observation, that since $\hat h$ has a finite support, there exists $K\subset \Z$ finite such that
\[
 H_s(\frac{t}{\sqrt n})=\sum_{l\in K} \hat{h}(t/\sqrt n+2l\pi)e^{i2l\pi(\overline{\eta}_{\om}(n)+s)}, \quad \text{for each $t$ such that $\lvert t\rvert <\tilde \delta \sqrt n$ and  $s\in \R$.}
\]
Hence,
\[
 \lvert H_s(\frac{t}{\sqrt n})-H_s(0)\rvert \le \sum_{l\in K} \lvert \hat{h}(t/\sqrt n+2l\pi)-\hat h(2l\pi)\rvert.
\]
The desired conclusion now follows from continuity of $\hat h$.
\end{proof}
Finally, term (IV) can be treated as in the proof of Theorem~\ref{thm:lclt} once we note that Lemmas~\ref{g2} and~\ref{g3}  imply that $\Lambda (it)<0$ for each
$t$ such that $\tilde \delta \le \lvert t\rvert \le \pi$.

\end{proof}
\appendix
\section{Technical results involving notions of volume growth}
\label{sec:usc}
In this section we recall some notions of volume growth under linear transformations on Banach spaces, borrowed from  \cite{GTQ2,BlumenthalMET}.
We then state and prove a result on upper semi-continuity of Lyapunov exponents (Lemma~\ref{lem:usc}).
We then prove Corollary~\ref{cor:METAdjoint} and Step~\ref{it:L1Norm} in the proof of Lemma~\ref{lem:simpleY1it}.

\begin{defn}\label{def:volGrowth}
Let $(\ba, \|\cdot\|)$ be a Banach space and $A\in L(\ba)$. For each $k\in \N$, let us define:
\begin{itemize}
\item $V_k(A)=\sup_{\dim E = k} \frac{m_{AE}(AS)}{m_E(S)}$, where
$m_{E}$ denotes the normalised Haar measure on  the linear subspace $E\subset B$, so that the unit ball in $B_E(0,1)\subset E$ has measure (volume) given by the volume of the Euclidean unit ball in $\R^k$, and $S\subset E$ is any non-zero, finite $m_E$ volume set: the choice of $S$ does not affect the quotient $ \frac{m_{AE}(AS)}{m_E(S)}$.
\item
$D_k(A)=\sup_{\|v_1\|=\dots=\|v_k\|=1} \prod_{i=1}^k d(Av_i, lin(\{Av_j: j<i\}))$, where $lin(X)$ denotes the linear span of the finite collection $X$ of elements of $B$, $lin(\emptyset)=\{0\}$, and $d(v,W)$ is the distance from the vector $v$ to the subspace $W\subset \ba$.

\item
 $F_k(A) := \sup_{\dim V = k}\inf_{v\in V, \|v\|=1} \|A v\| =  \sup_{\dim V = k}\inf_{v\in V\setminus\{0\}} \|A v\|/ \|v\|$.
\end{itemize}
\end{defn}
We note that each of $V_k(A), D_k(A)$ and $\Pi_{j=1}^k F_j(A)$ has the interpretation of  growth of $k$-dimensional volumes spanned by $\{A v_j\}_{1\leq j \leq k}$,  where the $v_j\in \ba$ are unit length vectors.

Given functions $F, G: L(\ba) \to \R$, we use the notation $F(A)\approx G(A)$ to mean that
there is a constant $c>1$ independent of $A \in L(\ba)$ (but possibly depending on $k$ if $F$ and/or $G$ do), such that $c^{-1}F(A) \leq G(A) \leq c F(A)$.
The symbols $\lesssim$ and $\gtrsim$ will denote the corresponding one-sided relations.
We start with the following technical lemma.

\begin{lemma}\label{lem:equivVol}
 For each $k\geq 1$, the following hold:
\begin{enumerate}
\item \label{it:subad}
$A\mapsto V_k(A)$ and $A \mapsto D_k(A)$ are sub-additive functions.
\item
 $V_k(A) \approx D_k(A) \approx \Pi_{j=1}^k F_j(A)$.
 \end{enumerate}
\end{lemma}

\begin{proof}
The first part is established in \cite{BlumenthalMET} and \cite{GTQ2}, for $V$ and $D$, respectively.

Next we show the second claim.
Assume $S\subset E$ is a parallelogram, $S=P[w_1, \dots, w_k]:=\{ \sum_{i=1}^k a_i w_i : 0\leq a_i\leq 1 \}$. Then, \cite[Lemma 1.2]{BlumenthalMET} shows that
\begin{equation}\label{eq:equivVol}
m_E(S)\approx \prod_{i=1}^k d(w_i, lin(\{w_j: j<i\})).
\end{equation}
That is, there is a constant $c>1$ independent of $E$ and $(w_1, \dots, w_k)$, but possibly depending on $k$, such that $c^{-1} m_E(S)\leq \prod_{i=1}^k d(w_i, lin(\{w_j: j<i\})) \leq c m_E(S)$.
By a lemma of Gohberg and Klein \cite[Chapter 4, Lemma 2.3]{kato}, it is possible to choose unit length $v_1, \dots, v_k \in E$ such that $d(v_i, lin(\{v_j: j<i\}))=1$ for every $1\leq i \leq k$.
Then, letting $S=P[v_1, \dots, v_k]$, we get that $m_E(S)\approx 1$ and
$\frac{m_{AE}(AS)}{m_E(S)}\approx  \prod_{i=1}^k d(Av_i, lin(\{Av_j: j<i\})) \leq D_k(A)$.
Thus, $V_k(A)\lesssim D_k(A)$.

On the other hand, for each collection of unit length vectors $w_1, \dots, w_k \in E$, we have that $S:=P[w_1, \dots, w_k]\subset  B_E(0,k)$. Hence, $m_E(S)\leq k$ and
$\frac{m_{AE}(AS)}{m_E(S)}\geq k^{-1} m_{AE}(AS)$.
It follows from \eqref{eq:equivVol} that $V_k(A)\gtrsim D_k(A)$.
Combining, we conclude $V_k(A) \approx D_k(A)$ as desired.

The fact that $D_k(A) \approx \Pi_{j=1}^k F_j(A)$ is established in \cite[Corollary 6]{GTQ2}.
\end{proof}

\begin{lemma}[Upper semi-continuity of Lyapunov exponents]\label{lem:usc}
Let $\mathcal R^\theta=(\Omega,\mathcal F,\mathbb
P,\sigma,\ba,\mathcal L^\theta)$ be a  quasi-compact cocycle for every  $\theta$ in a neighborhood $U$ of $\theta_0\in \C$.
Suppose that the family of functions $\{ \om \mapsto \log^+\|\mcl_\om^\theta\| \}_{\theta\in U}$ are dominated by an integrable function, and that for each $\om \in \Om$, $\theta \mapsto \mcl_\om^\theta$ is continuous in the norm topology of $\ba$, for $\theta\in U$. Assume that \ref{cond:unifNormBd} holds, and \ref{cond:METCond}  holds (with  $\mcl=\mc{L^\theta}$) for every $\theta\in U$.

Let $\lam_1^\theta=\mu_1^\theta \geq \mu_2^\theta \geq \dots >\ka(\theta)$ be the exceptional Lyapunov exponents of $\RR^\theta$, enumerated with multiplicity.
Then for every $k \geq 1$, the function $\theta \mapsto \mu^\theta_1 + \mu^\theta_2+\dots +\mu^\theta_k$ is upper semicontinuous at $\theta_0$.
\end{lemma}
\begin{proof}
The strategy of proof follows that of the finite-dimensional situation, using the $k$-dimensional volume growth rate interpretation of $ \mu^\theta_1 +\dots+ \mu^\theta_k$.
Recall that \ref{cond:METCond} ($\mathbb{P}$-continuity) implies the uniform measurability condition of \cite{BlumenthalMET}; see \cite[Remark 1.4]{BlumenthalMET}.
Hence, \cite[Corollary 3.1 \& Lemma 3.2]{BlumenthalMET}, together with Kingman's sub-additive ergodic theorem applied to the submultiplicative, measurable function $V_k$ (see Lemma~\ref{lem:equivVol}\eqref{it:subad}), imply that
$\mu^\theta_1 +\dots+ \mu^\theta_k= \inf_{n\geq 1} \frac1n\int \log V_k(\mcl_\om^{\theta,(n)}) d\bbp$.


Thus, upper semi-continuity of  $\theta \mapsto \mu^\theta_1 +\dots + \mu^\theta_k$ at $\theta_0$ would follow immediately once we show $\theta \mapsto \int \log V_k(\mcl_\om^{\theta,(n)}) d\bbp$ is upper semi-continuous at $\theta_0$ for every $n$. From now on, assume $\theta \in U$.
In view of the continuity hypothesis on $\theta \mapsto \mcl^\theta_\om$,
it follows from continuity of the composition operation $(L_1, L_2)\mapsto L_1\circ L_2$ with respect to the norm topology on $\ba$ and \cite[Lemma 2.20]{BlumenthalMET}, that
 $\theta \mapsto V_k(\mcl_\om^{\theta,(n)})$ is continuous
 for every $n\geq 1$ and \paeom.
Also, $\log V_k(\mcl_\om^{\theta,(n)}) \leq k \log\| \mcl_\om^{\theta,(n)} \| \leq k \sum_{j=0}^{n-1} \log^+  \| \mcl_{\sig^j\om}^{\theta} \|$. When $\theta\in U$, the last expression is dominated by an
integrable function with respect to $\bbp$, by
 the domination hypothesis and $\bbp$-invariance of $\sigma$.
 Thus, the (reverse) Fatou lemma yields
 $ \int \log V_k(\mcl_\om^{\theta_0,(n)}) d\bbp  \geq  \limsup_{\theta\to \theta_0}\int \log V_k(\mcl_\om^{\theta,(n)}) d\bbp$, as required.
\end{proof}

\subsection{Proof of Corollary~\ref{cor:METAdjoint}}
 We first note that the quasicompactness of $\mc{R}^*$ and condition~\ref{cond:METCond} follow from Remark~\ref{rmk:METCond*}. Thus, Theorem~\ref{thm:MET} ensures the existence of a unique measurable equivariant Oseledets splitting for $\mc{R}^*$.


Recall that, in the context of Corollary~\ref{cor:METAdjoint},  Lemma~\ref{lem:equivVol} shows that
$V_k, D_k: L(\B)\to \R$ are equivalent up to a constant multiplicative factor. Thus,  \cite[Lemma 3]{GTQ2} ensures that $V_k(A)$ and $V_k(A^*)$ are equivalent up to a multiplicative factor, independent of $A$, and the claim on Lyapunov exponents and multiplicities follows from \cite[Theorem 1.3]{BlumenthalMET}.
\qed

\subsection{Proof of  Lemma~\ref{lem:simpleY1it}, Step~\ref{it:L1Norm}}\label{sec:pfStep1}
We recall that for every $v\in S_1:= \{y\in \B : \|y\|_1=1\}$, $\|\mcl_\om^{it} v\|_1= \|\mcl_\om (e^{itg(\om, \cdot)}v) \|_1 \leq \| e^{itg(\om, \cdot)}v \|_1=1$, so it only remains to show that
$\inf_{v \in Y_\om^{it}\cap S_1} \|\mcl_\om^{it} v\|_1 =1$.
We will use the notation of Definition~\ref{def:volGrowth}, with the dependence on the Banach space $\ba$ made explicit, so that $F^{\ba}_j(\mcl) := \sup_{\dim V = j}\inf_{v\in V, \|v\|_\ba=1} \|\mcl v\|_\ba$. (In our context either $\ba=\B$ or $\ba=L^1$.)

Lemma~\ref{lem:equivVol} and \cite[Corollary 6]{GTQ2} ensure that $V^{\B}_d(\mcl)\approx  \Pi_{j=1}^d F^{\B}_j (\mcl)$.
 For shorthand, in the rest of the section we will denote $\|v\|:=\|v\|_{\BV}, \|v\|_1:=\|v\|_{L^1}, F^{\BV}_j(\mcl)=:F_j(\mcl)$ and $F^{L^1}_j(\mcl) =: F^{1}_j(\mcl)$, with similar conventions for $V_j$.

By Kingman's  sub-additive ergodic theorem and the relations $V^{\B}_d(\mcl)\approx \Pi_{j=1}^d F^{\B}_j (\mcl)$, each of the limits (i) $\lim_{n\to\infty} \frac{1}{n} \log \Pi_{j=1}^d F^1_j(\mcl_\om^{it,(n)}|_{Y^{it}_{\om}})$ and (ii) $\lim_{n\to\infty} \frac{1}{n} \log \Pi_{j=1}^d F_j(\mcl_\om^{it,(n)}|_{Y^{it}_{\om}})$ exists for \paeom, is independent of $\om$ and in fact it coincides with the sum of the top $d$ 
Lyapunov exponents (all of which are equal) of the cocycles
$(\Om, \mc{F}, \bbp, \sig, L^1, \{ \mcl^{it}_\om|_{Y^{it}_{\om}} \} )$ and $(\Om, \mc{F}, \bbp, \sig, \BV, \{\mcl^{it}_\om|_{Y^{it}_{\om}}\} )$, respectively. Thus, these limits agree by Lemma~\ref{lem:RandomSameExp} (see \cite[Theorem 3.3]{StancevicFroyland} for an alternative argument) and are hence equal to 0, because of the assumption that $\Lam(it)=0$.  That is, for \paeom,

\begin{equation*}\label{eq:sameLimPiF}
\lim_{n\to\infty} \frac{1}{n} \log \Pi_{j=1}^d F^1_j(\mcl_\om^{it,(n)}|_{Y^{it}_{\om}}) = \lim_{n\to\infty} \frac{1}{n} \log \Pi_{j=1}^d F_j(\mcl_\om^{it,(n)}|_{Y^{it}_{\om}})=0.
\end{equation*}

Recall that for \paeom, $\mcl_{\om}^{it}: Y_{\om}^{it} \to Y_{\sig\om}^{it}$ is a bijection, so $(\mcl^{it}_\om|_{Y_\om^{it}})^{-1}$ is well defined.
Let $A_\ep:=\{  \om \in \Om: \| (\mcl^{it}_\om|_{Y_\om^{it}})^{-1} \|_1 >\frac{1}{1-\ep} \}$.
Since $ \| \mcl_\om^{it} v\|_1 \leq 1$ for every $v \in Y_\om^{it}\cap S_1$, we have that $F^1_j(\mcl_\om^{it})\leq 1$ for every $j$. Also, if $\om \in A_\ep$, then $F^1_d(\mcl_\om^{it}|_{Y^{it}_{\om}})<1-\ep$ and therefore $\Pi_{j=1}^d F^1_j(\mcl_\om^{it}|_{Y^{it}_{\om}})< 1-\ep$. Thus, for \paeom,
\begin{equation}\label{eq:PiFLeq0}
\begin{split}
0&=\lim_{n\to\infty} \frac{1}{n}  \log \Pi_{j=1}^d F^1_j(\mcl_\om^{it,(n)}|_{Y^{it}_{\om}})  =
\lim_{n\to\infty} \frac{1}{n} \log V^1_d(\mcl_\om^{it,(n)}|_{Y^{it}_{\om}})\\
& \leq \int \log  V^1_d(\mcl_\om^{it}|_{Y^{it}_{\om}}) d\bbp(\om)
 \leq C  \int \log \Pi_{j=1}^d F^1_j(\mcl_\om^{it}|_{Y^{it}_{\om}}) d\bbp(\om) \leq C\bbp(A_\ep) \log(1-\ep) \leq 0,
\end{split}
\end{equation}
where $C>0$ is such that for every $A: L^1 \to L^1$, $V^{1}_d(A)\leq C \Pi_{j=1}^d F^{1}_j (A)$, as guaranteed by Lemma~\ref{lem:equivVol}.
Thus,  all inequalities in \eqref{eq:PiFLeq0} must be equalities and therefore
 $\bbp(A_\ep)=0$ for every $\ep>0$, which means that $\|(\mcl^{it}_\om|_{Y_\om^{it}})^{-1}\|_1 =1$ for \paeom. Thus, $\inf_{v \in Y_\om^{it}\cap S_1} \|\mcl_\om^{it} v\|_1 =1$, as claimed.
\qed

\section{Regularity of $F$}\label{sec:regofF}
In this section, we establish regularity properties of the  map $F$ defined in~\eqref{defF}.
\subsection{First order regularity of $F$} \label{regofF}
Let $\mc{S}'$ be the Banach space of all functions $\mc{V} \colon  \Om \times \X \to \mathbb C$ such that $\mc{V}_\om:=\mc{V}(\om, \cdot)\in \BV$ and $\esssup_{\om \in \Om} \lVert \mc{V}_\om \rVert_{\BV}<\infty$.
Note that $\mc{S}$, defined in \eqref{defS}, consists of those $\mc{V} \in \mc{S}'$ such that $\int \mc{V}_\om \, dm=0$ for \paeom.
We define $G \colon B_{\mathbb C}(0, 1) \times \mathcal S \to \mathcal S'$ and $H \colon B_{\mathbb C} (0, 1) \times \mc{S} \to L^\infty (\Om)$ by
 \[
  G(\theta, \mc{W})_\om=\mathcal L_{\sigma^{-1} \om}^{\theta}(\mc{W}_{\sigma^{-1} \om}+v_{\sigma^{-1} \om}^0) \quad \text{and} \quad  H(\theta, \mc{W})(\om)=\int \mathcal L_{\sigma^{-1} \om}^{\theta}(\mc{W}_{\sigma^{-1} \om}+v_{\sigma^{-1} \om}^0) \, dm,
 \]
 where $v_\om^0$ is defined in \eqref{v0om}.
 It follows easily from Lemmas~\ref{lem:boundedv} and~\ref{1214} (together with~\eqref{ktheta} which implies $\sup_{\lvert \theta \rvert <1} K(\theta)<\infty$) that $G$ and $H$ are well-defined.
We are interested in showing that $G$ and $H$ are differentiable on a
neighborhood of $(0,0)$.

\begin{lemma}\label{var}
 We have that
 \[
  \var (e^{\theta g(\sigma^{-1} \om, \cdot)}) \le \lvert \theta \rvert e^{\lvert \theta \rvert M} \var (g(\sigma^{-1} \om, \cdot)), \quad \text{for $\om \in \Om$.}
 \]
\end{lemma}

\begin{proof}
 The desired claim follows directly from condition~(V9) of Section~\ref{sec:var} applied to $f=g(\sigma^{-1} \om, \cdot)$ and $h(z)=e^{\theta z}$.
\end{proof}

\begin{lemma}\label{L5}
There exists $C>0$ such that
\begin{equation}\label{X2}
\var (e^{\theta_1 g(\sigma^{-1} \om, \cdot)}-e^{\theta_2 g(\sigma^{-1} \om, \cdot)}) \le Ce^{\lvert \theta_1-\theta_2\rvert M}\lvert \theta_1-\theta_2\rvert (e^{\lvert \theta_2\rvert M}+\lvert \theta_2\rvert e^{\lvert \theta_2\rvert M}), \quad \text{for $\om \in \Om$.}
\end{equation}
\end{lemma}

\begin{proof}
 We note that it follows from~(V8) that
\[
\begin{split}
\var (e^{\theta_1 g(\sigma^{-1} \om, \cdot)}-e^{\theta_2 g(\sigma^{-1} \om, \cdot)}) &=\var (e^{\theta_2 g(\sigma^{-1} \om, \cdot)}(e^{(\theta_1-\theta_2) g(\sigma^{-1} \om, \cdot)}-1))\\
&\le \lVert e^{\theta_2 g(\sigma^{-1} \om, \cdot)} \rVert_{L^\infty} \cdot \var (e^{(\theta_1-\theta_2) g(\sigma^{-1} \om, \cdot)}-1) \\
&\phantom{\le}+\var (e^{\theta_2 g(\sigma^{-1} \om, \cdot)}) \cdot \lVert e^{(\theta_1-\theta_2) g(\sigma^{-1} \om, \cdot)}-1\rVert_{L^\infty}.
\end{split}
\]
Moreover, observe that it follows from~\eqref{obs} that  $\lVert e^{\theta_2 g(\sigma^{-1} \om, \cdot)}\rVert_{L^\infty} \le e^{\lvert \theta_2\rvert M}$. On the other hand, by applying~(V9) for $f=g(\sigma^{-1}\om, \cdot)$ and
$h(z)=e^{(\theta_1-\theta_2) z}-1$, we obtain
\[
\var (e^{(\theta_1-\theta_2) g(\sigma^{-1} \om, \cdot)}-1) \le \lvert \theta_1-\theta_2\rvert e^{\lvert \theta_1-\theta_2\rvert M} \var (g(\sigma^{-1} \om, \cdot)).
\]
 Finally, we want to estimate $\lVert e^{(\theta_1-\theta_2) g(\sigma^{-1} \om, \cdot)}-1\rVert_{L^\infty}$.
By applying the mean value theorem for the map $z\mapsto e^{(\theta_1-\theta_2)z}$, we have
that for each $x\in [0, 1]$,
\[
\lvert e^{(\theta_1-\theta_2) g(\sigma^{-1} \om, x)}-1\rvert \le  e^{\lvert \theta_1-\theta_2\rvert M}\lvert \theta_1-\theta_2\rvert \cdot \lvert g(\sigma^{-1} \om, x)\rvert \le Me^{\lvert \theta_1-\theta_2\rvert M}\lvert \theta_1-\theta_2\rvert,
\]
and consequently
\[
\lVert e^{(\theta_1-\theta_2) g(\sigma^{-1} \om, \cdot)}-1\rVert_{L^\infty} \le  Me^{\lvert \theta_1-\theta_2\rvert M}\lvert \theta_1-\theta_2\rvert.
\]
The conclusion of the lemma follows directly from the above estimates together with~\eqref{obs} and Lemma~\ref{var}.
\end{proof}

\begin{lemma}
 $D_2 G$ exists and is continuous on $B_{\mathbb C}(0, 1) \times \mathcal S$.
\end{lemma}

\begin{proof}
 Since $G$ is an affine map in the second variable $\mc W$, we conclude that
 \begin{equation}\label{D2G}
   (D_2G(\theta, \mc{W})\mc{H})_\om=\mathcal L^\theta_{\sigma^{-1} \om}\mc{H}_{\sigma^{-1} \om}, \quad \text{for $\om \in \Om$ and $\mc H \in \mathcal S$.}
\end{equation}
We now establish the continuity of $D_2 G$. Take an arbitrary $(\theta_i, \mc W^i) \in B_{\mathbb C}(0, 1) \times \mathcal S$, $i\in \{1, 2\}$. We have
\[
\begin{split}
 \lVert D_2 G(\theta_1, \mc{W}^1)-D_2 G(\theta_2, \mc{W}^2)\rVert &=\sup_{\lVert \mc{H}\rVert_\infty \le 1}\lVert D_2 G(\theta_1, \mc{W}^1)(\mc{H})-D_2 G(\theta_2, \mc{W}^2)(\mc{H})\rVert_\infty \\
 &=\sup_{\lVert \mc{H}\rVert_\infty \le 1}\esssup_{\om \in \Om}\lVert \mathcal L^{\theta_1}_{\sigma^{-1} \om}\mc{H}_{\sigma^{-1} \om}-\mathcal L^{\theta_2}_{\sigma^{-1} \om}\mc{H}_{\sigma^{-1} \om}\rVert_{\BV}.
 \end{split}
\]
Observe that
\[
\begin{split}
 \lVert \mathcal L^{\theta_1}_{\sigma^{-1} \om}\mc{H}_{\sigma^{-1} \om}-\mathcal L^{\theta_2}_{\sigma^{-1} \om}\mc{H}_{\sigma^{-1} \om}\rVert_{\BV} &=
 \lVert \mathcal L_{\sigma^{-1} \om}
( (e^{\theta_1 g(\sigma^{-1} \om, \cdot)}-e^{\theta_2 g(\sigma^{-1} \om, \cdot)})\mathcal H_{\sigma^{-1} \om})\rVert_{\BV} \\
&\le K \lVert (e^{\theta_1 g(\sigma^{-1} \om, \cdot)}-e^{\theta_2 g(\sigma^{-1} \om, \cdot)})\mathcal H_{\sigma^{-1} \om}\rVert_{\BV} \\
 &=K \var ((e^{\theta_1 g(\sigma^{-1} \om, \cdot)}-e^{\theta_2 g(\sigma^{-1} \om, \cdot)})\mathcal H_{\sigma^{-1} \om})\\
&\phantom{=}+K\lVert (e^{\theta_1 g(\sigma^{-1} \om, \cdot)}-e^{\theta_2 g(\sigma^{-1} \om, \cdot)})\mathcal H_{\sigma^{-1} \om}\rVert_1.
\end{split}
\]
Take an arbitrary $x\in \X$. By applying the mean value theorem for the map $z\mapsto e^{zg(\sigma^{-1} \om, x)}$ and using~\eqref{obs}, we conclude that
\begin{equation}\label{X1}
 \lvert e^{\theta_1 g(\sigma^{-1} \om, x)}-e^{\theta_2 g(\sigma^{-1} \om, x)}\rvert \le Me^M \lvert \theta_1-\theta_2\rvert
\end{equation}
and thus
 \begin{equation}\label{X3}
\begin{split}
 \esssup_{\om \in \Om}\lVert  (e^{\theta_1 g(\sigma^{-1} \om, \cdot)}-e^{\theta_2 g(\sigma^{-1} \om, \cdot)})\mathcal H_{\sigma^{-1} \om}\rVert_1 &
 \le Me^M  \lvert \theta_1-\theta_2\rvert   \esssup_{\om \in \Om}\lVert \mc{H}_{\sigma^{-1} \om}\rVert_1  \\
 &\le Me^M  \lvert \theta_1-\theta_2\rvert   \esssup_{\om \in \Om}\lVert \mc{H}_{\sigma^{-1} \om}\rVert_{\BV}  \\
 &\le Me^M \lVert \mc{H}\rVert_\infty \cdot \lvert \theta_1-\theta_2\rvert.
\end{split}
 \end{equation}
 Furthermore,
\[
\begin{split}
 var ((e^{\theta_1 g(\sigma^{-1} \om, \cdot)}-e^{\theta_2 g(\sigma^{-1} \om, \cdot)})\mathcal  H_{\sigma^{-1} \om}) &\le \var (e^{\theta_1 g(\sigma^{-1} \om, \cdot)}-e^{\theta_2 g(\sigma^{-1} \om, \cdot)})
 \cdot \lVert \mc H_{\sigma^{-1} \om}\rVert_{L^\infty} \\
 &\phantom{\le}+ \lVert e^{\theta_1 g(\sigma^{-1} \om, \cdot)}-e^{\theta_2 g(\sigma^{-1} \om, \cdot)}\rVert_{L^\infty} \cdot \var (\mc H_{\sigma^{-1} \om}),
 \end{split}
\]
which, using~\eqref{X1}, implies that
\begin{equation}\label{X4}
 var ((e^{\theta_1 g(\sigma^{-1} \om, \cdot)}-e^{\theta_2 g(\sigma^{-1} \om, \cdot)})\mathcal  H_{\sigma^{-1} \om}) \le \big(C_{var}\var (e^{\theta_1 g(\sigma^{-1} \om, \cdot)}-e^{\theta_2 g(\sigma^{-1} \om, \cdot)}) +Me^M\lvert \theta_1-\theta_2\rvert \big)\lVert \mc{H}\rVert_\infty.
\end{equation}
 It follows from Lemma~\ref{L5} that
 \[
  \lVert D_2 G(\theta_1, \mc{W}^1)-D_2 G(\theta_2, \mc{W}^2)\rVert \le (KC+2KMe^M)\lvert \theta_1-\theta_2 \rvert,
 \]
which implies (Lipschitz) continuity of $D_2G$ on $B_{\mathbb C}(0, 1) \times \mathcal S$.
\end{proof}

\begin{lemma}\label{1108}
 $D_2H$ exists and is continuous on a neighborhood of $(0, 0) \in \mathbb C \times \mc{S}$.
\end{lemma}

\begin{proof}
 We first note that $H$ is also an affine map in the variable $\mc W$ which implies that
 \begin{equation}\label{1001}
   (D_2H(\theta, \mc{W})\mc{H})(\om)=\int \mathcal L_{\sigma^{-1} \om}^\theta \mc H_{\sigma^{-1} \om} \, dm, \quad \text{for $\om \in \Om$ and $\mc H \in \mathcal S$.}
\end{equation}
Moreover, using~\eqref{X1} we have that
\[
\begin{split}
 \lVert D_2H(\theta_1, \mc{W}^1)\mc{H}-D_2H(\theta_2, \mc{W}^2)\mc{H}\rVert_{L^\infty} &=\esssup_{\om \in \Om} \bigg{\lvert} \int \mathcal L_{\sigma^{-1} \om}^{\theta_1} \mc{H}_{\sigma^{-1} \om} \, dm
 -\int \mathcal L_{\sigma^{-1} \om}^{\theta_2} \mc{H}_{\sigma^{-1} \om} \, dm \bigg{\rvert} \\
 &=\esssup_{\om \in \Om} \bigg{\lvert} \int (e^{\theta_1 g(\sigma^{-1} \om, \cdot)}-e^{\theta_2 g(\sigma^{-1} \om, \cdot)})\mc H_{\sigma^{-1} \om} \, dm \bigg{\rvert}\\
 &\le Me^M \lvert \theta_1-\theta_2\rvert \esssup_{\om \in \Om} \lVert H_{\sigma^{-1} \om} \rVert_1 \\
 &\le Me^M \lvert \theta_1-\theta_2\rvert \esssup_{\om \in \Om} \lVert H_{\sigma^{-1} \om} \rVert_{\BV} \\
 &= Me^M \lvert \theta_1-\theta_2\rvert \cdot \rVert \mc{H}\rVert_\infty,
 \end{split}
\]
for every $(\theta_1, \mc W^1)$, $(\theta_2, \mc W^2)$ that belong to a sufficiently small neighborhood of $(0, 0)$ on which $H$ is defined. We conclude that $D_2H$ is continuous.
\end{proof}

\begin{lemma}\label{L6}
$D_1H$ exists and is continuous on a neighborhood of $(0, 0) \in \mathbb C \times \mc{S}$.
\end{lemma}

\begin{proof}
  We first note that
 \[
  H(\theta, \mc{W})(\om)=\int e^{\theta g(\sigma^{-1} \om, \cdot)}(\mc{W}_{\sigma^{-1} \om}+v_{\sigma^{-1} \om}^0)  \, dm.
 \]
We claim that for $\om \in \Om$ and $h\in B_{\mathbb C} (0, 1)$,
\begin{equation}\label{D1H}
 (D_1 H(\theta, \mc W)h)(\om)=\int h g(\sigma^{-1} \om, \cdot) e^{\theta g(\sigma^{-1} \om, \cdot)}(\mc{W}_{\sigma^{-1} \om}+v_{\sigma^{-1} \om}^0)  \, dm=: (L(\theta, \mc W)h)_\om.
\end{equation}
Note that $L(\theta, \mc W): B_{\mathbb C} (0, 1) \to L^\infty (\Om)$ is a bounded linear operator. We first note that for each $\om \in \Om$,
\begin{align*}
 &  (H(\theta +h, \mc W)-H(\theta, \mc W)-L(\theta, \mc W)h)(\omega) \displaybreak[0] \\
 &=  \int (e^{(\theta +h) g(\sigma^{-1} \om, \cdot)}-e^{\theta g(\sigma^{-1} \om, \cdot)}-h g(\sigma^{-1} \om, \cdot) e^{\theta g(\sigma^{-1} \om, \cdot)})(\mc{W}_{\sigma^{-1} \om}+v_{\sigma^{-1} \om}^0)  \, dm.
\end{align*}
For each $\om \in \Om$ and $x\in \X$, it follows from Taylor's remainder theorem applied to the function $z\mapsto e^{zg(\sigma^{-1}\om, x)}$ that for $|\theta|, |h|\leq \frac12$,
\begin{equation}\label{mh2}
\lvert  e^{(\theta +h) g(\sigma^{-1} \om, x)}-e^{\theta g(\sigma^{-1} \om, x)}-h g(\sigma^{-1} \om, x) e^{\theta g(\sigma^{-1} \om, x)}\rvert \le \frac 12 M^2 e^M \lvert h\rvert^2.
 \end{equation}
 Hence,
 \[
   \lVert H(\theta +h, \mc W)-H(\theta, \mc W)-L(\theta, \mc W)h\rVert_{L^\infty} \le \frac 12 M^2 e^M \lvert h\rvert^2 (\lVert \mc W\rVert_\infty +\lVert v^0\rVert_\infty),
\]
and therefore
\[
 \frac{1}{\lvert h\rvert}\lVert H(\theta +h, \mc W)-H(\theta, \mc W)-L(\theta, \mc W)h\rVert_{L^\infty} \to 0, \quad \text{when $h\to 0$.}
\]
We conclude that~\eqref{D1H} holds. Furthermore,
\begin{align*}
& (D_1 H(\theta_1, \mc W^1)h)(\om)-(D_1 H(\theta_2, \mc W^2)h)(\om) \displaybreak[0] \\
&=\int h g(\sigma^{-1} \om, \cdot) e^{\theta_1 g(\sigma^{-1} \om, \cdot)}(\mc{W}_{\sigma^{-1} \om}^1+v_{\sigma^{-1} \om}^0)  \, dm \displaybreak[0] \\
&\phantom{=}-\int h g(\sigma^{-1} \om, \cdot) e^{\theta_2 g(\sigma^{-1} \om, \cdot)}(\mc{W}_{\sigma^{-1} \om}^2+v_{\sigma^{-1} \om}^0)  \, dm \displaybreak[0] \\
&=\int h g(\sigma^{-1} \om, \cdot) e^{\theta_1 g(\sigma^{-1} \om, \cdot)}(\mc{W}_{\sigma^{-1} \om}^1-\mc{W}_{\sigma^{-1} \om}^2)\, dm \displaybreak[0] \\
&\phantom{=}+\int h g(\sigma^{-1} \om, \cdot) (e^{\theta_1 g(\sigma^{-1} \om, \cdot)}-e^{\theta_2 g(\sigma^{-1} \om, \cdot)})(\mc{W}_{\sigma^{-1} \om}^2+v_{\sigma^{-1} \om}^0)  \, dm.
\end{align*}
Note that
\[
 \esssup_{\om \in \Om} \bigg{\lvert} \int h g(\sigma^{-1} \om, \cdot) e^{\theta_1 g(\sigma^{-1} \om, \cdot)}(\mc{W}_{\sigma^{-1} \om}^1-\mc{W}_{\sigma^{-1} \om}^2)\, dm \bigg{\rvert} \le
 \lvert h\rvert Me^M \lVert \mc W^1-\mc W^2\rVert_\infty
\]
and, using~\eqref{X1},
\begin{align*}
 & \esssup_{\om \in \Om} \bigg{\lvert}\int h g(\sigma^{-1} \om, \cdot) (e^{\theta_1 g(\sigma^{-1} \om, \cdot)}-e^{\theta_2 g(\sigma^{-1} \om, \cdot)})(\mc{W}_{\sigma^{-1} \om}^2+v_{\sigma^{-1} \om}^0)  \, dm
 \bigg{\rvert} \displaybreak[0] \\
 & \le \lvert h\rvert M^2 e^M \lvert \theta_1-\theta_2\rvert (R+\lVert v^0\rVert_\infty),
\end{align*}
if $\mc W^2\in B_{\mathcal S}(0, R)$. Hence,
\[
 \lVert D_1 H(\theta_1, \mc W^1)-D_1 H(\theta_2, \mc W^2)\rVert \le Me^M \lVert \mc W^1-\mc W^2\rVert_\infty+M^2 e^M \lvert \theta_1-\theta_2\rvert (R+\lVert v^0\rVert_\infty),
\]
which implies the continuity of $D_1 H$.
\end{proof}

\begin{lemma}
  $D_1G$ exists and is continuous on a neighborhood of $(0, 0) \in \mathbb C \times \mc{S}$.

\end{lemma}

\begin{proof}
  We claim that for $\om \in \Om$ and $t\in \mathbb C$,
\begin{equation}\label{D1G}
 (D_1 G(\theta, \mc W)t)_\om =\mathcal  L_{\sigma^{-1} \om}(tg(\sigma^{-1}\om ,\cdot)e^{\theta g(\sigma^{-1} \om ,\cdot)}(\mc W_{\sigma^{-1} \om}+v^0_{\sigma^{-1} \om}))=: (L(\theta, \mc W)t)_\om.
\end{equation}
Note that $L(\theta, \mc W): B_{\mathbb C} (0, 1) \to \mc{S'}$ is a bounded linear operator.
We note that
\begin{align*}
 & (G(\theta +t, \mc W)-G(\theta, \mc W)-L(\theta, \mc W)t)_\om \displaybreak[0] \\
 &=  \mathcal  L_{\sigma^{-1} \om}((e^{(\theta+t) g(\sigma^{-1} \om ,\cdot)}-e^{\theta g(\sigma^{-1} \om ,\cdot)}-tg(\sigma^{-1}\om ,\cdot)e^{\theta g(\sigma^{-1} \om ,\cdot)})(\mc W_{\sigma^{-1} \om}+v^0_{\sigma^{-1} \om})),
 \end{align*}
and therefore
\begin{align*}
 &\lVert (G(\theta +h, \mc W)-G(\theta, \mc W)-L(\theta, \mc W)h)_\om\rVert_{\BV} \displaybreak[0] \\
 & \le K\lVert (e^{(\theta+t) g(\sigma^{-1} \om ,\cdot)}-e^{\theta g(\sigma^{-1} \om ,\cdot)}-tg(\sigma^{-1}\om ,\cdot)e^{\theta g(\sigma^{-1} \om ,\cdot)})(\mc W_{\sigma^{-1} \om}+v^0_{\sigma^{-1} \om})
  \rVert_{\BV} \displaybreak[0] \\
   & =K\var ((e^{(\theta+t) g(\sigma^{-1} \om ,\cdot)}-e^{\theta g(\sigma^{-1} \om ,\cdot)}-tg(\sigma^{-1}\om ,\cdot)e^{\theta g(\sigma^{-1} \om ,\cdot)})(\mc W_{\sigma^{-1} \om}+v^0_{\sigma^{-1} \om}) )
  \displaybreak[0] \\
  &\phantom{=}   +K \lVert (e^{(\theta+t) g(\sigma^{-1} \om ,\cdot)}-e^{\theta g(\sigma^{-1} \om ,\cdot)}-tg(\sigma^{-1}\om ,\cdot)e^{\theta g(\sigma^{-1} \om ,\cdot)})(\mc W_{\sigma^{-1} \om}+v^0_{\sigma^{-1} \om})\rVert_1.
\end{align*}
In the proof of Lemma~\ref{L6} we have showed that
\[
 \lVert e^{(\theta+t) g(\sigma^{-1} \om ,\cdot)}-e^{\theta g(\sigma^{-1} \om ,\cdot)}-tg(\sigma^{-1}\om ,\cdot)e^{\theta g(\sigma^{-1} \om ,\cdot)}\rVert_{L^\infty} \le \frac 12 M^2 e^M \lvert t\rvert^2.
\]
Moreover, by applying~(V9) for $f=g(\sigma^{-1} \om, \cdot)$ and
\[
 h(z)=e^{(\theta+t)z}-e^{\theta z}-tze^{\theta z},
\]
one can  conclude  that
\begin{equation}\label{t}
 \var ((e^{(\theta+t) g(\sigma^{-1} \om ,\cdot)}-e^{\theta g(\sigma^{-1} \om ,\cdot)}-tg(\sigma^{-1}\om ,\cdot)e^{\theta g(\sigma^{-1} \om ,\cdot)}) \le C\lvert t\rvert^2.
\end{equation}
The last two inequalities combined with (V8) readily imply that
\[
 \frac{1}{\lvert t\rvert} \lVert G(\theta +t, \mc W)-G(\theta, \mc W)-L(\theta, \mc W)t\rVert_\infty \to 0, \quad \text{when $t\to 0$,}
\]
which implies~\eqref{D1G}. Moreover,
\begin{align*}
 & (D_1 G(\theta_1, \mc W^1)t -D_1 G(\theta_2, \mc W^2)t)_\om \displaybreak[0] \\
 &=t\mathcal L_{\sigma^{-1} \om}(g(\sigma^{-1}\om ,\cdot)(e^{\theta_1 g(\sigma^{-1} \om ,\cdot)}-e^{\theta_2 g(\sigma^{-1} \om ,\cdot)})(\mc W_{\sigma^{-1} \om}^1+v^0_{\sigma^{-1} \om}))\displaybreak[0] \\
 &\phantom{=}-t\mathcal  L_{\sigma^{-1} \om}(g(\sigma^{-1}\om ,\cdot)e^{\theta_2 g(\sigma^{-1} \om ,\cdot)}(\mc W_{\sigma^{-1} \om}^2-\mc W_{\sigma^{-1} \om}^1)).
\end{align*}
Proceeding as in the previous lemmas and using~\eqref{X1} and Lemma~\ref{L5} together with a simple observation that
\[
\begin{split}
 \var (g(\sigma^{-1}\om ,\cdot)(e^{\theta_1 g(\sigma^{-1} \om ,\cdot)}-e^{\theta_2 g(\sigma^{-1} \om ,\cdot)})) &\le M\var (e^{\theta_1 g(\sigma^{-1} \om ,\cdot)}-e^{\theta_2 g(\sigma^{-1} \om ,\cdot)})
\\
&\phantom{\le} +\var (g(\sigma^{-1}\om ,\cdot))\lVert e^{\theta_1 g(\sigma^{-1} \om ,\cdot)}-e^{\theta_2 g(\sigma^{-1} \om ,\cdot)}\rVert_{L^\infty},
\end{split}
 \]
we easily obtain the continuity of $D_1G$.
\end{proof}

The following result is a direct consequence of the previous lemmas.

\begin{proposition}\label{difF}
The  map $F$ defined by~\eqref{defF} is of class $C^1$ on a neighborhood $(0, 0)\in \mathbb C \times \mc{S}$. Moreover,
 \[
  (D_2 F(\theta, \mc W) \mc H)_\om=\frac{1}{H(\theta, \mc W)(\om)}\mathcal L_{\sigma^{-1} \om}^\theta \mc H_{\sigma^{-1} \om}-\frac{\int \mathcal L_{\sigma^{-1} \om}^\theta
  \mc H_{\sigma^{-1} \om}\, dm}{[H(\theta, \mc W)(\om)]^2}G(\theta, \mc W)_\om-\mc H_\om,
 \]
for $\om \in \Om$ and $\mc H \in \mathcal S$ and
\[
 \begin{split}
 (D_1 F(\theta, \mc W))_\om &=\frac{1}{H(\theta, \mc W)(\om)}\mathcal L_{\sigma^{-1} \om}(g(\sigma^{-1} \om, \cdot)e^{\theta  g(\sigma^{-1} \om, \cdot)} (\mc W_{\sigma^{-1} \om}+
 v_{\sigma^{-1} \om}^0)) \\
 &\phantom{=}-\frac{\int g(\sigma^{-1} \om, \cdot)e^{\theta  g(\sigma^{-1} \om, \cdot)} (\mc W_{\sigma^{-1} \om}+
 v_{\sigma^{-1} \om}^0)\, dm}{[H(\theta, \mc W)(\om)]^2}\mathcal L_{\sigma^{-1} \om}^\theta (\mc W_{\sigma^{-1} \om}+v_{\sigma^{-1} \om}^0),
 \end{split}
\]
for $\om \in \Om$, where we have identified $D_1 F(\theta, \mc W)$ with its value at $1$, and $G$ is as defined at the beginning of Section \ref{regofF}.

\end{proposition}
\subsection{Second order regularity of $F$}\label{regofF2}

\begin{lemma}
 $D_{12}H$ and $D_{22}H$ exist and are continuous on a neighborhood of $(0,0)\in \mathbb C \times \mc{S}$.
\end{lemma}

\begin{proof}
 We first note that it follows directly from~\eqref{1001} that $D_{22} H=0$. We claim that
 \begin{equation}\label{D12H}
 (( D_{12}H(\theta, \mc W)h)\mc H)(\om )=h\int g(\sigma^{-1} \om, \cdot) e^{\theta g(\sigma^{-1} \om, \cdot)} \mc H_{\sigma^{-1} \om} \, dm, \quad \text{for $\om \in \Om$, $\mc H\in \mathcal S$ and
 $h\in \mathbb C$.}
\end{equation}
Indeed, we note that
\[
 ((D_2 H(\theta+h, \mc W)-D_2 H(\theta, \mc W)) \mc H)(\om)=\int (e^{(\theta +h)g(\sigma^{-1} \om, \cdot)}-e^{\theta g(\sigma^{-1} \om, \cdot)} )\mc H_{\sigma^{-1}\om}\, dm.
\]
Hence, using~\eqref{mh2},
\begin{align*}
  & \bigg{\lvert} ((D_2 H(\theta+h, \mc W)-D_2 H(\theta, \mc W)) \mc H)(\om)-h\int g(\sigma^{-1} \om, \cdot) e^{\theta g(\sigma^{-1} \om, \cdot)} \mc H_{\sigma^{-1} \om} \, dm\bigg{\rvert} \displaybreak[0] \\
  &= \bigg{\lvert}\int (e^{(\theta +h)g(\sigma^{-1} \om, \cdot)}-e^{\theta g(\sigma^{-1} \om, \cdot)} -h g(\sigma^{-1} \om, \cdot) e^{\theta g(\sigma^{-1} \om, \cdot)} )\mc H_{\sigma^{-1}\om}\, dm \bigg{\rvert}
\displaybreak[0] \\
&\le \frac 12 M^2e^M \lvert h\rvert^2 \lVert \mc H_{\sigma^{-1}\om}\rVert_1 \displaybreak[0]  \le \frac 12 M^2e^M \lvert h\rvert^2 \lVert \mc H_{\sigma^{-1}\om}\rVert_{\BV}.
  \end{align*}
Thus,
  \begin{align*}
 & \esssup_{\om \in \Om} \bigg{\lvert} ((D_2 H(\theta+h, \mc W)-D_2 H(\theta, \mc W)) \mc H)(\om)-h\int g(\sigma^{-1} \om, \cdot) e^{\theta g(\sigma^{-1} \om, \cdot)} \mc H_{\sigma^{-1} \om} \, dm\bigg{\rvert}
 \displaybreak[0] \\
 &\le \frac 12 M^2e^M \lvert h\rvert^2 \lVert \mc H\rVert_\infty,
\end{align*}
which readily implies~\eqref{D12H}. We now  establish the continuity of $D_{12}H$. By~\eqref{X1}, we have that
\begin{align*}
& \lvert  (( D_{12}H(\theta_1, \mc W^1)h)\mc H)(\om ) - (( D_{12}H(\theta_2, \mc W^2)h)\mc H)(\om )\rvert
\displaybreak[0] \\
&=\lvert h\rvert  \bigg{\lvert} \int  g(\sigma^{-1} \om, \cdot)(e^{\theta_1 g(\sigma^{-1} \om, \cdot)}
 -e^{\theta_2 g(\sigma^{-1} \om, \cdot)})\mc H_{\sigma^{-1}\om}\, dm \bigg{\rvert} \\
 &\le \lvert h \rvert M^2 e^M \lvert \theta_1-\theta_2 \rvert \cdot \lVert \mc H_{\sigma^{-1} \om }\rVert_{\BV}.
\end{align*}
Thus,
\[
 \lVert D_{12}H(\theta_1, \mc W^1)-D_{12}H(\theta_2, \mc W^2)\rVert \le M^2 e^M \lvert \theta_1-\theta_2 \rvert,
\]
which implies the continuity of $D_{12}H$.
\end{proof}

\begin{lemma}
 $D_{11}H$ and $D_{21}H$ exist and are continuous on a neighborhood of $(0,0) \in \mathbb C \times \mc{S}$.
\end{lemma}

\begin{proof}
 By identifying $D_1 H(\theta, \mc W)$ with its value in $1$, it follows from~\eqref{D1H} that
 \[
   (D_1 H(\theta, \mc W))(\om)=\int  g(\sigma^{-1} \om, \cdot) e^{\theta g(\sigma^{-1} \om, \cdot)}(\mc{W}_{\sigma^{-1} \om}+v_{\sigma^{-1} \om}^0)  \, dm.
\]
We claim that
\begin{equation}\label{D11H}
 (D_{11}H(\theta, \mc W) h)(\om)=h\int  g(\sigma^{-1} \om, \cdot)^2 e^{\theta g(\sigma^{-1} \om, \cdot)}(\mc{W}_{\sigma^{-1} \om}+v_{\sigma^{-1} \om}^0)  \, dm,
 \quad \text{for $\om \in \Om$ and $h\in \mathbb C$.}
\end{equation}
Indeed, observe that
\begin{align*}
 & (D_1 H(\theta+h, \mc W))(\om)- (D_1 H(\theta, \mc W))(\om) \displaybreak[0] \\
 &=\int  g(\sigma^{-1} \om, \cdot) (e^{(\theta +h) g(\sigma^{-1} \om, \cdot)}-e^{\theta g(\sigma^{-1} \om, \cdot)})(\mc{W}_{\sigma^{-1} \om}+v_{\sigma^{-1} \om}^0)  \, dm.
\end{align*}
Hence, using~\eqref{mh2}, we obtain that
\begin{align*}
 &\esssup_{\om \in \Om}\bigg{\lvert} (D_1 H(\theta+h, \mc W))(\om)- (D_1 H(\theta, \mc W))(\om)-h\int  g(\sigma^{-1} \om, \cdot)^2 e^{\theta g(\sigma^{-1} \om, \cdot)}(\mc{W}_{\sigma^{-1} \om}+v_{\sigma^{-1} \om}^0)  \, dm \bigg{\rvert} &\\
\displaybreak[0] \\
&=\esssup_{\om \in \Om} \bigg{\lvert} \int g(\sigma^{-1} \om , \cdot)(e^{(\theta +h) g(\sigma^{-1} \om, \cdot)}-e^{\theta g(\sigma^{-1} \om, \cdot)}-h g(\sigma^{-1} \om, \cdot)e^{\theta g(\sigma^{-1} \om, \cdot)})(\mc{W}_{\sigma^{-1} \om}+v_{\sigma^{-1} \om}^0)  \, dm \bigg{\rvert} &\\
\displaybreak[0] \\
&\le  \frac{1}{2} M^3e^M \lvert h\rvert^2 (\lVert \mc W\rVert_\infty +\lVert v^0\rVert_\infty),
\end{align*}
which readily implies that~\eqref{D11H} holds. We now establish the continuity of $D_{11}H$. It follows from~\eqref{X1} that
\begin{align*}
 &\lvert (D_{11}H(\theta_1, \mc W^1) h)(\om)-(D_{11}H(\theta_2, \mc W^2) h)(\om)\rvert \displaybreak[0] \\
 &=\lvert h\rvert \cdot  \bigg{\lvert} \int  g(\sigma^{-1} \om, \cdot)^2 e^{\theta_1 g(\sigma^{-1} \om, \cdot)}(\mc{W}_{\sigma^{-1} \om}^1+v_{\sigma^{-1} \om}^0)  \, dm-\int  g(\sigma^{-1} \om, \cdot)^2 e^{\theta_2 g(\sigma^{-1} \om, \cdot)}(\mc{W}_{\sigma^{-1} \om}^2+v_{\sigma^{-1} \om}^0)  \, dm
\bigg{\rvert} \displaybreak[0] \\
&\le \lvert h\rvert \cdot \bigg{\lvert} \int  g(\sigma^{-1} \om, \cdot)^2 e^{\theta_1 g(\sigma^{-1} \om, \cdot)}(\mc{W}_{\sigma^{-1} \om}^1+v_{\sigma^{-1} \om}^0)  \, dm-\int  g(\sigma^{-1} \om, \cdot)^2 e^{\theta_2 g(\sigma^{-1} \om, \cdot)}(\mc{W}_{\sigma^{-1} \om}^1+v_{\sigma^{-1} \om}^0)  \, dm
\bigg{\rvert}\displaybreak[0] \\
&\phantom{\le}+\lvert h\rvert \cdot \bigg{\lvert}\int  g(\sigma^{-1} \om, \cdot)^2 e^{\theta_2 g(\sigma^{-1} \om, \cdot)}(\mc{W}_{\sigma^{-1} \om}^1+v_{\sigma^{-1} \om}^0)  \, dm - \int  g(\sigma^{-1} \om, \cdot)^2 e^{\theta_2 g(\sigma^{-1} \om, \cdot)}(\mc{W}_{\sigma^{-1} \om}^2+v_{\sigma^{-1} \om}^0)  \, dm \bigg{\rvert}
\displaybreak[0] \\
&\le  \lvert h\rvert \cdot \bigg{(} M^3 e^M \lvert \theta_1-\theta_2\rvert (\lVert \mc W^1\rVert_\infty+\lVert v^0\rVert_\infty)+M^2e^M \lVert \mc W^1-\mc W^2\rVert_\infty \bigg{)},
\end{align*}
for \paeom, which implies the continuity of $D_{11}H$. Furthermore, we note that $D_1H$ is affine in $\mc W$, which implies that
\[
(D_{21}H(\theta, \mc W) \mc H)(\om)=\int  g(\sigma^{-1} \om, \cdot) e^{\theta g(\sigma^{-1} \om, \cdot)}\mc{H}_{\sigma^{-1} \om} \, dm.
\]
Continuity of $D_{21}H$ follows easily  from~\eqref{X1}.
\end{proof}

\begin{lemma}
 $D_{22} G$ and $D_{12}G$ exist and are continuous on a neighborhood of $(0,0) \in \mathbb C \times \mc{S}$.
\end{lemma}

\begin{proof}
 It follows directly from~\eqref{D2G} that $D_{22}G=0$.  We claim that
\begin{equation}\label{D12G}
 (D_{12}G(\theta, \mc W)h (\mc H))_{\om}=h \mathcal L_{\sigma^{-1} \om}(g(\sigma^{-1} \om, \cdot)e^{\theta g(\sigma^{-1} \om, \cdot)} \mc H_{\sigma^{-1} \om}) \quad
 \text{for $\om \in \Om$, $\mc H \in \mathcal S$ and $h\in \mathbb C$.}
\end{equation}
Indeed, we first note that
\[
 (D_2 G(\theta+h, \mc W)-D_2 G(\theta, \mc W))(\mc H)_\om=\mathcal L_{\sigma^{-1} \om}((e^{(\theta+h)g(\sigma^{-1} \om, \cdot)}-e^{\theta g(\sigma^{-1} \om, \cdot)})\mc H_{\sigma^{-1} \om}).
\]
We have  that
\begin{align*}
 &\lVert  \mathcal L_{\sigma^{-1} \om} ((e^{(\theta+h)g(\sigma^{-1} \om, \cdot)}-e^{\theta g(\sigma^{-1} \om, \cdot)}-hg(\sigma^{-1} \om, \cdot)e^{\theta g(\sigma^{-1} \om, \cdot)})\mc H_{\sigma^{-1} \om}) \rVert_{\BV} \displaybreak[0] \\
 & \le K \lVert (e^{(\theta+h)g(\sigma^{-1} \om, \cdot)}-e^{\theta g(\sigma^{-1} \om, \cdot)}-hg(\sigma^{-1} \om, \cdot)e^{\theta g(\sigma^{-1} \om, \cdot)})\mc H_{\sigma^{-1} \om} \rVert_{\BV} \displaybreak[0] \\
 &= K \var ((e^{(\theta+h)g(\sigma^{-1} \om, \cdot)}-e^{\theta g(\sigma^{-1} \om, \cdot)}-hg(\sigma^{-1} \om, \cdot)e^{\theta g(\sigma^{-1} \om, \cdot)})\mc H_{\sigma^{-1} \om}) \displaybreak[0] \\
 &\phantom{=}+K \lVert (e^{(\theta+h)g(\sigma^{-1} \om, \cdot)}-e^{\theta g(\sigma^{-1} \om, \cdot)}-hg(\sigma^{-1} \om, \cdot)e^{\theta g(\sigma^{-1} \om, \cdot)})\mc H_{\sigma^{-1} \om}\rVert_1 \displaybreak[0] \\
 &\le K\var (e^{(\theta+h)g(\sigma^{-1} \om, \cdot)}-e^{\theta g(\sigma^{-1} \om, \cdot)}-hg(\sigma^{-1} \om, \cdot)e^{\theta g(\sigma^{-1} \om, \cdot)}) \cdot \lVert \mc H_{\sigma^{-1} \om}\rVert_{L^\infty} \displaybreak[0] \\
 &\phantom{\le}+K\lVert e^{(\theta+h)g(\sigma^{-1} \om, \cdot)}-e^{\theta g(\sigma^{-1} \om, \cdot)}-hg(\sigma^{-1} \om, \cdot)e^{\theta g(\sigma^{-1} \om, \cdot)}\rVert_{L^\infty} \cdot \var ( \mc H_{\sigma^{-1} \om})\displaybreak[0] \\
 &\phantom{\le}+K\lVert e^{(\theta+h)g(\sigma^{-1} \om, \cdot)}-e^{\theta g(\sigma^{-1} \om, \cdot)}-hg(\sigma^{-1} \om, \cdot)e^{\theta g(\sigma^{-1} \om, \cdot)}\rVert_{L^\infty} \cdot \lVert \mc H_{\sigma^{-1} \om} \rVert_1
\end{align*}
It follows from~\eqref{mh2} and~\eqref{t} that
\[
 \frac{1}{\lvert h\rvert} \sup_{\lVert \mc H\rVert_\infty \le 1}\lVert  \mathcal L_{\sigma^{-1} \om} ((e^{(\theta+h)g(\sigma^{-1} \om, \cdot)}-e^{\theta g(\sigma^{-1} \om, \cdot)}-hg(\sigma^{-1} \om, \cdot)e^{\theta g(\sigma^{-1} \om, \cdot)})\mc H_{\sigma^{-1} \om}) \rVert_{\BV}
\to 0,
 \]
when $h\to 0$, which establishes~\eqref{D12G}. It remains to establish the continuity of $D_{12} G$.  We have
\begin{align*}
 & \lVert  (D_{12}G(\theta_1, \mc W^1)h (\mc H))_{\om}-(D_{12}G(\theta_2, \mc W^2)h (\mc H))_{\om}\rVert_{\BV} \displaybreak[0] \\
 &= \lvert h\rvert \cdot \lVert \mathcal L_{\sigma^{-1} \om}(g(\sigma^{-1} \om, \cdot)(e^{\theta_1 g(\sigma^{-1} \om, \cdot)} -e^{\theta_2 g(\sigma^{-1} \om, \cdot)})\mc H_{\sigma^{-1} \om})\rVert_{\BV} \displaybreak[0] \\
 &\le K\lvert h\rvert \cdot \lVert g(\sigma^{-1} \om, \cdot)(e^{\theta_1 g(\sigma^{-1} \om, \cdot)} -e^{\theta_2 g(\sigma^{-1} \om, \cdot)})\mc H_{\sigma^{-1} \om}\rVert_{\BV}\displaybreak[0] \\
 &=K\lvert h\rvert \cdot \var (g(\sigma^{-1} \om, \cdot)(e^{\theta_1 g(\sigma^{-1} \om, \cdot)} -e^{\theta_2 g(\sigma^{-1} \om, \cdot)})\mc H_{\sigma^{-1} \om}) \displaybreak[0] \\
 &\phantom{=}+K\lvert h\rvert \cdot \lVert g(\sigma^{-1} \om, \cdot)(e^{\theta_1 g(\sigma^{-1} \om, \cdot)} -e^{\theta_2 g(\sigma^{-1} \om, \cdot)})\mc H_{\sigma^{-1} \om}\rVert_1 \displaybreak[0] \\
 &\le K\lvert h\rvert \cdot \var (g(\sigma^{-1} \om, \cdot)(e^{\theta_1 g(\sigma^{-1} \om, \cdot)} -e^{\theta_2 g(\sigma^{-1} \om, \cdot)})) \cdot \lVert \mc H_{\sigma^{-1} \om} \rVert_1 \displaybreak[0] \\
 &\phantom{\le}+KM \lvert h\rvert \cdot \lVert e^{\theta_1 g(\sigma^{-1} \om, \cdot)} -e^{\theta_2 g(\sigma^{-1} \om, \cdot)}\rVert_{L^\infty} \cdot \var(\mc H_{\sigma^{-1} \om}) \displaybreak[0] \\
 &\phantom{\le}+KM \lvert h\rvert \cdot \lVert e^{\theta_1 g(\sigma^{-1} \om, \cdot)} -e^{\theta_2 g(\sigma^{-1} \om, \cdot)}\rVert_{L^\infty} \cdot \lVert \mc H_{\sigma^{-1} \om} \rVert_1.
 \end{align*}
 Moreover,
\[
 \begin{split}
  \var (g(\sigma^{-1} \om, \cdot)(e^{\theta_1 g(\sigma^{-1} \om, \cdot)} -e^{\theta_2 g(\sigma^{-1} \om, \cdot)})) &\le M\var (e^{\theta_1 g(\sigma^{-1} \om, \cdot)} -e^{\theta_2 g(\sigma^{-1} \om, \cdot)}) \\
  &\phantom{\le}+\var (g(\sigma^{-1} \om, \cdot))\cdot \lVert e^{\theta_1 g(\sigma^{-1} \om, \cdot)} -e^{\theta_2 g(\sigma^{-1} \om, \cdot)}\rVert_{L^\infty},
 \end{split}
\]
which together with~\eqref{X1} and Lemma~\ref{L5} gives the continuity of $D_{12}G$.
\end{proof}

\begin{lemma}
 $D_{11}G$ and $D_{21}G$ exist and are continuous on a neighborhood of $(0,0) \in \mathbb C \times \mc{S}$.

\end{lemma}

\begin{proof}
 By identifying $D_1 G(\theta, \mc W)$ with its value in $1$, it follows from~\eqref{D1G} that
 \[
   D_1 G(\theta, \mc W)_{\om }=\mathcal  L_{\sigma^{-1} \om}(g(\sigma^{-1}\om ,\cdot)e^{\theta g(\sigma^{-1} \om ,\cdot)}(\mc W_{\sigma^{-1} \om}+v^0_{\sigma^{-1} \om})), \quad \om \in \Om.
 \]
We claim that
\begin{equation}\label{D11G}
 (D_{11} G(\theta, \mc W)h)_\om=h \mathcal  L_{\sigma^{-1} \om}(g(\sigma^{-1}\om ,\cdot)^2e^{\theta g(\sigma^{-1} \om ,\cdot)}(\mc W_{\sigma^{-1} \om}+v^0_{\sigma^{-1} \om})).
\end{equation}
Indeed, we have
\begin{align*}
  & \lVert  D_1 G(\theta+h, \mc W)_{\om }- D_1 G(\theta, \mc W)_{\om }-h \mathcal  L_{\sigma^{-1} \om}(g(\sigma^{-1}\om ,\cdot)^2e^{\theta g(\sigma^{-1} \om ,\cdot)}(\mc W_{\sigma^{-1} \om}+v^0_{\sigma^{-1} \om}))
  \rVert_{\BV} \displaybreak[0] \\
  & =\lVert \mathcal  L_{\sigma^{-1} \om} (g(\sigma^{-1}\om ,\cdot) (e^{(\theta+h)g(\sigma^{-1} \om ,\cdot)}-e^{\theta g(\sigma^{-1} \om ,\cdot)}-hg(\sigma^{-1}\om ,\cdot)e^{\theta g(\sigma^{-1} \om ,\cdot)})
 (\mc W_{\sigma^{-1} \om}+v^0_{\sigma^{-1} \om}))\rVert_{\BV} \displaybreak[0] \\
 &\le K \lVert g(\sigma^{-1}\om ,\cdot) (e^{(\theta+h)g(\sigma^{-1} \om ,\cdot)}-e^{\theta g(\sigma^{-1} \om ,\cdot)}-hg(\sigma^{-1}\om ,\cdot)e^{\theta g(\sigma^{-1} \om ,\cdot)})
 (\mc W_{\sigma^{-1} \om}+v^0_{\sigma^{-1} \om})\rVert_{\BV} \displaybreak[0] \\
 &=\var (g(\sigma^{-1}\om ,\cdot) (e^{(\theta+h)g(\sigma^{-1} \om ,\cdot)}-e^{\theta g(\sigma^{-1} \om ,\cdot)}-hg(\sigma^{-1}\om ,\cdot)e^{\theta g(\sigma^{-1} \om ,\cdot)})
 (\mc W_{\sigma^{-1} \om}+v^0_{\sigma^{-1} \om})) \displaybreak[0] \\
 &\phantom{=}+\lVert g(\sigma^{-1}\om ,\cdot) (e^{(\theta+h)g(\sigma^{-1} \om ,\cdot)}-e^{\theta g(\sigma^{-1} \om ,\cdot)}-hg(\sigma^{-1}\om ,\cdot)e^{\theta g(\sigma^{-1} \om ,\cdot)})
 (\mc W_{\sigma^{-1} \om}+v^0_{\sigma^{-1} \om})\rVert_1 \displaybreak[0] \\
 &\le \var (g(\sigma^{-1}\om ,\cdot) (e^{(\theta+h)g(\sigma^{-1} \om ,\cdot)}-e^{\theta g(\sigma^{-1} \om ,\cdot)}-hg(\sigma^{-1}\om ,\cdot)e^{\theta g(\sigma^{-1} \om ,\cdot)})) \cdot
 \lVert \mc W_{\sigma^{-1} \om}+v^0_{\sigma^{-1} \om}\rVert_{L^\infty} \displaybreak[0] \\
 &\phantom{\le}+\lVert g(\sigma^{-1}\om ,\cdot) (e^{(\theta+h)g(\sigma^{-1} \om ,\cdot)}-e^{\theta g(\sigma^{-1} \om ,\cdot)}-hg(\sigma^{-1}\om ,\cdot)e^{\theta g(\sigma^{-1} \om ,\cdot)})\rVert_{L^\infty}
 \cdot \var (\mc W_{\sigma^{-1} \om}+v^0_{\sigma^{-1} \om}) \displaybreak[0] \\
 &\phantom{\le}+M\lVert e^{(\theta+h)g(\sigma^{-1} \om ,\cdot)}-e^{\theta g(\sigma^{-1} \om ,\cdot)}-hg(\sigma^{-1}\om ,\cdot)e^{\theta g(\sigma^{-1} \om ,\cdot)}\rVert_{L^\infty} \cdot
 \lVert \mc W_{\sigma^{-1} \om}+v^0_{\sigma^{-1} \om}\rVert_1,
 \end{align*}
and therefore~\eqref{D11G} follows directly from~\eqref{mh2} and~\eqref{t}. We now establish the continuity of $D_{11}G$. Observe that
\begin{align*}
  & \lVert  (D_{11} G(\theta_1, \mc W^1)h)_\om -(D_{11} G(\theta_2, \mc W^2)h)_\om \rVert_{\BV} \displaybreak[0] \\
  &=\lvert h\rvert \cdot \lVert \mathcal  L_{\sigma^{-1} \om}(g(\sigma^{-1}\om ,\cdot)^2e^{\theta_1 g(\sigma^{-1} \om ,\cdot)}(\mc W_{\sigma^{-1} \om}^1+v^0_{\sigma^{-1} \om})-
   g(\sigma^{-1}\om ,\cdot)^2e^{\theta_2 g(\sigma^{-1} \om ,\cdot)}(\mc W_{\sigma^{-1} \om}^2+v^0_{\sigma^{-1} \om}))\rVert_{\BV} \displaybreak[0] \\
   & \le K\lvert h\rvert \cdot \lVert g(\sigma^{-1}\om ,\cdot)^2e^{\theta_1 g(\sigma^{-1} \om ,\cdot)}(\mc W_{\sigma^{-1} \om}^1+v^0_{\sigma^{-1} \om})-
   g(\sigma^{-1}\om ,\cdot)^2e^{\theta_2 g(\sigma^{-1} \om ,\cdot)}(\mc W_{\sigma^{-1} \om}^2+v^0_{\sigma^{-1} \om})\rVert_{\BV} \displaybreak[0] \\
   &\le K\lvert h\rvert \cdot \lVert g(\sigma^{-1}\om ,\cdot)^2e^{\theta_1 g(\sigma^{-1} \om ,\cdot)}(\mc W_{\sigma^{-1} \om}^1+v^0_{\sigma^{-1} \om})-
   g(\sigma^{-1}\om ,\cdot)^2e^{\theta_2 g(\sigma^{-1} \om ,\cdot)}(\mc W_{\sigma^{-1} \om}^1+v^0_{\sigma^{-1} \om})\rVert_{\BV}\displaybreak[0] \\
   &\phantom{\le} +K\lvert h\rvert \cdot \lVert g(\sigma^{-1}\om ,\cdot)^2e^{\theta_2 g(\sigma^{-1} \om ,\cdot)}(\mc W_{\sigma^{-1} \om}^1+v^0_{\sigma^{-1} \om})-
   g(\sigma^{-1}\om ,\cdot)^2e^{\theta_2 g(\sigma^{-1} \om ,\cdot)}(\mc W_{\sigma^{-1} \om}^2+v^0_{\sigma^{-1} \om})\rVert_{\BV}.
 \end{align*}
The continuity of $D_{11}G$ now follows  easily from~\eqref{X1} and Lemma~\ref{L5}.
Finally, we note that $D_1 G$ is an affine map in $\mc W$ and therefore
\[
    (D_{21} G(\theta, \mc W) \mc H)_{\om }=\mathcal  L_{\sigma^{-1} \om}(g(\sigma^{-1}\om ,\cdot)e^{\theta g(\sigma^{-1} \om ,\cdot)}\mc H_{\sigma^{-1} \om} ),
\]
which can be showed to be  continuous by using~\eqref{X1} and Lemma~\ref{L5} again.
\end{proof}
The following result is a direct consequence of the previous lemmas.
\begin{proposition}\label{prop:F-C2}
The function  $F$ defined by~\eqref{defF} is of class $C^2$ on a neighborhood $(0, 0) \in \mathbb C \times \mc{S}$.
\end{proposition}

\section{Differentiability of  $\phi^\theta$, the top space for adjoint twisted cocycle $\mc{R}^{\theta *}$}\label{0245}

We begin with some auxiliary results.
\begin{lemma}
 There exists $C>0$ such that
 \begin{equation}\label{DAC}
  \lVert \mathcal L_\om^{\theta_1}-\mathcal L_{\om}^{\theta_2} \rVert \le C \lvert \theta_1-\theta_2 \lvert, \quad \text{for $\theta_1, \theta_2 \in B_{\mathbb C}(0, 1)$ and $\om \in \Om$.}
 \end{equation}

\end{lemma}

\begin{proof}
 For any $f\in \BV$ we have that
 \[
  \begin{split}
   \lVert \mathcal L_\om^{\theta_1}f-\mathcal L_{\om}^{\theta_2}f \rVert_{\BV}
   &=\lVert \mathcal L_\om (e^{\theta_1 g(\om, \cdot)}f-e^{\theta_2 g( \om, \cdot)}f)\rVert_{\BV}
   \le K\lVert (e^{\theta_1 g( \om, \cdot)}-e^{\theta_2 g( \om, \cdot)})f\rVert_{\BV}\\
   &=K\var ((e^{\theta_1 g( \om, \cdot)}-e^{\theta_2 g( \om, \cdot)})f)
   +K\lVert (e^{\theta_1 g( \om, \cdot)}-e^{\theta_2 g( \om, \cdot)})f\rVert_1
  \end{split}
\]
The claim of the lemma now follows directly from~\eqref{X2} and~\eqref{X1}.
\end{proof}
\begin{lemma}\label{lemxv}
The following statements hold:
\begin{enumerate}
\item There exists $K''>0$ such that
\begin{equation}\label{450}
\lVert \mcl_\om^{*, (n)} \phi \rVert_{\BV^*} \le K''e^{-\lambda n} \lVert \phi \rVert_{\BV^*} \quad \text{for $\phi \in \BV^*$ such that $\phi (v_\om^0)=0$ and $\om \in \Om$,}
\end{equation}
with $\lambda >0$  as in~\ref{cond:dec};
\item
Let $\phi_\om^0 \in \B^*$ be as in \eqref{eq:DualEig}. Then,
 \begin{equation}\label{454} \esssup_{\om \in \Om} \lVert \phi_\om^0\rVert_{\BV^*} <\infty. \end{equation}
\end{enumerate}
\end{lemma}
\begin{proof}
 Let $\Pi_\omega$ denote the projection on $\BV$ onto the subspace $\BV^0$ of functions of zero mean along the subspace spanned by $v_\om^0$. Furthermore, set
 \[
  \gamma(\om)=\inf \{\lVert f+g\rVert_{\BV}: f\in \BV^0, \ g\in span \{ v_\om^0\}, \ \lVert f\rVert_{\BV}=\lVert g\rVert_{\BV}=1\}.
 \]
  As in Lemma~1 in~\cite{DF} we have that   $\lVert \Pi_\om \rVert \le \frac{2}{\gamma (\om)}$. Take now arbitrary $f\in \BV^0$, $g\in span \{ v_\om^0\}$ such that $\lVert f\rVert_{\BV}=\lVert g\rVert_{\BV}=1$.
  It follows from~\ref{cond:unifNormBd} that
  \begin{equation}\label{haj}
    \lVert f+g\rVert_{\BV} \ge \frac{1}{K^n} \lVert \mathcal L_\om^{(n)}(f+g)\rVert_{\BV}  \ge \frac{1}{K^n}(\lVert \mathcal L_\om^{(n)}g\rVert_{\BV} -\lVert \mathcal L_\om^{(n)}f\rVert_{\BV}).
\end{equation}
Writing $g=\lambda v_\om^0$ with $\lvert \lambda \rvert=1/\lVert v_\om^0\rVert_{\BV}$, it follows from~\eqref{eq:boundedv} that
\[
 \lVert \mathcal L_\om^{(n)} g\rVert_{\BV} =\lvert \lambda \rvert \cdot \lVert v_{\sigma^n \om}^0\rVert_{\BV}=\frac{\lVert v_{\sigma^n \om}^0\rVert_{\BV}}{\lVert v_\om^0\rVert_{\BV}}\ge
 \frac{\lVert v_{\sigma^n \om}^0\rVert_1}{\lVert v_\om^0\rVert_{\BV}}\ge \frac{1}{\tilde K},
\]
where $\tilde K=\esssup_{\omega \in \Omega}\lVert v_\omega^0\rVert_{\BV}<\infty$.
By~\ref{cond:dec}  and~\eqref{haj},
\[
  \lVert f+g\rVert_{\BV} \ge \frac{1}{K^n}(1/\tilde K-K'e^{-\lambda n}).
\]
Then, we can choose $n$, independently of $\omega$, such that
\[
 \epsilon:=\frac{1}{K^n}(1/\tilde K-K'e^{-\lambda n})>0,
\]
which implies that $\gamma (\omega) \ge \epsilon$ and thus
\begin{equation}\label{din}
 \esssup_{\om \in \Om} \lVert \Pi_\om \rVert \le 2/\epsilon <\infty.
\end{equation}
Therefore,  for $\phi$ that belongs to annihilator of $v_\om^0$, using~\ref{cond:dec} and~\eqref{din} we have
\[
 \begin{split}
  \lVert \mathcal L_\om^{*, (n)}\phi \rVert_{\BV^*} =\sup_{\lVert f\rVert \le 1} \lvert \phi(\mathcal L_{\sigma^{-n} \om}^{(n)}f)\rvert &=
  \sup_{\lVert f\rVert \le 1} \lvert \phi(\mathcal L_{\sigma^{-n} \om}^{(n)}\Pi_{\sigma^{-n} \om}f)\rvert \\
  &\le K'e^{-\lambda n}\lVert \phi \rVert_{\BV^*} \cdot \lVert \Pi_{\sigma^{-n} \om}\rVert \\
  &\le \frac{2K'}{\epsilon}e^{-\lambda n}\lVert \phi \rVert_{\BV^*},
 \end{split}
\]
for every $n\ge 0$. We conclude that~\eqref{450} holds with $K''=2K'/\epsilon$.

Finally, \eqref{454} is follows directly from the straightforward fact that for \paeom, $\phi_\om^0(f)=\int f\,dm$.
\end{proof}
Next, we consider $\B^*$ with the norm topology, and associated Borel $\sigma-$algebra. Let
\[
 \mathcal N=\bigg{\{}\Phi \colon \Om \to \BV^*: \Phi \text{ is measurable}, \esssup_{\om \in \Om} \lVert \Phi_\om \rVert_{\BV^*}<\infty, \, \Phi_\om (v_\om^0)=0 \ \text{for \paeom} \bigg{\}}
\]
and
\[
 \mathcal N'=\bigg{\{}\Phi \colon \Om \to \BV^*: \Phi \text{ is measurable}, \esssup_{\om \in \Om} \lVert \Phi_\om \rVert_{\BV^*}<\infty \bigg{\}},
\]
where $\Phi_\om :=\Phi(\om)$. We note that $\mathcal N$ and $\mathcal N'$ are Banach spaces with respect to the norm
\[
 \lVert \Phi\rVert_\infty =\esssup_{\om \in \Om} \lVert \Phi_\om \rVert_{\BV^*}.
\]
We define $\mathcal G_1 \colon B_{\mathbb C}(0, 1) \times \mathcal N \to \mathcal N'$ by
\[
\mathcal G_1(\theta, \Phi)_\om=(\mcl_\om^\theta)^*(\Phi_{\sigma \om}+\phi_{\sigma \om}^0), \quad \om \in \Om.
\]
It follows readily from~\eqref{se2} and~\eqref{454} that $\mathcal G_1$ is well-defined.  Furthermore, we define $\mathcal G_2 \colon B_{\mathbb C}(0, 1) \times \mathcal N \to L^\infty(\Om)$ by
\[
\mathcal G_2 (\theta, \Phi)(\om)=(\Phi_{\sigma \om}+\phi_{\sigma \om}^0)(\mcl_\om^\theta v_\om^0), \quad \om \in \Om.
\]
Again, it follows from~\eqref{eq:boundedv}, \eqref{se2} and~\eqref{454} that $\mathcal G_2$ is well-defined.
\begin{lemma}\label{LEM}
$D_2 \mathcal G_1$ exists and is continuous on $B_{\mathbb C}(0, 1) \times \mathcal N$.
\end{lemma}
\begin{proof}
We first note that $\mathcal G_1$ is an affine map in the  variable $\Phi$ which implies that
\[
(D_2 \mathcal G_1(\theta, \Phi) \Psi)_\om=(\mcl_\om^\theta)^*\Psi_{\sigma \om}, \quad \text{for $\om \in \Om$ and $\Psi \in \mathcal N$.}
\]
Moreover, using~\eqref{DAC} we have
\[
\begin{split}
\lVert D_2\mathcal G_1(\theta_1, \Phi^1)-D_2\mathcal G_1(\theta_2, \Phi^2)\rVert &= \sup_{\lVert \Psi\rVert_\infty \le 1}\lVert D_2\mathcal G_1(\theta_1, \Phi^1)\Psi-D_2\mathcal G_1(\theta_2, \Phi^2) \Psi\rVert_\infty \\
&=\sup_{\lVert \Psi\rVert_\infty \le 1} \esssup_{\om \in \Om} \lVert (\mcl_\om^{\theta_1})^*\Psi_{\sigma \om}-(\mcl_\om^{\theta_2})^*\Psi_{\sigma \om}\rVert_{\BV^*} \\
&\le C\lvert \theta_1-\theta_2\rvert,
\end{split}
\]
for any $(\theta_1, \Phi^1), (\theta_2, \Phi^2) \in B_{\mathbb C}(0, 1) \times \mathcal N$. Hence, $D_2 \mathcal G_1$ is continuous on $B_{\mathbb C}(0, 1) \times \mathcal N$.
\end{proof}

\begin{lemma}\label{LEM1}
$D_1 \mathcal G_1$ exists and is continuous on a neighborhood of $(0, 0) \in \mathbb C \times \mathcal N$.
\end{lemma}

\begin{proof}
We claim that
\begin{equation}\label{545}
(D_1 \mathcal G_1 (\theta, \Phi) h)_\om (f)=(\Phi_{\sigma \om}+\phi_{\sigma \om}^0)(\mcl_\om (hg(\sigma^{-1} \om, \cdot)e^{\theta g(\sigma^{-1} \om, \cdot)}f)),
\end{equation}
for $f\in \BV$, $\om \in \Om$ and $h\in \mathbb C$. Denote the operator on the right hand side of~\eqref{545} by $L(\theta, \Phi)$.  We note that
\begin{align*}
 & (\mathcal G_1(\theta+h, \Phi)_\om-\mathcal G_1(\theta, \Phi)_\om-hL(\theta, \Phi)_\om)(f) \displaybreak[0] \\
 &= (\Phi_{\sigma \om}+\phi_{\sigma \om}^0)(\mathcal L_\om ((e^{(\theta+h)g(\sigma^{-1} \om, \cdot)}-e^{\theta g(\sigma^{-1} \om, \cdot)}-hg(\sigma^{-1} \om, \cdot)e^{\theta g(\sigma^{-1} \om, \cdot)})f)).
 \end{align*}
Therefore, it follows from~\ref{cond:unifNormBd} that
\begin{align*}
& \lVert \mathcal G_1(\theta+h, \Phi)-\mathcal G_1(\theta, \Phi)-h L(\theta, \Phi)\rVert_\infty \displaybreak[0] \\
&=\esssup_{\om \in \Om}\sup_{\lVert f\rVert_{\BV} \le 1}\lvert (\Phi_{\sigma \om}+\phi_{\sigma \om}^0)(\mathcal L_\om ((e^{(\theta+h)g(\sigma^{-1} \om, \cdot)}-e^{\theta g(\sigma^{-1} \om, \cdot)}-hg(\sigma^{-1} \om, \cdot)e^{\theta g(\sigma^{-1} \om, \cdot)})f)) \rvert  \displaybreak[0] \\
&\le K(\lVert \Phi\rVert_\infty +\lVert \phi^0\rVert_\infty)\esssup_{\om \in \Om}\sup_{\lVert f\rVert_{\BV} \le 1}\lVert (e^{(\theta+h)g(\sigma^{-1} \om, \cdot)}-e^{\theta g(\sigma^{-1} \om, \cdot)}-hg(\sigma^{-1} \om, \cdot)e^{\theta g(\sigma^{-1} \om, \cdot)})f\rVert_{\BV}.
\end{align*}
By~\eqref{mh2} and~\eqref{t}, we conclude that
\[
\lim_{h\to 0} \frac 1 h \lVert \mathcal G_1(\theta+h, \Phi)-\mathcal G_1(\theta, \Phi)-h L(\theta, \Phi)\rVert_\infty=0,
\]
and thus~\eqref{545} holds. Moreover,
\begin{align*}
& (D_1 \mathcal G_1 (\theta_1, \Phi^1) h)_\om (f)-(D_1 \mathcal G_1 (\theta_2, \Phi^2) h)_\om (f)\displaybreak[0] \\
&=(\Phi_{\sigma \om}^1-\Phi_{\sigma \om}^2)(\mcl_\om (hg(\sigma^{-1} \om, \cdot)e^{\theta_1 g(\sigma^{-1} \om, \cdot)}f)) \displaybreak[0] \\
&\phantom{=}+(\Phi_{\sigma \om}^2+\phi_{\sigma \om}^0)(\mcl_\om (hg(\sigma^{-1} \om, \cdot)(e^{\theta_1 g(\sigma^{-1} \om, \cdot)}-e^{\theta_2 g(\sigma^{-1} \om, \cdot)})f)),
\end{align*}
which in view of~\ref{cond:unifNormBd}, \eqref{obs},  \eqref{X2} and~\eqref{X1} easily implies that
 $D_1 \mathcal G_1$ is continuous.
\end{proof}

\begin{lemma}\label{LEM2}
$D_2 \mathcal G_2$ exists and is continuous on a neighborhood of $(0, 0) \in B_{\mathbb C}(0,1) \times \mathcal N$.
\end{lemma}

\begin{proof}
We note that $\mathcal G_2$ is affine map in the variable $\Phi$ and hence
\[
(D_2 \mathcal G_2(\theta, \Phi)\Psi)(\om)=\Psi_{\sigma \om}(\mcl_\om^\theta v_\om^0), \quad \om \in \Om.
\]
It follows from~\eqref{DAC} that
\[
\begin{split}
\lVert D_2 \mathcal G_2(\theta_1, \Phi^1)-D_2 \mathcal G_2(\theta_2, \Phi^2)\rVert  &=\sup_{\lVert \Psi\rVert_\infty \le 1}\lVert D_2 \mathcal G_2(\theta_1, \Phi^1)\Psi-D_2 \mathcal G_2(\theta_2, \Phi^2) \Psi \rVert_{L^\infty} \\
&=\sup_{\lVert \Psi\rVert_\infty \le 1} \esssup_{\om \in \Om}\lvert \Psi_{\sigma \om}(\mcl_\om^{\theta_1}v_\om^0-\mcl_\om^{\theta_2} v_\om^0)\rvert \\
&\le C\lvert \theta_1-\theta_2 \rvert \cdot \esssup_{\om \in \Om} \lVert v_\om^0\rVert_{\BV},
\end{split}
\]
and thus (in a view of~\eqref{eq:boundedv}) we conclude that  $D_2 \mathcal G_2$ is continuous.
\end{proof}

\begin{lemma}\label{LEM3}
$D_1 \mathcal G_2$ exists and is continuous on a neighborhood of $(0, 0) \in \mathbb C \times \mathcal N$.
\end{lemma}

\begin{proof}
We claim that
\begin{equation}\label{8301}
(D_1\mathcal G_2(\theta, \Phi)h)(\om) =(\Phi_{\sigma \om}+\phi_{\sigma \om}^0)(\mcl_\om(g(\sigma^{-1}\om, \cdot)e^{\theta g(\sigma^{-1} \om, \cdot)}v_\om^0)), \quad h\in \mathbb C, \om \in \Om.
\end{equation}
Let us denote the operator on the right hand side of~\eqref{8301} by $R(\theta, \Phi)$. We have that
\begin{align*}
 & (\mathcal G_2(\theta+h, \Phi)-\mathcal G_2(\theta, \Phi)-hR(\theta, \Phi))(\om) \displaybreak[0] \\
 &= (\Phi_{\sigma \om}+\phi_{\sigma \om}^0)(\mathcal L_\om ((e^{(\theta+h)g(\sigma^{-1} \om, \cdot)}-e^{\theta g(\sigma^{-1} \om, \cdot)}-hg(\sigma^{-1} \om, \cdot)e^{\theta g(\sigma^{-1} \om, \cdot)})v_\om^0)).
 \end{align*}
Therefore, it follows from~\ref{cond:unifNormBd} that
\begin{align*}
& \lVert \mathcal G_2(\theta+h, \Phi)-\mathcal G_2(\theta, \Phi)-h R(\theta, \Phi)\rVert_{L^\infty} \displaybreak[0] \\
&=\esssup_{\om \in \Om}\lvert (\Phi_{\sigma \om}+\phi_{\sigma \om}^0)(\mathcal L_\om ((e^{(\theta+h)g(\sigma^{-1} \om, \cdot)}-e^{\theta g(\sigma^{-1} \om, \cdot)}-hg(\sigma^{-1} \om, \cdot)e^{\theta g(\sigma^{-1} \om, \cdot)})v_\om^0)) \rvert  \displaybreak[0] \\
&\le K(\lVert \Phi\rVert_\infty +\lVert \phi^0\rVert_\infty)\esssup_{\om \in \Om}\lVert (e^{(\theta+h)g(\sigma^{-1} \om, \cdot)}-e^{\theta g(\sigma^{-1} \om, \cdot)}-hg(\sigma^{-1} \om, \cdot)e^{\theta g(\sigma^{-1} \om, \cdot)})v_\om^0\rVert_{\BV}.
\end{align*}
By~\eqref{eq:boundedv}, \eqref{mh2} and~\eqref{t}, we conclude that
\[
\lim_{h\to 0} \frac 1 h \lVert \mathcal G_2(\theta+h, \Phi)-\mathcal G_2(\theta, \Phi)-h R(\theta, \Phi)\rVert_{L^\infty}=0.
\]
Thus, \eqref{8301} holds. Moreover,
\begin{align*}
& (D_1 \mathcal G_2 (\theta_1, \Phi^1) h)(\om) -(D_1 \mathcal G_2 (\theta_2, \Phi^2) h)(\om) \displaybreak[0] \\
&=(\Phi_{\sigma \om}^1-\Phi_{\sigma \om}^2)(\mcl_\om (hg(\sigma^{-1} \om, \cdot)e^{\theta_1 g(\sigma^{-1} \om, \cdot)}v_\om^0)) \displaybreak[0] \\
&\phantom{=}+(\Phi_{\sigma \om}^2+\phi_{\sigma \om}^0)(\mcl_\om (hg(\sigma^{-1} \om, \cdot)(e^{\theta_1 g(\sigma^{-1} \om, \cdot)}-e^{\theta_2 g(\sigma^{-1} \om, \cdot)})v_\om^0)),
\end{align*}
which in view of~\ref{cond:unifNormBd}, \eqref{obs},  \eqref{X2} and~\eqref{X1} easily implies that $D_1 \mathcal G_2 (\theta_1, \Phi^1) \to D_1 \mathcal G_2 (\theta_2, \Phi^2)$ when $(\theta_1, \Phi^1) \to (\theta_2, \Phi^2)$. Hence, $D_1 \mathcal G_2$
is continuous.
\end{proof}
Let
\begin{equation}\label{fdc2}
\mathcal G(\theta, \Phi)_\om =\frac{(\mcl_\om^\theta)^*(\Phi_{\sigma \om}+\phi_{\sigma \om}^0)}{(\Phi_{\sigma \om}+\phi_{\sigma \om}^0)(\mcl_\om^\theta v_\om^0)}-\Phi_\om -\phi_\om^0.
\end{equation}
\begin{proposition}\label{fdc}
The map $\mathcal G$ is of class $C^1$ on a neighborhood of $(0, 0) \in \mathbb C \times \mathcal N$. Furthermore,
\begin{equation}\label{2212017}
((D_2\mathcal G(\theta, \Phi))\Psi)_\om=\frac{(\mcl_\om^\theta)^*\Psi_{\sigma \om}}{\mathcal G_2(\theta, \Phi)(\om)}-\frac{\Psi_{\sigma \om}(\mcl_\om^\theta v_\om^0)}{[\mathcal G_2(\theta, \Phi)(\om)]^2}\mathcal G_1(\theta, \Phi)_\om -\Psi_\om, \quad \om \in \Om, \Psi \in \mathcal N.
\end{equation}
\end{proposition}

\begin{proof}
The desired conclusion follows directly from  Lemmas~\ref{LEM}, \ref{LEM1}, \ref{LEM2} and~\ref{LEM3} after we note that $\mathcal G_2(0, 0)(\om)=1$ for $\om \in \Om$.
\end{proof}

\begin{lemma}\label{fdc1}
$D_2 \mathcal G(0, 0)$ is invertible.
\end{lemma}

\begin{proof}
By~\eqref{2212017},
\[
(D_2 \mathcal G(0, 0) \Psi)_\om=\mcl_\om^*\Psi_{\sigma \om}- \Psi_\om, \quad \text{for $\om \in \Om$ and  $\Psi \in \mathcal N$.}
\]
Now one can proceed as in the proof of Lemma~\ref{thm:IFT} to show that~\eqref{450} implies the desired conclusion.
\end{proof}
It follows from Proposition~\ref{fdc}, Lemma~\ref{fdc1} and the implicit function theorem that there exists a neighborhood $U$ of $0\in \mathbb C$ and a  smooth function $\mathcal F \colon U \to \mathcal N$ such that $\mathcal F(0)=0$ and
\begin{equation}\label{fdc3}
\mathcal G(\theta, \mathcal F(\theta))=0, \quad \text{for $\theta \in U$.}
\end{equation}
Finally, set
\[
\Psi(\theta)_\om=\frac{\mathcal F(\theta)_\om+\phi_\om^0}{(\mathcal F(\theta)_\om+\phi_\om^0)(v_\om^\theta)}, \quad \text{for $\om \in \Om$ and $\theta \in U$.}
\]
Using the differentiability of $\theta \mapsto v^\theta$, we observe that there exists a neighborhood $U' \subset U$ of $0\in \mathbb C$ such that $\Psi(\theta)$ is well-defined and differentiable for $\theta \in U'$. Furthermore, we note that
$\Psi(\theta)_\om (v_\om^\theta)=1$. Finally, it follows from~\eqref{fdc2} and~\eqref{fdc3} that
\[
(\mcl_\om^\theta)^*\Psi(\theta)_{\sigma \om}=C_\om^\theta \Psi(\theta)_\om,
\]
for some scalar $C_\om^\theta$. The arguments in Subsection~\ref{sec:choiceOsBases} imply that $\phi_\om^\theta=\Psi(\theta)_\om$. Therefore,  we have established the differentiability of
$\theta \to \phi^\theta$.

\section*{Acknowledgements}
We would like to thank the referee for carefully reading our manuscript and for useful comments that helped us improve the quality of the paper. 
We thank Yuri Kifer for raising the possibility of proving a quenched random LCLT, and for providing valuable feedback and references for this work.
The research of DD was supported by the Australian Research Council Discovery
Project DP150100017 and in part by the Croatian Science Foundation under the
project IP-2014-09-228. The research of GF was  supported by the Australian Research Council Discovery
Project DP150100017 and a Future Fellowship.
CGT was supported by ARC DE160100147.
SV was supported by the  project  APEX ``Syst\`{e}mes dynamiques: Probabilit\'{e}s et Approximation Diophantienne PAD'' funded by the R\'{e}gion PACA, by the Labex Archim\'ede (AMU University), by the Leverhulme Trust for support thorough the Network Grant IN-2014-021 and by the project Physeco, MATH-AMSud.
CGT thanks Jacopo de Simoi and Carlangelo Liverani for their hospitality in Rome and for conversations related to this topic.
Parts of this work were completed when (some or all) the authors met at AIM (San Jose), CPT \& CIRM
(Marseille), SUSTC (Shenzhen), the University of New South Wales and the University
of Queensland. We are thankful to all of these institutions for their support and hospitality.

\bibliographystyle{abbrv}

\end{document}